\documentclass[reqno]{amsart}

\usepackage[english]{babel}
\usepackage{amsmath, amssymb, latexsym, epsfig, epic, rotating, fancyhdr, amsthm, pifont, enumitem}

\usepackage{amscd,mathrsfs}
\usepackage{epstopdf}
\usepackage[ansinew]{inputenc}
\usepackage{graphicx}
\usepackage[usenames]{color}

\newtheoremstyle{tobthm}{3pt}{3pt}{\itshape}{0pt}{\bfseries}{.}{0.5eM}{}
\theoremstyle{tobthm}
\newtheorem*{thmmain}{Theorem}

\newtheorem{definition}{Definition}[section]
\newtheorem{thm}[definition]{Theorem}
\newtheorem{theorem}[definition]{Theorem}

\newtheorem{lem}[definition]{Lemma}
\newtheorem{lemma}[definition]{Lemma}
\newtheorem{cor}[definition]{Corollary}
\newtheorem{prop}[definition]{Proposition}
\newtheorem{proposition}[definition]{Proposition}

\newtheoremstyle{tobrem}{3pt}{3pt}{\normalfont}{0pt}{\bfseries}{.}{0.5em}{}
\theoremstyle{tobrem} 

\newtheorem{rem}[definition]{Remark}

\newcommand{\N}{\mathbb{N}}

\newcommand{\R}{\mathbb{R}}
\newcommand{\Z}{\mathbb{Z}}

\newcommand{\hlim}{\lim_{\mathcal{H}}}
\newcommand{\piab}{\pi_{\textrm{ab}}}
\newcommand{\D}{\mathbb{D}}

\newcommand{\transv}{\mathrel{\text{\tpitchfork}}}
\makeatletter
\newcommand{\tpitchfork}{%
  \vbox{
    \baselineskip\z@skip
    \lineskip-.52ex
    \lineskiplimit\maxdimen
    \m@th
    \ialign{##\crcr\hidewidth\smash{$-$}\hidewidth\crcr$\pitchfork$\crcr}
  }%
}
\makeatother

\numberwithin{equation}{section}
\numberwithin{figure}{section}

\title{Generic rotation sets in hyperbolic surfaces}

\author{J.~Alonso, J.~Brum, A.~Passeggi}

\begin{document}

\maketitle

\begin{abstract}

We show that for generic homeomorphisms homotopic to the identity in a closed and oriented surface of genus $g>1$, the rotation set is given by a union of at most $2^{5g-3}$ convex sets.
Examples showing the sharpness for this asymptotic order are provided.

\end{abstract}

\section{Introduction}

Rotation sets are natural invariants in order to classify the dynamics of homeomorphisms on manifolds which are isotopic to the identity. In dimension one this invariant leads to the well known
classification by Poincar\'{e} for the dynamics of orientation preserving homeomorphisms of the circle. In dimension two, the rotation theory is largely developed in the annular and
the toral cases. The crucial fact is that the rotation set in these contexts boils down to simple representations: an interval in the annular case and a planar convex set in the
toral case. From these basic geometries it is possible to get a fine classification of the associated topological dynamics. As sample of results see
\cite{F,F1,patrice,MiZi1,MiZi2,LiMa,F2,lectal}.

In higher genus the theory is far from being developed as it is not known whether the rotation set (which lives in the first real homology group of the surface, i.e. $\R^{2g}$) presents a \emph{simple} geometrical structure or not. At first sight one obtains, as in
the toral case, that this invariant can be represented by a continuum in $\R^{2g}$, but there
is no further information about its shape that could mimic the convexity property of the toral case.
So far, the rotation theory on higher genus surfaces stands on a series of results, in which some rotation vectors
are assumed to be in some positions in order to infer dynamical implications \cite{F3,Policot,Lessa} and recently \cite{zanatajacoia} which obtains convexity of the rotation set under rather strong conditions. 

\smallskip

The motivation for this article is then the crucial question:

\smallskip
Is there anything like a
\emph{simple structure} for the rotation set on higher genus surfaces?
\smallskip

Although we have not found a general answer yet, we obtain the following:
\begin{thmmain}
Let $\Sigma$ be a closed surface with genus $g>1$. Then, for a $C^0$ open and dense set of homeomorphisms in the homotopy class of the identity, the rotation set is given by a union of at
most $2^{5g-3}$ convex sets all containing zero.

If, moreover, this rotation set has non-empty interior,  then it is a full-dimensional rational polyhedron.
\end{thmmain}

By a rational polyhedron we mean a convex set with finitely many extremal points, all of them with rational coordinates.
The actual statement we obtain has more detail, giving important information about the location of these convex sets.
We will also provide a list of examples that illustrate different meaningful situations arising in the context of our result. Among these examples, we show that the number of convex sets needed to describe the rotation set of such generic homeomorphisms may grow exponentially with the genus $g$ of the surface, thus our bound of $2^{5g-3}$ is sharp in an asymptotic sense.

Next we make a more formal presentation of the result, and spell out a guide for the organization of the article.

\subsection{Definitions and Result}

Denote by $\Sigma_g$ a closed oriented surface of genus $g> 1$. The set of homeomorphisms on $\Sigma_g$ that are homotopic to the identity is denoted by $\textrm{Homeo}_0(\Sigma_g)$. We identify the first homological group $H_1(\Sigma_g;\R)$ with $\R^{2g}$. 

Given a map $f\in\textrm{Homeo}_0(\Sigma_g)$, we have an isotopy $(f_t)_{t\in[0,1]}$ from the identity to $f=f_1$. Fixing a set of uniformly bounded paths $\{\gamma_x\}_{x\in\Sigma_g}$ from
a base point $x_0\in\Sigma_g$ to $x\in\Sigma_g$, the {\em homological rotation set} is usually defined by
\small
$$\rho(f)=\left\{\lim_i \frac{[\gamma_{f^{n_i}(x_i)}^{-1}\cdot (f_t(x_i))_{t\in[0,n_i]} \cdot{\gamma_{x_i}}]}{n_i}:\ x_i\in\Sigma_g,\ n_i\nearrow\infty\right\}\subset \R^{2g},$$

\normalsize

where $[\gamma]$ is the homological class of a closed curve $\gamma$ and $(f_t(x))_{t\in[0,n]}$ is the concatenation of curves
\small
$$(f_t(x))_{t\in[0,1]}\cdot (f_t(f(x)))_{t\in[0,1]}\ldots \cdot(f_t(f^{n-1}(x))_{t\in[0,1]}.$$
\normalsize

It is worth to mention that this definition does not depend on the choice of the family $\{\gamma_x\}_{x\in\Sigma_g}$ nor the choice of the isotopy $(f_t)_{t\in[0,1]}$ (which is unique up to homotopy, see \cite{Hamstrom}).
\medskip

When $f$ is restricted to some invariant set $K\subset \Sigma_g$, the rotation set is defined by

\small
$$\rho_K(f)=\left\{\lim_i \frac{[\gamma_{f^{n_i}(x_i)}^{-1}\cdot (f_t(x_i))_{t\in[0,n_i]} \cdot{\gamma_{x_i}}]}{n_i}:\ x_i\in K,\ n_i\nearrow\infty\right\}\subset \R^{2g},$$

\normalsize

We now turn to describe the generic set of homeomorphisms for which our result holds. The key to our study is the family $\mathcal{A}_0(\Sigma_g)\subset\textrm{Homeo}_0(\Sigma_g)$ of {\em fitted
Axiom A diffeomorphisms}, which is given by the maps $f\in \textrm{Homeo}_0(\Sigma_g)$ such that:

\begin{enumerate}

\item $\Omega(f)$ is a zero-dimensional hyperbolic set.

\item Whenever $z\in W^s(x,f)]\cap W^u(y,f)$ for $x,y\in\Omega(f)$, then we have
$z\in W^s(x,f)\transv W^u(y,f)$.

\end{enumerate}

As proved in \cite{SS} the set $\mathcal{A}_0(\Sigma_g)$ is dense in $\textrm{Homeo}_0(\Sigma_g)$. Moreover, it satisfies {\em semi-stability}
\cite{Zni,Franks}: given $f\in\mathcal{A}_0(\Sigma_g)$ there exists a $C^0$-neighbourhood $U(f)$ of $f$ in $\textrm{Homeo}_0(\Sigma_g)$
such that, for any $g\in U(f)$ there exists a continuous surjection $h:\Sigma_g\to \Sigma_g$ in the homotopy class of the identity, verifying:
$$h\circ g=f\circ h$$
This last implies that for $g\in U(f)$ we have $\rho(g)=\rho(f)$.
Our $C^0$ open and dense set $\mathcal{U}$ is then given by the union of all those
$U(f)$ where $f$ ranges in $\mathcal{A}_0(\Sigma_g)$. Under these considerations, in order
to obtain the first part of our result it is enough to show the assertion for $f\in\mathcal{A}_0(\Sigma_g)$, which is given by the following statement.

\begin{theorem}\label{mainthmintroaxa} Let $f\in\mathcal{A}_0(\Sigma)$ for $\Sigma$ an orientable closed surface of genus $g>1$. Then $\rho(f)\subset H_1(\Sigma;\R)$ is a union of at most $2^{5g-3}$ convex sets containing $0$.
%

\end{theorem}

This result is in Section \ref{s.prin}, Theorem \ref{mainthm}. The version presented there
contains further information, locating these convex sets in certain homology subspaces and showing  they have finitely many vertices, all of them rational. We avoid these details here for the sake of simplicity, and only give an idea of where these homology subspaces come from. They are
 related to a decomposition of $\Sigma$ into essential subsurfaces, that we call \emph{filled Conley
surfaces} for $f\in \mathcal{A}_0(\Sigma)$, and we obtain from Conley's fundamental theorem \cite{conleyfundthm} after some modifications. We then build certain collections of these subsurfaces, which will be given in terms of the {\em heteroclinical relations} among the {\em basic pieces} of the map $f$ (this makes strong use of the properties of maps in $\mathcal{A}_0(\Sigma)$). Finally, it is shown that each convex set in the decomposition of $\rho(f)$ lives in the sum of the homology spaces of the subsurfaces in one of these collections.

The bound $2^{5g-3}$ comes from this construction, by showing that each of these homology subspaces contains at most $4$ convex sets of the decomposition. (The essential decomposition will allow annuli, but no repetition of their homotopy types, thus has a maximum of $5g-5$ subsurfaces). There are some coarse considerations involved in this bound, as we explain right after the statement of Theorem \ref{mainthm}. However, in section \label{s.ejemploexp} we provide a family of examples that realize Theorem \ref{mainthmintroaxa} with $2^{[g/2]}$ convex sets, showing that a sharp bound must be exponential.



In Section \ref{s.ej} we present a list of examples showing
remarkable situations about the positions these convex sets may take. The statement about rotation sets with non-empty interior will also be proved for $f\in \mathcal{A}_0(\Sigma)$ in Section \ref{s.prinint}, and thus deduced for the same $C^0$ open dense set $\mathcal{U}$.

\subsection{Structure of the article}\label{ss.structure}

Section \ref{s.pre} covers the pre-requisites: alternative definitions of the homology rotation set for higher genus closed oriented surfaces, and an overview of the theory of Axiom A dynamics. 
In Section
\ref{s.rotsetbp} there is a first approximation to Theorem \ref{mainthmintroaxa}: supported on
previous results \cite{Zi,Pass1} we first observe that the rotation sets of a fitted axiom A map $f$ restricted to the \emph{basic pieces} of $f$ are given by rational polyhedra. This already established result
only takes care of the rotation vectors of non-wandering points, but that is not sufficient, since the wandering dynamics also play a role in the rotation set. Thus in the second part of Section \ref{s.rotsetbp}
we study the wandering dynamics, and conclude that $\rho(f)$ is given
by a union of convex sets indexed on maximal chains of the directed graph given by
heteroclinical relations in-between basic pieces of the map $f$. This gives $\rho(f)$ as union of finitely many rational polyhedra, but notice the important fact
that at this point we are still far from Theorem \ref{mainthmintroaxa}, as generically the number of such chains is arbitrarily large. The following sections represent the core of the article
and are devoted to prove that the convex sets just obtained can be encapsulated into at most $2^{5g-3}$ convex sets.

\smallskip

To achieve this objective we start by associating topological information to each basic piece, though the concepts of
\emph{Quasi-invariant surfaces} and \emph{Conley surfaces} that are introduced in Section
\ref{s.conley} and will be the main tools for this work. Basically, we classify basic pieces into three cases:
\emph{trivial}, \emph{annular} and \emph{curved}, according to the topology of their associated Conley surfaces. From the definition one easily obtains
that the number of \emph{curved basic pieces} is bounded by the topology of $\Sigma$,
thus the problem we had is now reduced to control those chains where \emph{trivial}
and \emph{annular} basic pieces participate. Before that, in Section \ref{s.agujas}, we apply the tools of Section \ref{s.conley} to show that the rotation set is star shaped about the origin. In Section \ref{s.triviales} we prove that \emph{trivial} basic pieces can be completely neglected in order to compute $\rho(f)$.

\smallskip

Section \ref{s.annular} is devoted to deal with chains
containing \emph{annular} basic pieces. In this case we can not neglect them, rather we
have to develop a method to encapsulate the rotation sets of different chains containing \emph{annular} pieces whose Conley annuli have the same homotopy types. This is achieved by a serie of techincal results that culminate in Theorem \ref{conexion cruzada}. Then Section \ref{s.prin}  gives the extended version of Theorem \ref{mainthmintroaxa}, and completes its proof.
Section \ref{s.prinint} covers the case of a non-empty interior rotation set, which gives the second part of our main result,
and  Section \ref{s.ej} contains a list of substantial examples.

\medskip

\textbf{Acknowledgments:} \emph{We are grateful to Pierre-Antoine Guih\'eneuf for his inputs on the
example 10.4. We also thank Mart\'in Sambarino and Rafel Potrie for useful discussions concerning hyperbolic dynamics.}

\section{Preliminaries}\label{s.pre}

In this section we introduce the main objects of this article, as well as the foundational results we will need to use. We will focus in two objects: \emph{homological rotation set} and \emph{fitted axiom A diffeomorfisms}. A third important tool will be the \emph{Fundamental Theorem of Dynamical Sistems} due to Conley, which will be deferred until section \ref{s.conley}, as we need to develop a series of lemmas concerning this seminal result.

The reader familiar with these subjects can skip the section.

\subsection{Homological rotation set}

We shall give four different definitions for homological rotation sets which coincide, each of them useful in different contexts. Let $f\in\textrm{Homeo}_0(\Sigma_g)$, and $(f_t)_{t\in[0,1]}$ some isotopy from the identity to $f$. First recall the classical definition from the introduction: 
\begin{equation}\label{def1} \rho(f)=\left\{\lim_i \frac{[\gamma_{f^{n_i}(x_i)}^{-1}\cdot (f_t(x_i))_{t\in[0,n_i]} \cdot{\gamma_{x_i}}]}{n_i}:\ x_i\in\Sigma_g,\ n_i\nearrow\infty\right\}\subset \R^{2g} \end{equation}
where $\{\gamma_x\}_{x\in\Sigma_g}$ is a bounded family of paths from a base point $x_0\in\Sigma_g$ to $x\in\Sigma_g$, and $(f_t)_{t\in[0,n]}$ is the natural self-concatenation of the isotopy $(f_t)_{t\in[0,1]}$. We remark that the sequences $x_i$ and $n_i$ in the definition range over those for which the limit exists and this convention is used throughout the article.

Next we assume that the base point $x_0$ is fixed through the isotopy, i.e. $f_t(x_0)=x_0$ for all $t\in[0,1]$. Then we can rewrite the homological rotation set as

 \begin{equation}\label{def2} \rho(f)=\left\{\lim_i \frac{[\gamma_{f^{n_i}(x_i)}^{-1}\cdot f^{n_i}(\gamma_{x_i})]}{n_i}:\ x_i\in\Sigma_g,\ n_i\nearrow\infty\right\}\subset \R^{2g} \end{equation}

For $g>1$, it is a consequence of Lefschetz index theory that $f$ has an {\em irrotational fixed point}, i.e. the isotopy $f_t$ can be chosen to fix this point for all $t$. Thus we can always find a base point $x_0$ for this second definition.

We shall consider the {\em abelian cover} $\pi_{ab}:\Sigma_g^{ab}\to\Sigma_g$, which is the covering space that corresponds to the commutator subgroup of $\pi_1(\Sigma_g)$. Thus the group of deck transformations of $\Sigma_g^{ab}$ is $H_1(\Sigma_g;\Z)$. We will identify $\Sigma_g^{ab}$ with a subsurface of $H_1(\Sigma_g,\R)=\R^{2g}$, invariant under the group of integer translations, so that the deck transformations of $\Sigma_g^{ab}$ are exactly (the restrictions to $\Sigma_g^{ab}$ of) these translations.

This can be done by the following embedding: Let $\omega_1,\ldots,\omega_{2g}$ be {\em generic} $1$-forms on $\Sigma_g$ that represent a basis of the cohomology group $H^1(\Sigma_g;\Z)$. Let $x_0$ be a base point in $\Sigma_g^{ab}$ and $\{\beta_x\}_{x\in\Sigma_g^{ab}}$ a family of curves in $\Sigma_g^{ab}$ from $x_0$ to $x\in\Sigma_g^{ab}$. Then the embedding can be written as

$$\Phi:\Sigma_g^{ab}\to \R^{2g} \ :\ \Phi(x)=\left(\int_{\pi_{ab}(\beta_x)}\omega_1,\ldots, \int_{\pi_{ab}(\beta_x)}\omega_{2g}\right) $$
which does not depend on the choice of curves $\beta_x$.

Let $\hat{f}_t$ be the lift of the homotopy $f_t$ to $\Sigma^{ab}_g$ with $\hat{f}_0=Id$. We obtain $\hat{f}=\hat{f}_1$, which is a lift of $f$. The homological rotation set of $f$ can then be expressed as follows,

\begin{equation}\label{def3} \rho(f)=\left\{\lim_i \frac{\hat{f}^{n_i}(x_i)-x_i }{n_i}:\ x_i\in\Sigma_g^{ab},\ n_i\nearrow\infty\right\}\subset \R^{2g} \end{equation}

If $D\subset \Sigma_g^{ab}$ is a fundamental domain, we can give yet another expression as above by taking limits of the form $\hat{f}^{n_i}(x_i)/n_i$ where $x_i\in D$. The last equivalent definition involves Hausdorff limits of compact subsets of $\R^{2g}$. The metric on $\R^{2g}$ is the standard Euclidean metric. Let $D\subset \Sigma_g^{ab}$ be a compact fundamental domain. Then we have that

\begin{equation}\label{def4} \rho(f)=\bigcup \left\{\mathcal{L} : \mathcal{L}=\hlim\frac{\hat{f}^{n_i}(D)}{n_i},\ n_i\nearrow\infty \right\} \end{equation}

In the case of the torus ($g=1$), we have that $\Sigma_1^{ab} = \R^2$ and that the Hausdorff limit $\hlim\frac{\hat{f}^{n}(D)}{n}$ exists \cite{MiZi1}. It is still open whether this is also true for $g>1$.

The equivalence between all definitions of $\rho(f)$ given above is straightforward:

\begin{prop}\label{p.defrotsetprop}

Definitions \ref{def1}, \ref{def2}, \ref{def3} and \ref{def4} are equivalent.

\end{prop}

\begin{cor}\label{c.proprotset1}

For any $f\in\textrm{Homeo}_0(\Sigma_g)$, $\rho(f)$ is a compact and connected set containing $0$.

\end{cor}

When restricted to an $f$-invariant set $K\subset\Sigma_g$, the rotation set $\rho_K(f)$ can be defined equivalently by any of the following options:

\begin{enumerate}

\item $\rho_K(f)=\left\{\lim_i \frac{[\gamma_{f^{n_i}(x_i)}^{-1}\cdot (f_t(x_i))_{t\in[0,n_i]} \cdot{\gamma_{x_i}}]}{n_i}:\ x_i\in K,\ n_i\nearrow\infty\right\}\subset \R^{2g},$

\item $\rho_K(f)=\left\{\lim_i \frac{[\gamma_{f^{n_i}(x_i)}^{-1}\cdot f^{n_i}(\gamma_{x_i})]}{n_i}:\ x_i\in K,\ n_i\nearrow\infty\right\}\subset \R^{2g},$

\item $\rho_K(f)=\left\{\lim_i \frac{\hat{f}^{n_i}(x_i)-x_i }{n_i}:\ x_i\in\Sigma_g^{ab}\cap\piab^{-1}(K),\ n_i\nearrow\infty\right\}\subset \R^{2g},$

\item $\rho_K(f)=\bigcup\left\{ \mathcal{L}:\mathcal{L}=\hlim\frac{\hat{f}^{n_i}(\hat{K})}{n_i},\ n_i\nearrow\infty\right\}$, where $\hat{K} = D\cap \piab^{-1}(K)$ and $D$ is a compact fundamental domain of $\Sigma_g^{ab}$.

\end{enumerate}

That is exactly analogous to the discussion for the global rotation set.

\subsection{Fitted axiom A diffeomorphisms}

In this section we review the definition of fitted axiom A diffeomorphisms, and state the results about them that are relevant for this work. Throughout this article, we define diffeomorphisms as bijective maps $f:M\to M$ so that both $f$ and $f^{-1}$ are $C^1$ maps with invertible differentials at every point of $M$.

Recall that an invariant compact set $\Lambda$ of a diffeomorphism $f:M\to M$ defined in a Riemannian manifold,
is \emph{hyperbolic} if there exists a $df$-invariant splitting $TM=E^s\oplus E^u$ defined on $\Lambda$, a positive constant $C$ and a positive constant $\lambda\in(0,1)$ so that:
\begin{itemize}

\item[(i)] $d_xf(E^s_x)=E_{f(x)}^s$ and $\|d_x f^n|_{E_s^x}\|<C\lambda^n$,

\item[(ii)] $d_xf(E^u_x)=E_{f(x)}^u$ and $\|d_x f^{-n}|_{E_u^x}\|<C\lambda^n$.

\end{itemize}
\medskip
For this kind of invariant set we have the celebrated Stable manifold Theorem. Let us state the definitions involved in that result.
For a homeomorphism $f:M\to M$
\begin{itemize}

\item the $\varepsilon$-\emph{stable} set of a point $x\in M$ is defined by
\small
$$W^s_{\varepsilon}(x,f)=\{y\in M:\ d(f^n{x},f^n{y})<\varepsilon\ n\in\N\},$$
\normalsize

\item the $\varepsilon$-\emph{unstable} set of a point $x\in M$ is defined by
\small
$$W^u_{\varepsilon}(x,f)=\{y\in M:\ d(f^{-n}{x},f^{-n}{y})<\varepsilon\ n\in\N\},$$
\normalsize

\item the \emph{stable} set of a point $x\in M$ is defined by
\small
$$W^s(x,f)=\{y\in M:\ d(f^n(x),f^n(y))\to_{n\to+\infty}0\},$$
\normalsize

\item the \emph{unstable} set of a point $x\in M$ is defined by
\small
$$W^u(x,f)=\{y\in M:\ d(f^n(x),f^n(y))\to_{n\to-\infty}0\}.$$
\normalsize

\end{itemize}

\begin{thm}[Stable manifold Theorem, \cite{Katok}, Section 6]\label{t.stamani}

Let $\Lambda$ be \textcolor{red}{a} hyperbolic compact invariant set of a diffeomorphism $f:M\to M$. Then
there exists $\varepsilon>0$ so that for any $x\in\Lambda$
\begin{enumerate}

\item $W^s_{\varepsilon}(x,f)$ is a $C^1$ embedded manifold tangent to $E^s_x$, which varies $C^1$
continuously with the point $x\in\Lambda$.

\item $W^s_{\varepsilon}(x,f)\subset W^s(x,f)$.

\item $W^s(x,f)=\bigcup_{n\in\N}f^{-n}(W^s_{\varepsilon}(f^n(x),f))$ which is $C^1$ immersed manifold.

\end{enumerate}
The analogous statement holds for the unstable set.
\end{thm}

A \emph{basic piece} $\Lambda\subset M$ of a diffeomorphism $f:M\to M$ is a hyperbolic set having the following two properties:

\begin{enumerate}

\item $\Lambda$ is transitive,

\item $\Lambda$ is a local maximal invariant set: there exists a neighbourhood $U$ of $\Lambda$ so that
\small
$$\Lambda=\bigcap_{n\in\Z}f^n(U).$$
\normalsize

\end{enumerate}

Whenever $\Lambda=\bigcap_{n\in\N}f^{n}(U)$ we say that $\Lambda$ is an \emph{attractor}. If $\Lambda$ is an attractor
for $f^{-1}$ we say that $\Lambda$ is a \emph{repellor}.  The dimension of the unstable bundle is constant on $\Lambda$, due to the continuity of the stable and unstable bundles and the existence of a transitive point. We call this integer \emph{index} of $\Lambda$, which in surfaces can be $0, 1,$ or $2$:
$0$ for an attractor, $2$ for a repellor. A basic piece of index $1$ is called a \emph{saddle} piece.

\smallskip

A \emph{wandering point} $x$ of $f:M\to M$ is a point contained in an open set $U$ so that $U\cap f^n(U)=\emptyset$
for every positive integer. The \emph{non-wandering} set is given by the complement of the set of wandering points, and is usually denoted by $\Omega(f)$. A diffeomorphism $f:M\to M$ is called \emph{axiom A} if
\begin{itemize}

\item[(i)] $\Omega(f)=\textrm{cl}[\textrm{Per}(f)]$,

\item[(ii)] $\Omega(f)$ is hyperbolic.

\end{itemize}

We have the following renowned result by S. Newhouse.

\begin{thm}[Spectral decomposition theorem, \cite{new}]\label{t.spedect}

Let $f:M\rightarrow M$ be an axiom A diffeomorphism. Then there exist a finite number of pairwise disjoint basic pieces $\Lambda_1,...,\Lambda_n\subset M$ for $f$ such that $\Omega(f)=\Lambda_1\cup...\cup\Lambda_n$. Further, for each $i=1,...,n$ there exists a finite number of compact and pairwise disjoint sets $\Lambda_{i,0},...,\Lambda_{i,k_i-1}$ such that:

\begin{itemize}

\item for every $j=0,...,k_i-1$ we have $f(\Lambda_{i,j})=\Lambda_{i,j+1\mbox{ }mod(k_i)}$;

\item for every $j=0,...,k_i-1$ and $x\in \Lambda_{i,j}$ the sets $W^u(x,f),W^s(x,f)$ are dense in $\Lambda_{i,j}$;

\item for every $j=0,...,k_i-1$ we have that $f^{k_i}\mid_{\Lambda_{i,j}}$ is topologically mixing.

\end{itemize}

\end{thm}

\medskip

In the case where $\Omega(f)$ has topological dimension zero, the dynamics on each basic piece can be related to a sub-shift of finite type. In what follows we recall some aspects of this identification. 










In this article we call a {\em partition} of a set $K\subset\Sigma_g$ to a covering $\mathcal{P}$ of $K$ by finite disjoint sets. We say that $\mathcal{P}$ is {\em regular} if its elements are topological squares, intersecting $K$ in relative open sets of $K$.
Given a partition $\mathcal{P}=\{P_1,\ldots,P_N\}$ of some $f$-invariant set $K$, we may consider the {\em itinerary function} $\xi:K\to\{1,\ldots,N\}^{\Z}$, defined by
$$\xi_x(i)=j\mbox{ iff }f^i(x)\in P_j.$$

We say that a regular partition $\mathcal{P}$ of an $f$-invariant set $K$ is a \emph{Markov Partition}
whenever $\xi$ is a homeomorphism defining a conjugacy between $f:K\to K$ and a
sub-shift of finite type $\sigma:\Sigma_A\to\Sigma_A$, $\Sigma_A\subset\{1,\ldots,N\}^{\Z}$. Recall that a sub-shift of finite type is the restriction of the full-shift to a family of sequences associated to the an incidence $n\times n$
matrix $A$.

For a basic piece, we say that a Markov partition is \emph{adapted} when its elements are rectangles with sides contained in some local stable or unstable manifolds. If the piece is an atractor/repellor we expect all the sides of the rectangles to be in the local stable/unstable manifold of that piece. 
If it is a saddle, the rectangles would have a pair of opposite sides in some stable manifolds, and the other two sides in unstable manifolds.












We are now ready to state the theorem about Markov partitions associated to totally disconnected basic pieces. The result that associates Markov partitions to basic pieces is due to R. Bowen \cite{bow}, and though it works in any dimension, we are going to use a more refined statement for surfaces. Denote by $d_{\mathcal{P}}$ the maximal diameter of sets in $\mathcal{P}$.

\begin{thm}\cite{Be}\label{t.partitions}

Let $\Lambda$ be a totally disconnected basic piece of an axiom A surface diffeomorphism $f:\Sigma\to \Sigma$. Then for every $\varepsilon>0$ there exists an adapted Markov partition $\mathcal{P}$ of $\Lambda$ with $d_{\mathcal{P}}<\varepsilon$. 

\end{thm}















\smallskip

Hence Theorem \ref{t.spedect} and Theorem \ref{t.partitions} give us a precise description of the dynamics on the non-wandering set of an axiom A diffeomorphism for which $\Omega(f)$ has dimension zero.

\smallskip

We say that an axiom A diffeomorphism $f:M\to M$ is \emph{fitted} if:
\begin{itemize}

\item[(i)] $\Omega(f)$ is zero-dimensional,

\item[(ii)] whenever $z\in W^s(x,f)\cap W^u(y,f)$ then $z$ is a point of transversal intersection between
the two manifolds $W^s(x,f)$ and $W^u(y,f)$ (we denote this by $z\in W^s(x,f)\transv W^u(y,f)$).

\end{itemize}

Denote by $\mathcal{A}(M)$ the set of fitted axiom A diffeomorphisms of $M$, and $\mathcal{A}_0(M)$ the subset of those that are isotopic to the identity. The following is an interesting result by M. Shub and D. Sullivan, showing that fitted axiom A diffeomorphisms are dense in the $C^0$ topology.

\begin{thm}\cite{SS}\label{t.dense}

Let $M$ be a compact manifold. Then $\mathcal{A}(M)$ is dense in $\textrm{Homeo}(M)$ with the $C^0$ topology.

\end{thm}

A nice exposition of this result in surfaces can be found in \cite{Franks}. A complementary interesting result is given by Z. Nitecki, showing the semi-stability of these axiom A diffeomorphisms in the $C^0$ topology. Recall that a map $g$ is semi-conjugated to $f$ if there exists a continuous and onto map $h$ so that $h\circ g=f\circ h$.

\begin{thm}\cite{Zni}

Let $f:M\to M$ be an axiom A diffeomorphism with the transversality condition (ii). Then, there exists a $C^0$ neighbourhood $\mathcal{U}(f)\subset \textrm{Homeo}(M)$ of $f$ such that any $g\in \mathcal{U}(f)$ is semi-conjugated to $f$ by a semiconjugacy $h$ homotopic to the identity.

\end{thm}

\smallskip

Following, we state a list of well known properties for this kind of maps.

\begin{prop}\label{p.propfaxa}

Let $f:M\to M$ be a fitted axiom A diffeomorphism. Then the following properties hold:

\begin{enumerate}

\item Whenever we have three basic pieces $\Lambda_1,\ \Lambda_2,\ \Lambda_3$ with $W^u(x,f)\cap W^s(y,f)\neq\emptyset$ and $W^u(x',f)\cap W^s(y',f)\neq\emptyset$ for $x\in\Lambda_1,\  y\in\Lambda_2, x'\in\Lambda_2$ and $y'\in\Lambda_3$, then we have that $W^u(x,f)\cap W^s(z,f)\neq\emptyset$ for some
    $z\in\Lambda_3$.

\item For any basic piece $\Lambda$ there exists a neighbourhood $V$ of $\Lambda$ so that
if $W^u(x',f)\cap V\neq\emptyset$ for $x'\in\Omega(f)$ then
$W^u(x',f)\cap W^s(x,f)\neq\emptyset$ for some $x\in\Lambda$. In particular these two manifolds have
a transverse intersection.

\item Given any basic piece $\Lambda$ we have that
$\Lambda^u=\textrm{cl}\left[\bigcup_{x\in\Lambda}W^u(x,f)\right]$ is an attractor. If we take the stable manifold
instead the unstable manifold we obtain a repellor.

\end{enumerate}

\end{prop}

The first point is based in the so called $\lambda$-Lemma. The second and third result
are folklore, and can be obtained from the $C^1$-\emph{stability} of the fitted axiom A diffeomorphisms (see \cite{Katok} for these results). 

One of the key objects we will consider is the {\em heteroclinic relation} on the set of basic pieces of a fitted axiom A diffeomorphism $f$. If $\mathcal{G}_f=\{\Lambda_1,\ldots\Lambda_N\}$ is the set of basic pieces of $f\in\mathcal{A}(M)$, we write that $\Lambda_i\prec\Lambda_j$ if $W^u(x,f)\cap W^s(y,f)\neq\emptyset$ for some $x\in\Lambda_i$ and $y\in\Lambda_j$. The relation $\prec$ is a partial order on $\mathcal{G}_f$ by Proposition \ref{p.propfaxa} (point (1) for transitivity, and point (2) and the definition of basic piece for antisymmetry). When $\Lambda_i\prec\Lambda_j$ we say that $\Lambda_i$ {\em precedes} $\Lambda_j$ or that $\Lambda_j$ {\em follows} $\Lambda_i$. It is a standard fact that the equivalence relation generated by $\prec$ has a single class, i.e. the directed graph structure on $\mathcal{G}_f$ spanned by $\prec$ is connected.

We call a \emph{chain} to any totally ordered set of basic pieces. A maximal chain is a chain which is not properly contained in any other chain.

\begin{rem}\label{r.piecespower} For $f\in \mathcal{A}(M)$ there exists $k\in\N$ so that every basic piece of $f^k$ is topologically mixing. Moreover, given $\Lambda,\Lambda'$ basic pieces of $f^k$, we have that $\Lambda\prec\Lambda'$ if and only if there are $\Lambda_\ast,\Lambda_\ast'$ basic pieces of $f$ satisfying $\Lambda\subseteq \Lambda_\ast$, $\Lambda'\subseteq \Lambda_\ast'$ and $\Lambda_\ast\prec\Lambda_\ast'$\end{rem}



\subsection{Sub-shifts and shadowing}

In what follows we consider sub-shifts of finite type $\sigma:\Sigma_A\subset\Sigma_m\to\Sigma_A$.
We call a \emph{word} of $\Sigma_A$ to the restriction of
an element of $\Sigma_A$ to a subset of the form $\{i,i+1,\ldots,i+n\}\subset\Z$, and its \emph{length} to the number $n+1$. If $\tau$ is such a word, its length will be denoted by $|\tau|$. Given two
words $\tau_1$ and $\tau_2$ with domains $\{i,\ldots,i+n\}$ and $\{i+n+1,\ldots,i+n+k\}$ respectively, the \emph{concatenation} is given by $\tau_1\cdot\tau_2:\{i,\ldots,i+n+k\}\to \{0,\dots,m-1\}$ defined as
$\tau_1(j)$ whenever $j\in\{i,\ldots,i+n\}$ and as $\tau_2(j)$ whenever $j\in\{i+n+1,\ldots,i+n+k\}$. Note that $\tau_1\cdot\tau_2$ may or may not be a word of $\Sigma_A$, depending on the values $\tau_1(i+n)$ and $\tau_2(i+n+1)$.
The concatenation of finitely many and even infinitely many words is naturally defined. The next lemma is very simple so we omit the proof.

\begin{lemma}\label{l.preconvca1}

Consider a transitive sub-shift of finite type $\sigma:\Sigma_A\to\Sigma_A$. Then there exists a positive integer $E$ so that given any two words $\tau_1$ and $\tau_2$ defined in $\{a-n_1,\ldots,a\}$ and $\{b,\ldots,b+n_2\}$ respectively, with $b-a>E$, we have a word $\nu$ defined on $\{a+1,\ldots,b-1\}$, such that
$$\tau_1\cdot \nu\cdot \tau_2$$
is a word of $\Sigma_A$.
\end{lemma}

If $\tau$ is a word of $\Sigma_A$ defined on $\{i,\ldots,i+n\}$ and $m\in\Z$, we define $\tau_m$ on $\{i+m,\ldots,i+n+m\}$ by $\tau_m(j)=\tau(j-m)$. Notice that $\tau_m$ is also a word of $\Sigma_A$. It consists of the same sequence of symbols as $\tau$, but with shifted domain.
We say that a word $\tau$ defined on $\{i,\ldots,i+p-1\}$ is \emph{periodic} if $\tau\cdot\tau_p$ is a word, where $\tau_p$ is defined in $\{i+p,\ldots,i+2p-1\}$ as above. From a periodic word $\tau$ we can concatenate in a natural way and define powers $\tau^n:=\tau\cdot\tau_p\cdot\ldots\cdot \tau_{(n-1)p}$ of $\tau$ for $n\in\N$, which are also words. Taking the natural infinite concatenation of periodic words $\tau$ we recover the periodic points of $\sigma:\Sigma_A\to\Sigma_A$.

\medskip

In what follows we want to introduce a generalization of the \emph{Shadowing Lemma} for chains of basic pieces of a fitted axiom A diffeomorphism $f$.
For a small enough partition $\mathcal{P}=\{P_1,\ldots,P_n\}$ of a basic piece $\Lambda$, the
classic \emph{Shadowing Lemma} states that any pseudo-orbit whose jumps are contained in the boxes
$P_1,\ldots,P_n$ is shadowed by a real orbit from a distance bounded above by $d_{\mathcal{P}}$, and it is a crucial step towards proving the existence of Markov partitions. In order to study the rotation set generated by chains, we want to shadow certain pseudo-orbits going from one basic piece to another one, which is larger with respect to the heteroclinical relation.
We need to introduce some notation.

\medskip

Let $\mathcal{S}=(\Lambda_1\preceq\cdots\preceq\Lambda_N)$ be a sequence of basic pieces of $f$, possibly with repetition, ordered by the heteroclinical relation.  Let $\mathcal{P}_i$ be a Markov partition of $\Lambda_i$ that induces a conjugation of $f$ with a sub-shift $\Sigma_{A_i}$ on $\Lambda_i$. A {\em word sequence} supported on $\mathcal{S}$ is $(w_1,\ldots,w_N)$ where $w_i$ is a word of $\Sigma_{A_i}$ defined on $I_i\subset\Z$, and $\max I_i<\min I_{i+1}$ for $i=1,\ldots,N-1$.

We also say that the word sequence $(w_1,\ldots,w_N)$ is supported on $\mathcal{C}$, the underlying chain formed by the elements of $\mathcal{S}$. For $n\in\N$, we say that $(w_1,\ldots,w_N)$ is $n$-{\em spaced} if $\min I_{i+1}-\max I_{i}>n$ for every $i=1,\ldots,N-1$. It will be useful to notice that any set of words $w_1,\ldots,w_N$ can be arranged in a $n$-spaced sequence by shifting domains, i.e. there exist $p_1,\ldots,p_N\in\Z$ so that $((w_1)_{p_1},\ldots,(w_N)_{p_N})$ is $n$-spaced.

Finally, the word sequence $(w_1,\ldots,w_N)$ is {\em shadowed} if there exists $x\in M$ so that
$$f^j(x)\mbox{ belongs to the element }P_{w_{i}(j)}\mbox{ of }\mathcal{P}_{i}\mbox{ whenever }j\in I_{i}.$$


\medskip

The result we need is a generalization of the \emph{shadowing lemma} (see \cite{Katok}) that works for a union of basic pieces that belong to some chain. We omit its proof, which runs exactly as the one of the shadowing lemma for basic pieces, making use of the transversality property of the fitted axiom A map.

\begin{thm}\label{t.shadchain}

Let $f:M\to M$ be a fitted axiom A diffeomorphism and $\mathcal{S}=(\Lambda_1\preceq\dots\preceq\Lambda_N)$ a sequence of basic pieces of $f$. Then there exists $\varepsilon_0$ so that for any family of adapted Markov partitions $\mathcal{P}_i$ of $\Lambda_i$, $i=1,\dots,N$, with $d_{\mathcal{P}_i}<\varepsilon_0$, there is a positive integer $n_0$ for which any $n_0$-spaced word sequence $(w_1,\dots,w_N)$ is shadowed.

\end{thm}

\section{Rotation set for fitted axiom A diffeomorphisms}\label{s.rotsetbp}

In the first part of this section we will show that rotation sets of basic pieces are \emph{rational polyhedra}, namely, convex sets of $H^1(\Sigma_g,\R)\subseteq \R^{2g}$ with finitely many extremal points that lie in $\mathbb{Q}^{2g}$. The proof of this fact is based on ideas of \cite{Zi} and \cite{Pass1}. Then we will turn to study the rotation set generated by chains of basic pieces in fitted axiom A diffeomorphisms.

For a chain $\mathcal{C}=\{\Lambda_1\prec\ldots\prec\Lambda_N\}$ of basic pieces of $f\in\mathcal{A}_0(\Sigma_g)$, define

$$\rho_{\mathcal{C}}(f) = \textrm{conv}\left(\bigcup_{i=1}^N\rho_{\Lambda_i}(f)\right) $$

The principal theorem of the section, whose proof relies on Theorem \ref{t.shadchain}, is the following.
\begin{thm}\label{t.rhorhochain}

Let $f\in\mathcal{A}_0(\Sigma_g)$ and $\mathcal{G}$ its graph of basic pieces. Then


$$\rho(f)=\bigcup_{\mathcal{C}}\rho_{\mathcal{C}}(f) $$ where $\mathcal{C}$ ranges over all maximal chains of $\mathcal{G}_f$.

\end{thm}

Together with the results about the rotation sets of basic pieces, this implies that $\rho(f)$ is a finite union of convex sets, which are rational polyhedra. We start with some definitions.

\smallskip

Let $K\subset \Sigma_g$ be an invariant compact set for $f\in\mathcal{A}_0(\Sigma_g)$, admitting a Markov partition $\mathcal{P}$. This implies that $K$ is contained in a finite union of pairwise disjoint topological disks, namely the elements of $\mathcal{P}$.
This allows us to construct a fundamental domain $D\subset\Sigma^{ab}_g$ for the abelian cover, so that every element of $\mathcal{P}$ is contained in the projection of the interior of $D$.
Thus for every $P\in\mathcal{P}$ we have that $\hat{P}:=\pi_{ab}^{-1}(P)\cap D$ is homeomorphic to $P$ through the projection $\pi_{ab}$. In this configuration we say that the fundamental domain $D$ is \emph{adapted} to $\mathcal{P}$.
We say that $\mathcal{P}$ is a \emph{rotational} Markov partition if there is a fundamental domain $D\subset\Sigma^{ab}_g$ that is adapted for both  $\mathcal{P}$ and $f^*(\mathcal{P})=\{f(P):\ P\in\mathcal{P}\}$. In other words, so that both $P$ and $f(P)$ are contained in $\pi_{ab}(\textrm{int}(D))$ for every $P\in \mathcal{P}$.


\smallskip


Consider $\mathcal{P}=\{P_1,\ldots,P_N\}$ be a rotational Markov partition for $K$ with an adapted fundamental domain $D$.  Define $\hat{P}_i:=\pi_{ab}^{-1}(P_i)\cap D$ for $i=1,\ldots,N$, and notice that there are unique
deck transformations $\mathcal{T}_i$ of $\Sigma_g^{ab}$ so that $\hat{f}(\hat{P}_i)\subset \mathcal{T}_i(D)$. Recall that the deck transformations
for this covering are given by translations by vectors in $\Z^{2g}$, so we can write $\mathcal{T}_i(x)=x+T_i,\ T_i\in\Z^{2g}$.
Hence to every $P_i\in\mathcal{P}$ (for $i=1,\ldots,N$) we can associate $T_i\in\Z^{2g}$ with $\hat{f}(\hat{P}_i)\subset D+T_i$.
From this simple fact we can obtain the following important lemma.

\begin{lemma}\label{l.moving}

In the situation above, for every $x\in K$ and every positive integer $n\geq 1$ we have that
$$ \left| \hat{f}^n(\hat{x}) - \sum_{j=0}^{n-1}T_{\xi_x(j)} \right|<\textrm{diam}(D),$$
where $\hat{x}$ is the unique point in $\pi^{-1}_{ab}(x)\cap D$ and $\xi_x$ is the itinerary function for the rotational Markov partition applied to $x$.

\end{lemma}

\begin{proof}

We proceed by induction proving the following fact: for every $n\geq 1$ we have $\hat{f}^n(\hat{x})$
in $D+\sum_{j=0}^{n-1}T_{\xi_x(j)}$. The base case $n=1$ is trivial, so assume that it holds for $n-1\in\N$. Then
$$\hat{f}^{n-1}(\hat{x})\in\left(D+\sum_{j=1}^{n-2}T_{\xi_x(j)}\right)\cap \pi^{-1}_{ab}(P_{\xi_x(n-1)})=\hat{P}_{\xi_x(n-1)}+ \sum_{j=1}^{n-2}T_{\xi_x(j)}.$$

Now, using the fact that these Deck transformations commute with the map $\hat{f}$, we have
$$\hat{f}^n(\hat{x})=\hat{f}(\hat{f}^{n-1}(\hat{x}))\in \hat{f}(\hat{P}_{\xi_x(n-1)}) + \sum_{j=1}^{n-2}T_{\xi_x(j)}\subset D+\sum_{j=1}^{n-1}T_{\xi_x(j)}.$$

\end{proof}

This clearly implies the following corollary.

\begin{cor}\label{c.rotsetsim}

Let $f\in\textrm{Homeo}_0(\Sigma_g)$, $K$ an $f$-invariant set and $\mathcal{P}$ a rotational
Markov partition of $K$. Then,

\small
$$\rho_{K}(f)=\left\{\lim_{i}\frac{\sum_{j=1}^{n_i-1} T_{\xi_{x_i}(j)}}{n_i}:\ x_i\in\pi_{\textrm{ab}}^{-1}(K)\cap D,\ n_i\nearrow +\infty\right\}$$
\normalsize

where $\xi_{x_i}$ is the itinerary of $x_i$ for the rotational Markov partition $\mathcal{P}$.

\end{cor}

From this fact we can apply the result of \cite{Zi} to conclude the following.

\begin{cor}\label{c.convbp}

Let $f\in\textrm{Homeo}_0(\Sigma_g)$, and $K$ an $f$-invariant set that admits a rotational
Markov partition $\mathcal{P}$. Then $\rho_K(f)$ is a polyhedron of $\R^{2g}$ with vertices
in $\mathbb{Q}^{2g}$, each of which is realized by a periodic orbit.

\end{cor}

We now turn to study the rotation set generated by chains of basic pieces of fitted axiom A diffeomorphisms. Given a chain $\mathcal{C}=\{\Lambda_1\prec\cdots\prec\Lambda_N\}$
of $f\in\mathcal{A}_0(\Sigma_g)$ define

$$\rho_{\mathcal{C}}(f) = \textrm{conv}\left(\bigcup_{i=1}^N\rho_{\Lambda_i}(f)\right) $$


 the \emph{rotation set of the chain}. Our goal is to prove the following result:

\begin{prop}\label{p.convchain}

Let $\mathcal{C}=\{\Lambda_{1}\prec\cdots\prec\Lambda_N\}$ be a chain of basic pieces of $f\in\mathcal{A}_0(\Sigma_g)$.
Then $\rho_{\mathcal{C}}(f)$ is a rational polyhedron contained in the rotation set $\rho(f)$.

\end{prop}

\smallskip

By Corollary \ref{c.convbp} we know that $\rho_{\mathcal{C}}(f)$ is a rational polyhedron in $\R^{2g}$ so we only need to prove the inclusion $\rho_{\mathcal{C}}(f)\subset \rho(f)$.

Denote $E_i=\{v_{i,1},\ldots,v_{i,m_i}\}$ the set of extremal points of $\rho_{\Lambda_i}(f)$ for $i=1,\ldots,N$ and recall from Corollary \ref{c.convbp} that the elements of $E_i\subseteq\mathbb{Q}^{2g}$ are realized by periodic orbits in $\Lambda_i$. We introduce the set of indices  $J=\{(i,j):i=1,\ldots,N,\ j=1,\ldots,m_i\}$ which we consider lexicographically ordered. We also consider the sequence $\mathcal{S}=(\Lambda_{i,j}=\Lambda_i:(i,j)\in J)$, which has underlying chain $\mathcal{C}$ and repeats $m_i$ times the piece $\Lambda_i$.

Take $\epsilon_0>0$ given by Theorem \ref{t.shadchain} for the sequence $\mathcal{S}$, and $\mathcal{P}_1,\ldots,\mathcal{P}_N$ rotational Markov partitions of $\Lambda_1,\ldots,\Lambda_N$ respectively, satisfying $d_{\mathcal{P}_i}<\epsilon_0$. Consider also $n_0$ given by Theorem \ref{t.shadchain} for the sequence $\mathcal{S}$ and the partitions $\mathcal{P}_{i,j}=\mathcal{P}_i$. For $i=1,\ldots,N$ consider the sub-shift $\Sigma_{A_i}$ given by $\mathcal{P}_i$ and for each $(i,j)\in J$ take $w_{i,j}$ the periodic word of the sub-shift $\Sigma_{A_i}$ associated to the periodic orbit in $\Lambda_i$ realizing $v_{i,j}$.



Given $X=(k_{i,j})\in \N^J$ we can choose $(p_{i,j})\in \mathbb{N}^J$ so that shifting $w_{i,j}^{k_{i,j}}$ by $p_{i,j}$ we obtain a $n_0$-spaced word sequence of the form
$$w(X)=((w_{i,j})^{k_{i,j}}_{p_{i,j}}:(i,j)\in J).$$

Note that the lexicographic order on $J$ gives the order of this sequence.
There are many possible choices for $(p_{i,j})$ that lead to different word sequences as above. To determine $w(X)$ uniquely, we choose $(p_{i,j})$ with $p_{1,1}=0$ and so that the gaps between the domains of the words $(w_{i,j})^{k_{i,j}}_{p_{i,j}}$ have length $n_0+1$.


As a consequence of Theorem \ref{t.shadchain} we have the following.


\smallskip


\begin{lemma}\label{l.prevrotsetchain}

In the situation above, for any $X\in\N^J$ the word sequence $w(X)$ is shadowed.

\end{lemma}

The goal now is to compute the rotation vectors generated by points shadowing the word sequences involved in Lemma \ref{l.prevrotsetchain}. So take $X\in\N^J$ and $x\in\Sigma_g$ that shadows $w(X)$. We consider $$n= \sum_{(i,j)\in J} k_{i,j}|w_{i,j}| + (n_0+1)(|J|-1)$$ which is the number of iterations needed for $x$ to complete the shadowing of $w(X)$, recalling that $|J|-1$ is the number of gaps between the domains of the words of $w(X)$ and $n_0+1$ is the length of those gaps. We remark that both $x$ and $n$ depend on $X$, while $N_0 := (n_0+1)(|J|-1)$ is constant.

Let $\hat{x}$ be the lift of $x$ to a fundamental domain $D\subset\Sigma_g^{ab}$, which is adapted for all the partitions $\mathcal{P}_i$ and $f^*(\mathcal{P}_i)$ for $i=1,\ldots,N$. Then Lemma \ref{l.moving}, together with a telescopic argument, gives us the estimation
$$\left| \hat{f}^n(\hat{x}) - \sum_{(i,j)\in J} k_{i,j}|w_{i,j}|v_{i,j} \right| < \delta_f N_0 + \textrm{diam}(D)|J|$$

where $$\delta_f = \sup\{|\hat{f}(y)-y|:y\in\Sigma_g^{ab}\}$$ Notice that $\delta_f<\infty$ since $\hat{f}$ is a lift of a map that is isotopic to the identity.

Now we fix $t_{i,j}\in[0,1]$ for $(i,j)\in J$ with $\sum_{(i,j)\in J} t_{i,j}=1$. It is clear that we can take $X\to\infty$ (meaning, each coordinate goes to $\infty$) along a sequence that makes $\frac{k_{i,j}|w_{i,j}|}{n}\to t_{i,j}$. So we get that

$$\lim_{X\to\infty}\frac{\hat{f}^n(\hat{x})}{n} = \lim_{X\to\infty}  \sum_{(i,j)\in J} \frac{k_{i,j}|w_{i,j}|}{n}v_{i,j} = \sum_{(i,j)\in J} t_{i,j}v_{i,j}. $$

Thus we find that any convex combination of the points in $E_1\cup\cdots\cup E_N$ is obtained as a limit in the definition of rotation set, concluding that $\rho_{\mathcal{C}}(f)\subset \rho(f)$ and proving Proposition \ref{p.convchain}.






\medskip

Thus, we obtained that for $f\in\mathcal{A}_0(\Sigma_g)$ we have
\begin{equation}\label{e.1inc}
\bigcup_{\mathcal{C}}\rho_{\mathcal{C}}(f)\subset\rho(f)
\end{equation}
where $\mathcal{C}$ ranges over all chains of $\mathcal{G}_f$. Clearly we obtain the same result considering only the maximal chains.


Note that the sets appearing in the union in \ref{e.1inc} are rational polyhedra in $\R^{2g}$. In what follows we prove the complementary inclusion, in order to obtain Theorem \ref{t.rhorhochain}. We start with two preliminary results.

\begin{lemma}\label{l.entornos1}

Let $\Lambda_1,\ldots,\Lambda_N$ be the basic pieces of $f\in\mathcal{A}_0(\Sigma_g)$. There exists a family of pairwise disjoint neighbourhoods $V_1,\ldots,V_N$ of $\Lambda_1,\ldots,\Lambda_N$ respectively, so that
whenever $f^n(V_i)\cap V_j\neq \emptyset$ for some $n\in\N$, we have $\Lambda_i\prec\Lambda_j$.

\end{lemma}

\begin{proof}

We start by considering, for each basic piece $\Lambda_i$, the neighbourhood $V'_i$ given by item (2) of Proposition \ref{p.propfaxa}. Thus whenever $\Lambda_i\not\prec \Lambda_j$ we have that $W^u(\Lambda_i,f):=\bigcup_{x\in\Lambda_i}W^u(x,f)$ does not meet $V'_j$. Further, as $\textrm{cl}(W^u(\Lambda_i,f))$ is an attractor (by item (3) of Proposition \ref{p.propfaxa}), we can consider a forward-invariant neighbourhood $S_{i,j}$ which we can assume does not meet $V'_j$. Let
$$S_i=\bigcap_{j:\ \Lambda_i\not\prec\Lambda_j}S_{i,j}$$
and consider $V_i=V'_i\cap S_i$ for $i=1,\ldots,N$. These neighbourhoods satisfy the statement: Let $i_0,j_0\in\{1,\ldots,N\}$ so that $f^n(V_{i_0})\cap V_{j_0}\neq\emptyset$ for some positive integer $n$. Then $f^n(S_{i_0})\cap V'_{j_0}\neq\emptyset$, so $\Lambda_{i_0}\prec \Lambda_{j_0}$.

\end{proof}

We remark that the neighbourhoods $V_1,\ldots,V_N$ in Lemma \ref{l.entornos1} can be taken arbitrarily small. In the following lemma we make some further observations about these neighbourhoods. We use the notation $$\mathcal{O}(x,f):=\{f^n(x):n\in\Z \}$$  for the orbit of $x$ under $f$.

\begin{lemma}\label{l.saltos2}

Let $\Lambda_1,\ldots,\Lambda_N$ be the basic pieces of $f\in\mathcal{A}_0(\Sigma_g)$, and $V_1,\ldots,V_N$ a family of neighborhoods of these basic pieces as in Lemma \ref{l.entornos1}.
Then
\begin{enumerate}
\item there exists a positive constant $\kappa\in\N$ so that for every $x\in\Sigma_g$ $$\left|\mathcal{O}(x,f)\cap \left(\Sigma_g\setminus \bigcup_{i=1}^N V_i\right)\right|<\kappa.$$

\item If $f^{n_1}(x)\in V_{i_1}$ and $f^{n_2}(x)\in V_{i_2}$ with $n_1<n_2$, then $\Lambda_{i_1}\preceq \Lambda_{i_2}.$

\end{enumerate}
\end{lemma}

\begin{proof}

It is a general fact for a map in a compact metric space that given any neighbourhood of the non-wandering set, the iterations of any point which are contained in the complement of that neighbourhood is uniformly bounded, which implies the first fact. The second fact is straightforward from Lemma \ref{l.entornos1}.
\end{proof}

For the rest of this section we take $f\in\mathcal{A}_0(\Sigma_g)$, consider $\Lambda_1,\ldots,\Lambda_N$ its basic pieces and let $V_1,\ldots,V_N$ be the family of neighbourhoods given in Lemma \ref{l.entornos1}. Since we can take these neighbourhoods arbitrarily small, we may assume that the components of $V_i$  form a rotational Markov partition $\mathcal{P}_i$ of $\Lambda_i$, for $i=1,\ldots,N$. The following is a consequence of the shadowing lemma for hyperbolic pieces.

\begin{lemma}\label{l.saltos3} In the situation above, consider $x\in\Sigma_g$ and $i\in\{1,\ldots,N\}$ so that $f^j(x)\in V_i$ for all $j\in I=\{a,\ldots,a+m\}$. Then there exists $x'\in\Lambda_i$ so that $f^j(x)$ and $f^j(x')$ belong to the same element of $\mathcal{P}_i$ for every $j\in I$.
\end{lemma}

We fix $\kappa$ as in Lemma \ref{l.saltos2} for the neighbourhoods $V_1,\ldots,V_N$ under consideration. Then Lemma \ref{l.saltos2} implies the following.






\begin{cor}\label{l.entornos3} In the context above, for every $x\in\Sigma_g$ and every $n\in\N$ we have a unique sequence of intervals $I_1,\ldots,I_r\subset \{0,\ldots,n\}$, and a sequence of basic pieces $\Lambda_{i_1},\ldots,\Lambda_{i_r}$, so that

\begin{enumerate}
\item $\max I_{s}<\min I_{s+1}$ for $s=1,\ldots,r-1$,
\item $\left|\{0,\ldots,n\}\setminus\bigcup_{s=1}^r I_s\right|<\kappa$,
\item $f^j(x)\in V_{i_s}$ for $j\in I_s$, $s=1,\ldots,r$, and the intervals $I_s$ are maximal among those with this property,
\item $\Lambda_{i_1}\preceq\cdots\preceq\Lambda_{i_r}$.




\end{enumerate}

\end{cor}

Notice that both the number of intervals (i.e. $r$) and their lengths are dependent on $x$ and $n$. However, due to (2) and the maximality condition in (3), we have that $r$ is bounded by $\kappa+1$ plus the maximum length of a chain (observe that (4) allows for repetition of basic pieces). Denote $\mathcal{S}(x,n)=(\Lambda_{i_1}\preceq\cdots\preceq\Lambda_{i_r})$ the sequence of basic pieces given in Corollary \ref{l.entornos3}, and by $\mathcal{C}(x,n)$ its underlying chain.

We are now ready to conclude that $\rho(f)\subset\bigcup_{\mathcal{C}}\rho_{\mathcal{C}}(f)$ which, together with Proposition \ref{p.convchain}, implies Theorem \ref{t.rhorhochain}. To do so, take an arbitrary $v\in\rho(f)$ and write it as $$v=\lim_k\frac{\hat{f}^{n_k}(x_k)-x_k}{n_k}$$ where $n_k\nearrow\infty$ and $x_k\in\Sigma_g^{ab}$. Let $y_k=\pi_{ab}(x_k)$, and notice that $\mathcal{S}(y_k,n_k)$ varies in a finite set, since its length is bounded. Thus, up to taking a sub-sequence we may assume that $\mathcal{S}(y_k,n_k)$ is constant and denote it by $\mathcal{S}=\mathcal{S}(y_k,n_k)=(\Lambda_{i_1}\preceq\cdots\preceq\Lambda_{i_r})$. We will show that $v\in\rho_{\mathcal{C}}(f)$ where $\mathcal{C}=\mathcal{C}(y_k,n_k)$ is the underlying chain of $\mathcal{S}$.

For each $k\in\N$ let $I^{(k)}_1,\ldots,I^{(k)}_r$ be the intervals in $\{0,\ldots,n_k\}$ obtained applying Corollary \ref{l.entornos3} to $y_k$ and $n_k$. As noted in the paragraph above, $r$ is constant, but the intervals themselves depend on $k$ and they cover $\{0,\ldots,n_k\}$ up to at most $\kappa$ elements. Denote $I^{(k)}_s=\{a^{(k)}_s,\ldots,b^{(k)}_s\}$ for $s=1,\ldots,r$, and also set $b^{(k)}_0=0$ and $a^{(k)}_{r+1}=n_k$. Then we have $$0=b^{(k)}_0\leq a^{(k)}_1\leq b^{(k)}_1 < a^{(k)}_2 \leq b^{(k)}_2 < \cdots < a^{(k)}_r\leq b^{(k)}_r \leq a^{(k)}_{r+1} = n_k $$ $$\sum_{s=0}^r (a^{(k)}_{s+1}-b^{(k)}_s) \leq \kappa $$ and $f^j(y_k)\in V_{i_s}$ whenever $a^{(k)}_s\leq j \leq b^{(k)}_s$. In order to avoid heavy notation we will write $a_i=a^{(k)}_i$ and $b_i=b^{(k)}_i$, omitting the explicit dependence on $k$ until we need to focus on it.


Next we expand $\hat{f}^{n_k}(x_k)-x_k$ as a telescopic sum
$$\hat{f}^{n_k}(x_k)-x_k = (\hat{f}^{a_{r+1}}(x_k) - \hat{f}^{b_r}(x_k)) + (\hat{f}^{b_r}(x_k) - \hat{f}^{a_r}(x_k))+\cdots +  (\hat{f}^{a_1}(x_k)-\hat{f}^{b_0}(x_k)) $$
which allows us to obtain the bound
$$\left| (\hat{f}^{n_k}(x_k) - x_k) - \sum_{s=1}^r (\hat{f}^{b_s}(x_k)-\hat{f}^{a_s}(x_k)) \right| \leq \kappa\delta_f $$ where we recall that $\delta_f = \sup\{|\hat{f}(y)-y|:y\in\Sigma_g^{ab}\}$.
Now we consider $y^s_k\in\Lambda_{i_s}$ for $s=1,\ldots,r$, given by applying Lemma \ref{l.saltos3} to $y_k$ for each $i_s$.  Thus $f^j(y_k)$ and $f^j(y^s_k)$ are in the same element of $\mathcal{P}_{i_s}$ whenever $a_s\leq j \leq b_s$. Let $x^s_k$ be the lift of $y^s_k$ such that $\hat{f}^{a_s}(x^s_k)$ is in the same fundamental domain as $\hat{f}^{a_s}(x_k)$. Equivalently, we may require that $\hat{f}^{a_s}(x^s_k)$ is in the same component of $\pi_{ab}^{-1}(V_{i_s})$ as $\hat{f}^{a_s}(x_k)$.

Since $\mathcal{P}_{i_s}$ is a rotational Markov partition, we get that $\hat{f}^{b_s}(x^s_k)$ is in the same fundamental domain as  $\hat{f}^{b_s}(x_k)$, using a similar argument to that of Lemma \ref{l.moving}. Thus $$\left| (\hat{f}^{b_s}(x_k)-\hat{f}^{a_s}(x_k)) - (\hat{f}^{b_s}(x^s_k)-\hat{f}^{a_s}(x^s_k))  \right|\leq 2\, \textrm{diam}(D)$$ where we recall that $\pi_{ab}(x^s_k)\in\Lambda_{i_s}$.

Combining the two previous bounds, we may conclude that $$v=\lim_k \frac{1}{n_k} \sum_{s=1}^r (\hat{f}^{b_s}(x_k)-\hat{f}^{a_s}(x_k)) = \lim_k \frac{1}{n_k}\sum_{s=1}^r (\hat{f}^{b_s}(x^s_k)-\hat{f}^{a_s}(x^s_k)) $$ which we will write as $$v = \lim_k \sum_{s=1}^r \frac{b_s-a_s}{n_k}\cdot
\frac{\hat{f}^{b_s}(x^s_k)-\hat{f}^{a_s}(x^s_k)}{b_s-a_s} $$


As we study the limit as $k\to+\infty$, it is convenient to recall explicitly the dependence on $k$ of $a_s=a^{(k)}_s$ and $b_s=b^{(k)}_s$. Recalling their definition and point (2) in Corollary \ref{l.entornos3}, it is clear that $$\sum_{s=1}^r \frac{b^{(k)}_s-a^{(k)}_s}{n_k} \to 1  \mbox{ and } \frac{b^{(k)}_s-a^{(k)}_s}{n_k}\in [0,1].$$ Passing to an appropriate sub-sequence in $k$ we may assume convergence: $$\frac{b^{(k)}_s-a^{(k)}_s}{n_k}\to t_s  \mbox{ for } t_s\in [0,1] \mbox{ with } \sum_{s=1}^r t_s = 1, \mbox{ and also }$$ 
$$  \lim_k\frac{\hat{f}^{b^{(k)}_s}(x^s_k)-\hat{f}^{a^{(k)}_s}(x^s_k)}{b^{(k)}_s-a^{(k)}_s} = v_s \mbox{ for } v_s \in \rho_{\Lambda_{i_s}}(f)  $$
Therefore we get $v = \sum_{s=1}^r t_sv_s \in \rho_{\mathcal{C}}(f) = \textrm{conv}(\rho_{\Lambda_{i_1}}(f),\ldots,\rho_{\Lambda_{i_r}}(f)) $, as desired. This concludes the proof of Theorem \ref{t.rhorhochain}.

\section{Quasi-invariant surfaces and Conley decomposition}\label{s.conley}

We introduce here two important tools that will be used in the next sections: \emph{quasi-invariant surfaces} (defined below) and decompositions of our configuration space $\Sigma$ given by the Fundamental Theorem of the Dynamical Systems \cite{conleyfundthm}. These two objects are strongly related, as the mentioned decompositions of $\Sigma$ are given by families of surfaces that may  only fail to be \emph{quasi-invariant} by small technical details. On the other hand some \emph{quasi-invariant surfaces} that we will employ are not associated to these decompositions. Let us start by defining quasi-invariant surfaces for a
homeomorphism $f:\Sigma\to\Sigma$.

When we talk about a {\em subsurface} $S\subseteq\Sigma$ we understand it to be compact and topologically embedded in $\Sigma$, not necessarily connected. As such, the boundary $\partial S$ is a union of disjoint simple closed curves, which is non empty unless $S=\Sigma$, and the complement $\Sigma\setminus S$ is also a subsurface.

Given an homeomorphism $f:\Sigma\to\Sigma$
we say that $S$ is \emph{quasi-invariant} under $f$ if there exists a disjoint family of annuli $A_{\gamma}$, one for each boundary component $\gamma\subseteq\partial S$ so that:
\begin{enumerate}

\item $A_{\gamma}\subseteq\Sigma\setminus \Omega(f)$,
\item $A_{\gamma}\cap S=\gamma$,
\item $f(S)\cup f^{-1}(S)\subseteq S\bigcup(\cup_{\gamma} A_{\gamma})$,
\item Either $f(\gamma)$ or $f^{-1}(\gamma)$ is contained in $A_{\gamma}\setminus S$.

\end{enumerate}


If $S\subseteq\Sigma$ is a subsurface, we denote by $\textrm{Fill}(S)$ the union of $S$ and all the connected components of $\Sigma\setminus S$ that are topological disks. We say that $S$ is {\em filled} when $\textrm{Fill}(S)=S$.

\begin{rem}\label{r.quasiconexas} Notice that if $S$ is quasi-invariant then its connected components and $\text{Fill}(S)$ are also quasi-invariant. 
\end{rem}

Given a connected subsurface $S\subset\Sigma$ let $\textrm{i}_*:\pi_1(S)\to \pi_1(\Sigma)$ be the morphism induced by the inclusion map. We say that a connected subsurface $S$ is:

\begin{itemize}

\item \emph{trivial} whenever $\textrm{Im}(\textrm{i}_*)$ is given by the trivial element,

\item \emph{annular} whenever $\textrm{Im}(\textrm{i}_*)$ is isomorphic to $\Z$,

\item \emph{curved} whenever $\textrm{Im}(\textrm{i}_*)$ is not trivial nor isomorphic to $\Z$.
\end{itemize}

Next we want to introduce some important lemmas concerning quasi-invariant surfaces that will be used in the following sections. We start by introducing the relevant notations. For a subsurface $S$ of $\Sigma$ we denote by $H_1(S)$ the image in $H_1(\Sigma;\R) = \R^{2g}$ of the homology map $H_1(S;\R)\to H_1(\Sigma;\R)$ induced by the inclusion. Notice that $H_1(S)=H_1(\textrm{Fill}(S))$. 

Every time we consider a closed oriented surface $\Sigma$ of negative Euler characteristic, we shall fix a Poincar\'{e} covering $\pi:\D\to\Sigma$. Recall that in this situation, given any map $f:\Sigma\to\Sigma$ in the homotopy class of the identity, we have a unique lift $F:\D\to\D$ of $f$ which extends as the identity map to $\partial\D$. We call it the \emph{canonical lift} of $f$.

\begin{lemma} \label{l.qisurface}
Let $f\in \mathcal{A}_0(\Sigma)$ and $S\subset \Sigma$ be a connected surface that is quasi-invariant under $f$. Then:
\begin{enumerate}

\item $S\cap\Omega(f)$ is $f$-invariant. Moreover, it is a union of basic pieces,

\item If $S$ is non-trivial and $\tilde{S}$ is a connected component of $\pi^{-1}(S)$ we have that $\tilde{S}\cap\pi^{-1}(\Omega(f))$ is $F$-invariant,

\item If $S$ is trivial, then there is a closed curve $\alpha$ in $\Sigma$ such that $\rho_{\Lambda}(f) = \{[\alpha]\}$ for every basic piece $\Lambda$ contained in $S$,

\item If $S$ is non-trivial and $\Lambda\subseteq S$ is a basic piece, we have $\rho_{\Lambda}(f) \subset H_1(S)$.

\end{enumerate}

\end{lemma}

\begin{proof}

Consider $(A_{\gamma})$ the disjoint family of closed annuli given by the quasi-invariance of $S$ and define $$S_1:=S\cup\left(\bigcup_{\gamma\subseteq\partial S}A_{\gamma}\right).$$ Since $S_1\setminus S$ is contained in the complement of $\Omega(f)$ and $f(S)\cup f^{-1}(S)\subseteq S_1$ we deduce that $f(S\cap\Omega(f))$ and $f^{-1}(S\cap\Omega(f))$ are contained in $S$ proving the first part of point (1). The second part follows from the first and the transitivity of basic pieces, noting that $S\cap\Omega(f)$ is both open and closed in $\Omega(f)$. 

For point (2) assume that $S$ is non trivial and consider $\tilde{S}$ a connected component of $\pi^{-1}(S)$. Let $\tilde{S}_1$ be the connected component of $\pi^{-1}(S_1)$ containing $\tilde{S}$. Since $F$ lifts $f$ we know that $F^{\pm 1}(\tilde{S})$ must be contained in some lift of $S_1$. On the other hand, $F$ restricts to the identity on $\partial\mathbb{D}$ and the different lifts of $S_1$ can be identified by their accumulation points on $\partial\mathbb{D}$, so  we get that $F^{\pm 1}(\tilde{S})\subseteq\tilde{S}_1$. Observe that $\tilde{S}_1\setminus\tilde{S}\subseteq \pi^{-1}(\cup_{\gamma} A_{\gamma})$, which does not meet $\pi^{-1}(\Omega(f))$, and this concludes the proof of point (2) by the same argument used for point (1).

We proceed with the proof of point (3). Consider $f_t$ a homotopy between $f_0=Id$ and $f_1=f$, and let $F_t$ be the lift of $f_t$ to the universal cover $\mathbb{D}$ that begins at $F_0 = Id$. According to our notation, we have $F_1=F$. Denote $D:=\text{Fill}(S)$ and $D_1:=\text{Fill}(S_1)$. Under the assumption for point (3)  both $D$ and $D_1$ are disks. Notice also that $f(D)\subseteq D_1$. Let $\tilde{D}\subseteq\tilde{D}_1$ be lifts of $D_0$ and $D_1$. Then $F(\tilde{D})$ is contained in $g\tilde{D}_1$ for some  $g\in\pi_1(\Sigma)$, and we take $\alpha$ a closed curve representing $g$. If $\Lambda\subset S$ is a basic piece and $x\in\Lambda$, we consider a lift $\tilde{x}\in \tilde{S}$ and notice that $\{F_t(\tilde{x})\}_{t\in[0,1]}$ can be homotoped to $\alpha$ with endpoints varying inside fundamental domains of the covering (to be more specific, inside $\tilde{D}$ and $g\tilde{D}_1$ respectively). As this holds for any $x\in\Lambda$ (including $f^k(x)$ for $k\geq 1$) we see that $\rho_{\Lambda}(f) = \{[\alpha]\}$ as desired.

Finally we turn to point (4). We can assume that $S$ is filled, since $H_1(\text{Fill}(S))=H_1(S)$ and the filling of a quasi-invariant surface is also quasi-invariant. Let $\pi_S:\hat{\Sigma}_S\to\Sigma$ be the covering space that corresponds to the subgroup $\textrm{i}_*\pi_1(S)\leq\pi_1(\Sigma)$. Observe that $\hat{\Sigma}_S$ consists of a $1:1$ lift of $S$, that we denote $\hat{S}$, together with open annuli attached to each component of $\partial \hat{S}$. Since $S_1$ retracts by deformation onto $S$ we have that $\hat{\Sigma}_S$ also contains a $1:1$ lift of $S_1$, that we call $\hat{S}_1$, with $\hat{S}\subset \hat{S}_1$.

Let $\overline{f}_t$ be the lift of the homotopy $f_t$ to this cover with $\overline{f}_0=Id$.   We denote by $\overline{f}=\overline{f}_1$  the corresponding lift of $f$. Since $\hat{S}_1\setminus \hat{S}\subseteq\pi_S^{-1}(\cup A_{\gamma})$, which does not meet $\pi_S^{-1}(\Omega(f))$, we deduce that $\pi_S^{-1}(\Omega(f))\cap \hat{S}$ is $\overline{f}$-invariant. If $\Lambda\subset S$ is a basic piece and $x\in\Lambda$, we take a lift $\hat{x}\in\hat{S}$ and notice that $\overline{f}^n(\hat{x})\in \hat{S}$ for every $n\in\N$. Since $\hat{\Sigma}_S$ retracts by deformation onto $\hat{S}$ we get that $\{\overline{f}_t(\hat{x})\}_{t\in[0,n]}$ is homotopic with fixed endpoints to a curve contained in $\hat{S}$. Projecting this homotopy by $\pi_S$ we get a homotopy with fixed endpoints of $\{f_t(x)\}_{t\in[0,n]}$ into $S$. As this holds for every $x\in\Lambda$,  we get that   $$\rho_{\Lambda}(f)\subseteq H_1(S)$$ for every basic piece $\Lambda\subseteq S$.


\end{proof}

For a fixed point $x_0$ of $f$ we introduce the notation $\rho(f,x_0)$ for the corresponding rotation vector. In terms of the homotopy $f_t$  of $Id$ to $f$, we may write $\rho(f,x_0)$ as the homology class of the loop $\{f_t(x_0)\}_{t\in [0,1]}$. In terms of the abelian lift $\hat{f}$, we have $\rho(f,x_0)=\hat{f}(\hat{x}_0)-\hat{x}_0 $ for $\hat{x}_0\in \Sigma_{ab}$ lifting $x_0$.

\begin{lemma} \label{l.filling}
Let $f\in \mathcal{A}_0(\Sigma)$ and $S\subset \Sigma$ be a non-trivial connected surface that is quasi-invariant under $f$.
\begin{enumerate}
\item If $S$ is curved, then there is a zero-rotation fixed point in $S$.
\item If $\Lambda$ is a basic piece contained in $\text{Fill}(S)\setminus S$, then there is a fixed point $x_0$ in $S$ with $\rho_{\Lambda}(f)=\rho(f,x_0)$.
\end{enumerate}
\end{lemma}

\begin{proof} Following the notation in the proof of point (4) of Lemma \ref{l.qisurface}, consider $\pi_S:\hat{\Sigma}_S\to \Sigma$ the covering space associated to $\textrm{i}_*\pi_1(S)=\textrm{i}_*\pi_1(\text{Fill}(S))$, and $\hat{S}$ the closed lift of $S$ to this cover. In this case $S$ is not necessarily filled, but the construction is the same as the one for $\text{Fill}(S)$, and $\hat{S}$ is still homeomorphic to $S$. Also consider an homotopy $f_t$  between $f_0=Id$ and $f_1=f$, take $\overline{f}_t$ the lift of $f_t$ to $\hat{\Sigma}_S$ starting at $\overline{f}_0 = Id$, and let $\overline{f}=\overline{f}_1$, which is a lift of $f$. 


Define $$S_0:= \text{Fill}\left(\bigcup_{n\in\Z}\overline{f}^n(\hat{S})\right).$$

Notice that $S_0$ is $\overline{f}$-invariant by construction, and by the properties of quasi-invariance for $S$, it consists of $\text{Fill}(\hat{S})$ together with semi-open annuli $C_{\gamma}$  attached to each boundary component $\gamma\subset\partial\text{Fill}(\hat{S})$. Moreover, for each such $\gamma$ we have $C_{\gamma}\cap\pi_S^{-1}(\Omega(f))=\emptyset$, and in particular $C_{\gamma}$ contains no fixed points of $\overline{f}$. We shall consider $\overline{f}$ restricted to $S_0$. It is isotopic to the identity in $S_0$ (which can be seen by a standard construction), and has finitely many fixed points. 


The key of the proof is the {\em Lefschetz-Nielsen index theorem} applied to $\overline{f}$ in $S_0$. The first equation given by this theorem relates the Euler characteristic of $S_0$ with the sum of indices of contractible fixed points of $\overline{f}$:
\begin{equation}\label{ecu1}
\sum_{x\in\text{Fix}(\overline{f})\,:\,\rho(\overline{f},x)=0} \text{Ind}(\overline{f},x) = \chi(S_0)
\end{equation}
On the other hand, recall that by Lemma \ref{l.qisurface} (point (2)), if $D$ is a quasi-invariant disk, there exists $[\alpha]\in H_1(S_0)$ (possibly $0$) so that every basic piece $\Lambda\subseteq D$ satisfies $\rho_{\Lambda}(\overline{f})=[\alpha]$. In particular, all fixed points in $D$ have the same rotation vector. Therefore, by the classical Lefschetz index theorem we get:

\begin{equation}\label{ecu2}
\sum_{x\in\text{Fix}(\overline{f})\cap D\,:\,\rho(\overline{f},x)=v} \text{Ind}(\overline{f},x) =  \left\{ \begin{array}{ll}
         1 & \mbox{if $v =[\alpha]$};\\
        0 & \mbox{otherwise}.\end{array} \right.
\end{equation}

To prove point (1) notice that since $S$ is curved we have $\chi(S_0)<0$. Therefore equations \ref{ecu1} and \ref{ecu2} imply the existence of $p\in\text{Fix}(\overline{f})\cap\hat{S}$ with $\rho(\overline{f},p)=0$. Thus $\pi_S(p)$ is the desired fixed point.

The proof of point (2) is very similar. In this case we need to consider the second part of the Lefschetz-Nielsen index theorem, for fixed points in a non-zero rotation class. Namely, if $[\alpha]\neq 0$ then:

\begin{equation}\label{ecu3} \sum_{x\in\text{Fix}(\overline{f})\,:\,\rho(\overline{f},x)=[\alpha]} \text{Ind}(\overline{f},x) = 0
\end{equation}

Consider $\Lambda$ the basic piece in the hypothesis of point (2), and let $\hat{\Lambda}$ be the corresponding lift in $\text{Fill}(\hat{S})$. Notice that $\rho_{\Lambda}(f) = (\pi_S)_*\rho_{\hat{\Lambda}}(\overline{f})$, thus it suffices to work with $\overline{f}$ in $S_0$ . First assume that $\rho_{\hat{\Lambda}}(\overline{f})\neq 0$. Notice that the disk containing $\hat{\Lambda}$ contributes with a $+1$ in the left side of \ref{ecu3} while the other disks in the complement of $\hat{S}$ contribute with non-negative values. Therefore, applying equations \ref{ecu2} and \ref{ecu3} to $v=\rho_{\hat{\Lambda}}(\overline{f})$ we prove point (2) for this case.

If $\rho_{\hat{\Lambda}}(\overline{f})=0$ and $S$ is curved we are in the hypothesis of case (1). It only remains the case where $S$ is annular, in which $\chi(S_0)=0$ and the disk containing $\hat{\Lambda}$ produces a $+1$ contribution in the left side of equation \ref{ecu1}. Thus equations \ref{ecu1} and \ref{ecu2} give the desired result.

\end{proof}

\begin{rem}\label{r.filling} With similar arguments as those in the proof of Lemma \ref{l.filling}, we can show that if $S$ is a quasi-invariant surface that is not homeomorphic to an annulus, then there exists a fixed point $p\in S$. This holds even when $S$ is trivial, which is the only case not covered in Lemma \ref{l.filling}.
\end{rem}

Next we recall Conley's Fundamental Theorem for homeomorphisms in $\mathcal{A}_0(\Sigma)$.

\begin{thm}[\cite{conleyfundthm}]\label{t.conley}

Let $f\in\mathcal{A}_0(\Sigma)$. Then there exists a continuous function
$\mathcal{L}:\Sigma\to\R$ so that:

\begin{enumerate}

\item $\mathcal{L}(x)\geq \mathcal{L}(f(x))$ for all $x\in\Sigma$, and

\item for any pair $x,y\in\Omega(f)$ we have $\mathcal{L}(x)=\mathcal{L}(y)$ iff $x$ and $y$
belong to the same basic piece of $f$.

\end{enumerate}

\end{thm}

A function $\mathcal{L}$ as in Theorem \ref{t.conley} is called a \emph{Lyapunov function}. Lyapunov functions are not unique, and neither are the partitions of $\Sigma$ induced by their
level sets. By the construction in \cite{conleyfundthm}, $\mathcal{L}$ can be considered smooth in the complement of $\Omega(f)$.

Our aim is to give a decomposition of $\Sigma$ into subsurfaces, each containing one basic piece of $f$, that comes from some partition by level sets of a Lyapunov function. That will be useful in order to decompose $\Sigma$ into surfaces with quasi-invariance properties, each containing a single basic piece. The starting idea is to take a partition of $\Sigma$ using level sets of a Lyapunov function $\mathcal{L}$, choosing the levels as to isolate the different basic pieces of $f$. There are technical issues with such a partition by level sets that force us to make some modifications. Namely, we do not want these surfaces to have components contained in the wandering set of $f$. The gist of these modifications is to attach each of these wandering components, if they exist, to a suitable neighbouring subsurface. Next we give the detail, including a precise definition for what we consider as {\em Conley surface} associated to $\mathcal{L}$.

We point out that, for ease of notation, we call a subsurface {\em wandering} if it does not meet $\Omega(f)$.


Given a surface $S$ we define $\text{WFill}(S)$ as the union of $S$ with all the connected components of $\Sigma\setminus S$ which are homeomorphic to either disks or annuli and do not meet $\Omega(f)$.  We say that $\text{WFill}(S)$ is the \emph{wandering filling}, or \emph{W-filling} of $S$, and that $S$ is {\em W-filled} when $S=\text{WFill}(S)$. Notice that disjoint subsurfaces meeting $\Omega(f)$ have disjoint W-fillings. On the other hand, the W-filling of their union may be strictly larger than the union of their respective W-fillings, as there may be wandering annuli bounded between the two subsurfaces.

As motivation for our definition of W-filling, we point out that for a surface of the form  $S=\mathcal{L}^{-1}([a,b])$, where  $\mathcal{L}$ is a Lyapunov function and $a,b\in\R\setminus\mathcal{L}(\Omega(f))$, the W-filling of $S$ agrees with the union of {\em all} the components of $\Sigma\setminus S$ that do not meet $\Omega(f)$. The proof of this fact follows that of Lemma \ref{l.filling}, and we omit it.


In what follows we consider a Lyapunov function $\mathcal{L}$ that is smooth in the complement of $\Omega(f)$. Given a basic piece $\Lambda$, we say that a subsurface $S$ is a \emph{Conley surface} of $\Lambda$ if it verifies the following conditions:


\begin{enumerate}

\item Each component of $\partial S$ is contained in some level set $\mathcal{L}^{-1}(a)$ with $a\notin \mathcal{L}(\Omega(f))$. (Where $a$ depends on the component of $\partial S$).
\item $S\cap\Omega(f)=\Lambda$,
\item every connected component of $S$ meets $\Lambda$,
\item $S$ is W-filled.
\end{enumerate}

Points (3) and (4) are technical, their purpose is to avoid wandering components in $S$ or $\Sigma\setminus S$. The existence of Conley surfaces follows from Theorem \ref{t.conley}. 
Moreover, we can decompose $\Sigma$ into Conley surfaces, as we show in the next result.

\begin{lem}\label{l.masconley} Given $f\in\mathcal{A}_0(f)$ with basic pieces $\Lambda_1,\ldots,\Lambda_N$ and a Lyapunov function $\mathcal{L}$ that is smooth in $\Sigma\setminus\Omega(f)$, we can find a family of surfaces $S_{1},\ldots,S_{N}$ associated to $\mathcal{L}$ so that:
\begin{enumerate}

\item $\textrm{int}(S_i)\cap\textrm{int}(S_j)=\emptyset$ whenever $i\neq j$.

\item $\bigcup_{i=1}^N S_i=\Sigma$,

\item For every $i=1,\ldots,N$ we have that $S_i$ is a Conley surface for $\Lambda_i$.

\end{enumerate}

\end{lem}

A family of surfaces that satisfies the statement of Lemma \ref{l.masconley} will be called a \emph{Conley decomposition} of $\Sigma$ for $f$.

\tiny
\begin{figure}[h!]\begin{center}
\centerline{\includegraphics[scale=.5]{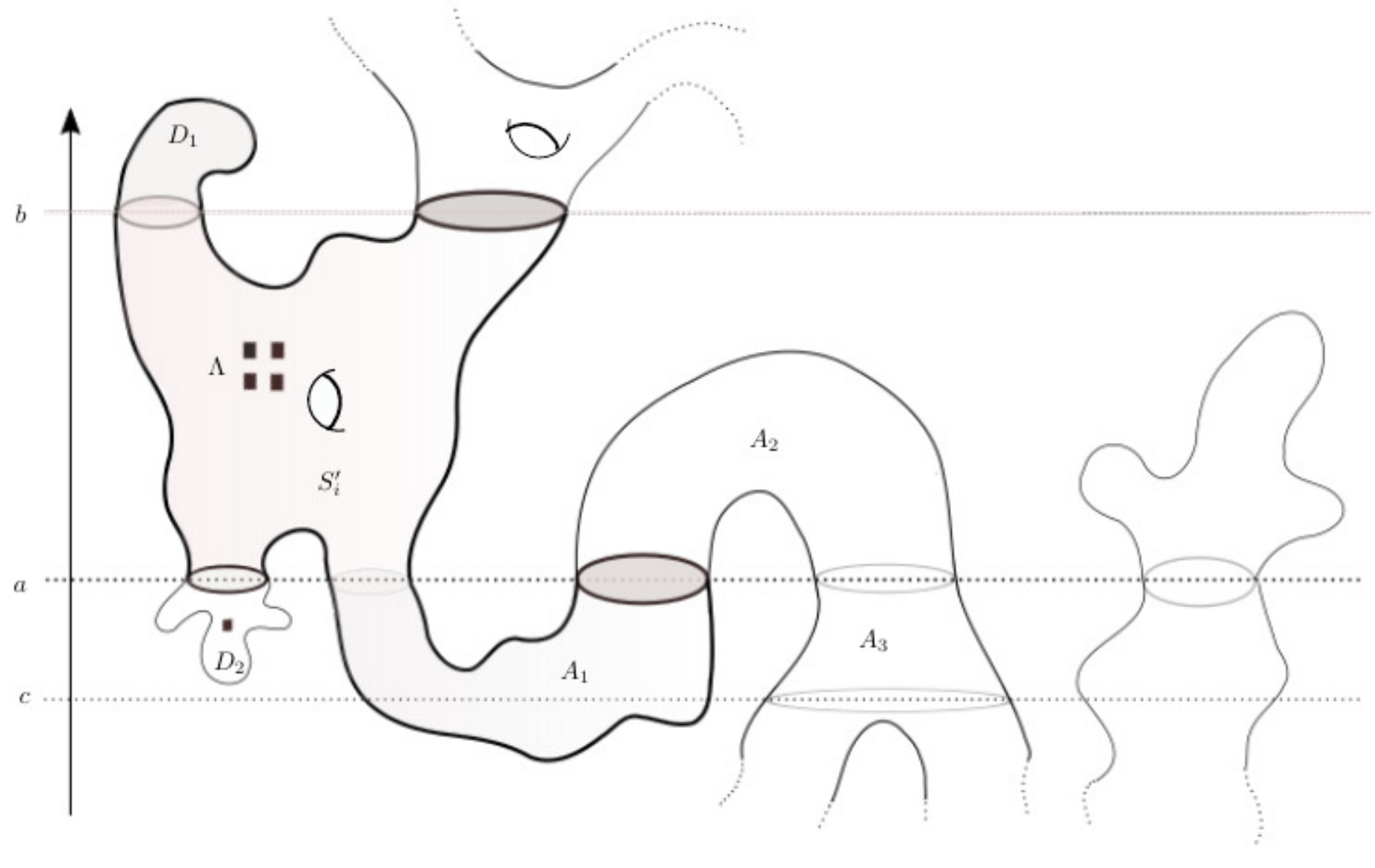}}
\caption{The picture illustrates in gray a Conley surface, which can be obtained as follows: starting with $S'_i$, then taking its \textrm{WFill} by adding up the disk $D_1$ and not $D_2$, as this last disk meets the non-wandering set, which is depicted by the black box. Then one can add adjacent
wandering annuli as is the case of $A_1$, and could further add the annuli $A_2$ and $A_3$.}\label{shift}
\end{center}\end{figure}
\normalsize

\begin{proof}

As we outlined earlier, we begin by taking $a_1,\ldots,a_{N+1}\in\R\setminus\mathcal{L}(\Omega(f)) $ as regular values of $\mathcal{L}$ so that $S^0_i=\mathcal{L}^{-1}([a_i,a_{i+1}])$ contains $\Lambda_i$. Let $S'_i$ be the union of all the connected components of $S^0_i$ that meet $\Omega(f)$. These surfaces satisfy (1)-(3) in the definition of Conley surfaces, and are pairwise disjoint except possibly at their boundaries. We also observe that if $\Sigma\setminus \cup_i S'_i$ is non empty, then its connected components are contained in the wandering set and consist on either disks or annuli, as can be proved by the same argument used for Lemma \ref{l.filling}.

We take $S''_i = \textrm{WFill}(S'_i)$. Notice that $S''_i$ is already a Conley surface for the basic piece $\Lambda_i$. Their interiors are also pairwise disjoint, by the properties of the W-filling.

On the other hand, $\cup_iS''_i$ may not yet cover $\Sigma$. Passing to $S''_i = \textrm{WFill}(S'_i)$ adds in the wandering disk components of $\Sigma\setminus \cup_i S'_i$, as well as the wandering annuli that have both boundary components in the same surface $S'_i$. (In fact, these annuli are not present, but the proof is cumbersome and we will not need it). Thus the components of $\Sigma\setminus \cup_i S''_i$ are wandering annuli, with boundary components on different surfaces $S''_i$. For each of these annuli, we choose one neighbouring surface $S''_i$ and attach the annuli to it.

This gives the construction of the surfaces $S_i$ in the statement.

\end{proof}

\begin{rem} \label{r.masconley}
\noindent
\begin{enumerate}
\item The construction in Lemma \ref{l.masconley} gives Conley surfaces $S$ that have the following form:

There are $a,b\in\R\setminus\mathcal{L}(\Omega(f))$ so that $S$ is the W-fill of the union of the components of $\mathcal{L}^{-1}([a,b])$ that meet $\Omega(f)$, together with wandering annuli, each attached to one boundary component.

\item Every Conley decomposition of $\Sigma$ associated to the Lyapunov function $\mathcal{L}$ is obtained from the construction in the proof of Lemma \ref{l.masconley}, except possibly by varying each curve on the decomposition within the same class of the Morse Theory of $\mathcal{L}$.

\item With similar arguments as in the proof of Lemma \ref{l.masconley}, we can obtain that any Conley surface for $f$ is part of a Conley decomposition. Thus Lemma \ref{l.masconley} gives a general way to construct Conley surfaces.

\end{enumerate}

\end{rem}


As we pointed out at the beginning of this section, Conley and quasi-invariant surfaces are closely related. Showing this relationship will be our next goal. We begin with a preliminary result that will be useful in general for proving quasi-invariance of subsurfaces. It is worth to point out that this is specific for hyperbolic surfaces, and is not true for the torus.

\begin{lemma}\label{l.wandering} Consider $f\in\mathcal{A}_0(\Sigma)$ and $\gamma$ an essential simple closed curve satisfying $$f^n(\gamma)\cap\gamma=\emptyset \text{ for every } n>0$$ Then the annulus bounded by $\gamma$ and $f^n(\gamma)$ does not meet $\Omega(f)$ for every $n>0$.
\end{lemma}
\begin{proof} Since $f$ is isotopic to the identity, $f^n(\gamma)$ is homotopic to $\gamma$  for all $n$, and the subsurface of $\Sigma$ with boundary $\gamma\cup f(\gamma)$ is an annulus $A_0$. We consider the semi-open annulus $A=A_0\setminus f(\gamma)$. 
Suppose there exists $k>0$ so that $f^k(A)\cap A\neq\emptyset$. We can consider the first such $k$, so that $f^m(A)\cap  A=\emptyset$ for $0<m< k$. Notice that $f^k(\gamma)$ cannot meet $A$, since $f^{k-1}(A)$ is adjacent to $f^k(A)$ at $f^k(\gamma)$. Since $f^{k+1}(\gamma)\cap (\gamma\cup f(\gamma))=\emptyset$, this leaves us with two possibilities: either $A\subset f^k(A)$ or $f^{k+1}(\gamma)$ is contained in the interior of $A$. The first possibility is ruled out by considering the annulus $B$ between $f^k(\gamma)$ and $\gamma$, and observing that $f$ would map $B$ to $f(B)$ reversing orientation. To exclude the second one, observe that $\cup_{i=0}^{i=k}f^i(A)$ would be homeomorphic to a torus, which contradicts the hypothesis that $\Sigma$ is a connected surface of higher genus.


Therefore we have that $f^i(A)\cap A=\emptyset$ for every $i>0$, which implies that $\cup_{i=0}^{i=\infty}f^i(A)$ is a semi-open annulus disjoint from $\Omega(f)$. In particular, the annulus between $\gamma$ and $f^n(\gamma)$ does not meet $\Omega(f)$ for every $n>0$.
\end{proof}

The following result gives the relationship we seek, going from Conley surfaces to quasi-invariant surfaces.  

\begin{lemma}\label{l.conleyesquasi} Let $f\in\mathcal{A}_0(\Sigma)$ and $S$ be a Conley surface which has a non-trivial connected component. Then the following hold.
\begin{enumerate}
\item $S$ is connected,
\item there exists $k\in\N$ so that $S$ is quasi-invariant under $f^k$,
\item $\text{Fill}(S)$ is quasi-invariant under $f$.
\end{enumerate}
Moreover, if $S'$ is another connected and non-trivial Conley surface that is adjacent to $S$ in the same Conley decomposition of $\Sigma$, then $\text{Fill}(S)\cup \text{Fill}(S')$ is also quasi-invariant.
\end{lemma}

\begin{proof} Take $S_{\ast}$ a non-trivial connected component of $S$. We shall split the proof into a series of claims, showing first that $\text{Fill}(S_{\ast})$ is quasi-invariant and that $S=S_{\ast}$. This will imply points (1) and (3) immediately, and will be useful for the rest of the statement, that will be left for the end of the proof.

In order to simplify notation we will assume first that $S$ is the W-filling of the connected components that meet $\Omega(f)$ of $\mathcal{L}^{-1}([a,b])$ for some $a,b\in\R$. That is to say that $S$ has the form given in Remark \ref{r.masconley} point (1), with no wandering annuli attached to the boundary. In other words, we are assuming $S$ is of the form $S''_i$ in the proof of Lemma \ref{l.masconley}. We shall prove Claims 1 to 3 below assuming this form for $S$, and afterwards we will generalize for $S$ as in point (1) of Remark \ref{r.masconley}, possibly with wandering annuli attached at the boundary. By points (2) and (3) of the same Remark \ref{r.masconley}, this gives the lemma for $S$ a general Conley surface.


Denote $S_0:=\text{Fill}(S_{\ast})$ and for each boundary component $\gamma\subseteq\partial S_0$ define

\begin{equation}
\gamma':=  \left\{ \begin{array}{ll}
         f(\gamma) & \mbox{if $\mathcal{L}(\gamma)=a$};\\
        f^{-1}(\gamma) & \mbox{if $\mathcal{L}(\gamma)=b$}.\end{array} \right.
\end{equation}

{\bf Claim 1:} $S_0$ is quasi-invariant.

\vspace{.1cm}

Observe that $\gamma'\subseteq\Sigma\setminus S_0$ for every $\gamma\subseteq\partial S_0$, by point (1) in Theorem \ref{t.conley}. First we see that $$\gamma_1'\cap\gamma_2'=\emptyset \text{ whenever } \gamma_1\neq\gamma_2$$
When $\gamma_1$ and $\gamma_2$ have the same image under $\mathcal{L}$ it follows from injectivity of either $f$ or $f^{-1}$. When they have different images under $\mathcal{L}$, it follows from point (1) in Theorem \ref{t.conley}.


Notice that since $S_0$ is filled and non-trivial, each component $\gamma\subseteq\partial S_0$ is essential and so $\gamma$ and $\gamma'$ bound an annulus $A_{\gamma}$, adjacent to $S_0$ by $\gamma$. By point (1) in Theorem \ref{t.conley} we see that $f^n(\gamma)\cap\gamma=\emptyset \text{ for every } n>0$, thus we may apply Lemma \ref{l.wandering} to conclude that $A_{\gamma}\cap\Omega(f)=\emptyset$. Next we see that $$A_{\gamma_1}\cap A_{\gamma_2}=\emptyset \text{ whenever } \gamma_1\neq\gamma_2$$

Suppose this is not the case. Since all the boundary curves are pairwise disjoint and both annuli are disjoint from $S_0$, we get that $\gamma'_1\subset \text{int}(A_{\gamma_2})$ and $C:=A_{\gamma_1}\cup A_{\gamma_2}$ is an annulus. Also notice that $\text{int}(C)$ is a component of $\Sigma\setminus S_0$ and $C\cap \Omega(f)=\emptyset$, but this contradicts the fact that $S_0$ is W-filled.



Define $$S_1:=S_0\cup(\cup_{\gamma} A_{\gamma})$$

We have that $f(S_0)\subseteq S_1$, since  $f(S_0)\cap S_1$ is non-empty and both open and closed in $f(S_0)$. Similarly $f^{-1}(S_0)\subseteq S_1$. Therefore we have shown that $S_0$ is quasi-invariant under $f$, using the family of annuli $A_{\gamma}$. This shows Claim 1.



\vspace{.1cm}
{\bf Claim 2:} $S=S_{\ast}$.

\vspace{.1cm}

Let $\Lambda$ be the basic piece associated to $S$. We obtained from the previous claim that $\Omega(f)\cap S_0$ is $f$-invariant, so $\Lambda\subset S_0$. Thus if there are any components of $S$ other than $S_{\ast}$, they must be contained in $S_0 = \text{Fill}(S_{\ast})$, for they have to meet $\Lambda$. In particular these components must be trivial and contained in $S_0\setminus S_{\ast}$. We shall introduce some notation for the components of $S_0\setminus S_{\ast}$, which are all disks with boundaries in either $\mathcal{L}^{-1}(a)$ or $\mathcal{L}^{-1}(b)$. Let $D^-_1,\ldots,D^-_m$ be the components of $S_0\setminus S_{\ast}$ whose boundaries are contained in $\mathcal{L}^{-1}(a)$, and let $\mathcal{D}^- = \cup_i D_i$. Analogously define $D^+_1,\ldots,D^+_l$ and $\mathcal{D}^+$ for the components with boundary in $\mathcal{L}^{-1}(b)$.

We will show that $\mathcal{D}^-\cap\Omega(f)$ is $f$-invariant. This will imply, by transitivity of the basic pieces, that $\Lambda$ does not meet $\mathcal{D}^-$ and thus $S$ has no components in $\mathcal{D}^-$. The same will hold for $\mathcal{D}^+$, and that will give this Claim.

To prove that $\mathcal{D}^-\cap\Omega(f)$ is $f$-invariant, we look at $f(D_i^-)$ for each $i=1,\ldots,m$. On one hand $D_i^-\subset S_0$, thus $f(D_i^-)\subset S_1 $. On the other, by Theorem \ref{t.conley} and the fact that $\partial D_i^-\subset \mathcal{L}^{-1}(a)$, we see that $f(D_i^-)$ lies in $S_1\cap \mathcal{L}^{-1}((-\infty,a))$ which is contained in $\mathcal{D}^-\cup (\cup_{\gamma}A_{\gamma})$. Since $D_i^-$ is connected, $f(D_i^-)$ must be contained either in some component of $\mathcal{D}^-$, or in $A_{\gamma}$ for some $\gamma\subset\partial S_0$. Since $S_{\ast}$ is W-filled, we have that $D_i^-$ meets $\Omega(f)$, and so does $f(D_i^-)$, thus it cannot be contained in any of the wandering annuli $A_{\gamma}$. We obtain that $f(D_i^-)\subset \mathcal{D}^-$. As this holds for every $i=1,\ldots,m$, we get $f(\mathcal{D}^-)\subset \mathcal{D}^-$ and so $\mathcal{D}^-\cap\Omega(f)$ is $f$-invariant, as we desired. The same proof works for $\mathcal{D}^+$, and this proves Claim 2.

\vspace{.1cm}
{\bf Claim 3:} Points (1), (2) and (3) in the statement.
\vspace{.1cm}

Point (1) comes directly from Claim 2, recalling the definition of $S_{\ast}$, and point (3) from combining Claims 1 and 2, recalling that $S_0 = \text{Fill}(S_{\ast}) = \text{Fill}(S)$.

It remains to show point (2). From the proof of Claim 2 we see that for each $i=1,\ldots,m$ we have $f(D_i^-)\subset D_{\sigma(i)}^-$ for some $\sigma(i)\in\{1,\ldots,m\}$. By transitivity of basic pieces $\sigma$ must be a permutation of $\{1,\ldots,m\} $. Similarly $f^{-1}(D_j^+)\subset D_{\tau(j)}^+$ for a permutation $\tau$ of $\{1,\ldots,l\}$. Pick $k$ so that $\sigma^k=Id$ and $\tau^k=Id$. Thus $f^k(D_i^-)\subset\text{int}(D_i^-)$ and $f^{-k}(D_j^+)\subset\text{int}(D_j^+)$ for every $i=1,\ldots,m$ and $j=1,\ldots,l$, and this defines annuli $A^-_i$ and $A^+_j$ as the respective closures of $D_i^-\setminus f^k(D_i^-)$ and $D_j^+\setminus f^{-k}(D_j^+)$. It is clear that these annuli do not meet $\Omega(f)$. For $\gamma$ in the boundary of $S_0$, which is the essential boundary of $S$, let

$$
A'_{\gamma} :=  \left\{ \begin{array}{ll}
  \cup_{s=1}^k f^s(A_{\gamma})        & \mbox{if $\mathcal{L}(\gamma)=a$};\\
    \cup_{s=1}^k f^{-s}(A_{\gamma}) & \mbox{if $\mathcal{L}(\gamma)=b$}.\end{array} \right.
$$

It is straightforward to verify the definition of quasi-invariance under $f^k$ for the set of annuli we just described, consisting on $A'_{\gamma}$ for $\gamma$ in the essential boundary of $S$, $A^-_i$ for $i=1,\ldots,m$, and $A^+_j$ for $j=1,\ldots,l$. 
This concludes Claim 3.

\vspace{.1cm}
We have shown Claims 1 through 3 assuming that $S$ is one of the surfaces $S''_i$ in the construction for Lemma \ref{l.masconley}. Attaching some wandering annuli to $S$ would make the notation more cumbersome, and would actually add difficulty only in proving the disjointness of the boundary annuli for the quasi-invariance in Claim 1. This can be sorted out by observing that if an essential wandering annulus is adjacent to $S$ by $\gamma \subset \partial S$, it is eventually contained in $\cup_{s=0}^m f^{s}(A_{\gamma})$ or $\cup_{s=0}^m f^{-s}(A_{\gamma})$ for large enough $m$. (Where $A_{\gamma}$ was defined in Claim 1.) 

\vspace{.1cm}
{\bf Claim 4:} If $S'$ is another connected and non-trivial Conley surface that is adjacent to $S$ in the same Conley decomposition, then $\text{Fill}(S)\cup \text{Fill}(S')$ is quasi-invariant under $f$.
\vspace{.1cm}

To obtain the Conley decomposition under consideration, which includes $S$ and $S'$, we follow the construction for Lemma \ref{l.masconley}. To obtain the statement for a general Conley decomposition we resort to point (2) of Remark \ref{r.masconley}.


By such construction of the Conley decomposition we see that $S\cup S'$ must be W-filled, since any wandering annulus between $S$ and $S'$ should have been already contained in $S$ or $S'$. For any component $\gamma$ in the boundary of $\text{Fill}(S)\cup \text{Fill}(S')$ we can associate an annulus $A_{\gamma}$ that comes from Claim 1 for either $S$ or $S'$. These annuli give the quasi-invariance of $\text{Fill}(S)\cup \text{Fill}(S')$ by similar arguments as those used for Claim 1, so we omit the details. The only delicate point is the disjointness of the annuli $A_{\gamma}$, in which we need to use that $\text{Fill}(S)\cup \text{Fill}(S')$ is W-filled.

\end{proof}

\begin{rem}\label{r.discos} If $S$ is a non-trivial Conley surface then $\text{cl}[\text{Fill}(S)\setminus S]$ can be written as $\mathcal{D}^+\sqcup \mathcal{D}^-$, where $\mathcal{D}^+$ and $\mathcal{D}^-$ are disjoint unions of disks, with $f(\mathcal{D}^+)\subseteq \text{int}(\mathcal{D}^{+})$ and $f^{-1}(\mathcal{D}^-)\subseteq \text{int}(\mathcal{D}^{-})$.
 \end{rem}

Once we choose a fixed Conley decomposition, Lemma \ref{l.conleyesquasi} lets us classify the basic pieces of $f$ into {\em trivial}, {\em annular} or {\em curved}, according to their associated Conley surfaces. Given a basic piece $\Lambda$, let us call $S_{\Lambda}$ to its Conley surface, then

\begin{itemize}
\item $\Lambda$ is {\em trivial} if every component of $S_{\Lambda}$ is trivial. In this case $S_{\Lambda}$ is contained in  a disjoint union of disks,  either attracting or repelling.
\item Otherwise, $S_{\Lambda}$ is connected by Lemma \ref{l.conleyesquasi}. We say that $\Lambda$ is {\em annular/curved} if $S_{\Lambda}$ is annular/curved respectively.
\end{itemize}

\begin{cor} Let $f\in\mathcal{A}_0(\Sigma)$ and $S$ be a Conley surface
of a non-trivial basic piece $\Lambda$. If $\Lambda'\subset \text{Fill}(S)$ is a trivial basic piece, then $\rho_{\Lambda'}(f)\subset\rho_{\Lambda}(f)$.
\end{cor}
\begin{proof} This follows from Lemma \ref{l.conleyesquasi} point (1) and Lemma \ref{l.filling} point (2).
\end{proof}

In general, for $f$ a homeomorphism of $\Sigma$, a subsurface $S\subset\Sigma$, and $n>0$, we denote $$S^{(n)}= \text{Fill}\left(\bigcup_{|j|\leq n} f^j(S)\right)$$
We point out that if $S$ is filled and quasi-invariant then $S^{(1)}=S\cup(\cup A_{\gamma})$ where $A_{\gamma}$ are annuli attached to the components of $\partial S$. The same is true for $S^{(n)}$, by taking $f^n$. Getting this property is the reason why we fill {\em after} taking the union, and not {\em before}, which could yield a strictly smaller set. 
Then Lemma \ref{l.conleyesquasi} implies the following.
\begin{lem} \label{l.conleycasesess0} Assume $S$ is the fill of either
\begin{itemize}
\item  a non-trivial Conley surface or
\item the union of two non-trivial Conley surfaces which are adjacent in the same Conley decomposition of $\Sigma$.
\end{itemize} Then the following hold for every $n>0$,
\begin{enumerate}
\item $S^{(n)}\cap\Omega(f)=S\cap\Omega(f)$,
\item $S^{(n)}$ is connected, filled and retracts by deformation onto $S$ rel $\Omega(f)$,
\item $S^{(n)}$ is quasi-invariant under $f^n$.
\end{enumerate}
\end{lem}

Although Conley decompositions will play the central role in our main Theorem \ref{mainthm}, we will also need another construction of quasi-invariant surfaces for sections \ref{s.agujas} and \ref{s.triviales}. This is given in the next lemma, with the extra assumption of having topologically mixing basic pieces, which by Remark \ref{r.piecespower} applies to some power of any $f\in\mathcal{A}_0(\Sigma)$. Recall that for a basic piece $\Lambda$ of an axiom A diffeomoprhism
$f$ we defined $\Lambda^u=\textrm{cl}[W^u(\Lambda,f)]$.

\begin{lemma}\label{l.construquasi} Suppose $f\in\mathcal{A}_0(\Sigma)$ has topologically mixing basic pieces. Assume $U$ is a subsurface such that $\text{cl}[f(U)]\subseteq\text{int}(U)$, and every connected component of $U$ meets $\Omega(f)$. Let $K\subseteq U$ be a continuum and $\Lambda_1,\ldots,\Lambda_N$ basic pieces satisfying:
\begin{itemize}
\item $\Lambda_i\subseteq U$ for $i=1,\ldots,N$
\item $K\cap[\cup_{i=1}^{i=N}\Lambda_i^u]=\emptyset$

\end{itemize}
Then there exists a quasi-invariant surface $S$ such that $K\subseteq S$ and $$S\cap\Omega(f)=(U\cap\Omega(f))\setminus\{\Lambda:\Lambda_i\preceq\Lambda \text{ for some }i=1,\ldots,N\}$$
\end{lemma}

Together with Remark \ref{r.piecespower}, this will let us build quasi-invariant surfaces with some control over the basic pieces contained in them. By choosing $K$ appropriately, we will also be able to join some of those basic pieces with an arc within the surface $S$. These specific requirements make it difficult to obtain $S$ through an arbitrary Conley decomposition, as union of Conley surfaces, so we provide this alternative construction.

\begin{proof}

Notice that we may replace $U$ with $\text{WFill}(U)$ without loss of generality, so we assume $U$ is W-filled. Consider $U_1,\ldots,U_k$ the connected components of $U$. Since they all meet $\Omega(f)$, by recurrence we have   $f(U_i)\subseteq U_{\sigma(i)}$ where $\sigma$ is a permutation of $\{1,\ldots,k\}$. Since $f$ is topologically mixing on its basic pieces, we get that $\sigma=Id$.

First we shall prove that $U$ is quasi-invariant: For each component $\gamma$ of $\partial U$ define $\gamma':=f^{-1}(\gamma)$. We will show that $\gamma$ and $\gamma'$ bound an annulus $A_{\gamma}$ adjacent to $U$ by $\gamma$. When $\gamma$ is essential this is clear, since $f$ is isotopic to $Id$. When $\gamma$ is inessential there are two cases to distinguish:
\begin{enumerate}
\item $\gamma$ bounds a disk that contains a trivial component $U_j$ of $U$.
\item $\gamma$ bounds a disk in $\text{Fill}(U_j)\setminus U_j$, for $U_j$ any component of $U$.
\end{enumerate}
The first case follows from the fact that $f(U_j)\subset U_j$. For the second one, let $D_1,\ldots,D_m$ be the components of $\text{Fill}(U_j)\setminus U_j$, and $\gamma_l = \partial D_l$ for $l=1,\ldots,m$. Since $U$ is W-filled we have $D_l\cap\Omega(f)\neq \emptyset$ for all $l=1,\ldots,m$, thus $f(D_l)\supset D_{\tau(l)}$ where $\tau$ is a permutation of $\{1,\ldots,m\}$. We get that $\tau=Id$ by the topologically mixing property of the basic pieces, therefore $\gamma_l$ and $\gamma'_l$ bound an annulus for every $l=1,\ldots,m$. The disjointness of the annuli $A_{\gamma}$ comes from the injectivity of $f$ and the fact that $U$ is W-filled, as in the proof of Lemma \ref{l.conleyesquasi}. Thus the quasi-invariance of $U$ follows.

We will arrive at our desired surface $S$ by removing a subsurface $Z$ from $U$, that we shall construct next.
Since $\Lambda_i\subseteq \text{int}(U)$ for $i=1,\ldots,N$ and $U$ is forward-invariant under $f$, we get that $\Lambda_i^u\subseteq U$ for every $i=1,\ldots,N$. Recall that by Proposition \ref{p.propfaxa} point (3) and Lemma \ref{l.entornos1} we have that for every $i=1,\ldots,N$:
\begin{itemize}
\item $\Lambda_i^u$ is an attractor for $f$, and
\item $\Lambda_i^u\cap\Omega(f)=\cup_{\Lambda\succeq\Lambda_i}\Lambda$.
\end{itemize}

Thus for every $i=1,\ldots,N$ there is a subsurface $Z_i\subseteq U$, obtained as a neighbourhood of $\Lambda_i^u$, that satisfies:
\begin{itemize}
\item $\text{cl}[f(Z_i)]\subseteq \text{int}(Z_i)$, and
\item $Z_i\cap\Omega(f)=\Lambda_i^u \cap\Omega(f) =\cup_{\Lambda\succeq\Lambda_i}\Lambda$.
\end{itemize}

Define $Z$ as the union of the connected components of $\text{WFill}(\cup_{i=1}^{i=N}Z_i)$ that meet $\Omega(f)$. We can choose the surfaces $Z_i$ so that their boundaries intersect transversally, thus making $Z$ a subsurface. Since $\text{cl}[f(Z)]\subset \text{int}(Z) $, we get that $Z$ is quasi-invariant, by the same argument we used for $U$. Taking the surfaces $Z_i$ as small enough neighbourhoods of $\Lambda_i^u$ we may assume that  $\text{cl}[Z]\subset \text{int}(U)$ and $Z\cap K=\emptyset$. Then $S = U\setminus Z$ is the desired surface. Checking the details is straightforward.

\end{proof}

As the final point of this section, we turn back to Conley surfaces and look at the dynamic behaviour of their boundaries.

\begin{lem}\label{sl.borde} Let $S$ be a non-trivial Conley surface and $\gamma\subseteq \partial S$ be an essential boundary component. Then the following holds.

\begin{enumerate}
\item If $f(\gamma)\cap S=\emptyset$ then the annulus between $\gamma$ and $f^n(\gamma)$  meets $S$ only at $\gamma$ for every $n>0$,
\item If $f(\gamma)\cap S \neq\emptyset$ then $f^n(\gamma)\cap S\neq\emptyset$ for every $n>0$.
\end{enumerate}

\end{lem}
\begin{proof} Since $\gamma$ is essential, the surface between $\gamma$ and $f^n(\gamma)$ is an annulus $A_n$. If for some $n$ we have that $f^n(\gamma)\cap S=\emptyset$, then $(S\setminus\gamma)\cap A_n$ is both open and closed in $S\setminus\gamma$, thus either $(S\setminus\gamma)\cap A_n=\emptyset$ or $S\setminus\gamma\subseteq A_n$. Lemma \ref{l.wandering} excludes the second possibility, so $(S\setminus\gamma)\cap A_n=\emptyset$ whenever $f^n(\gamma)\cap S=\emptyset$.


For point (1) assume that $f(\gamma)$ does not intersect $S$. By Lemma \ref{l.conleyesquasi} and the fact that $(S\setminus\gamma)\cap A_1=\emptyset$ we have that $f^n(\gamma)\cap S=\emptyset$ for every $n>0$. Therefore we must have $(S\setminus\gamma)\cap A_n=\emptyset$, which gives point (1).

We prove point (2) by contradiction, supposing that $f^n(\gamma)\cap S=\emptyset$ for some $n>0$. Then we get that $(S\setminus\gamma)\cap A_n=\emptyset$, which contradicts the fact that $f(\gamma)\cap S\neq\emptyset$.
\end{proof}

The following relates the essential boundary of a Conley surface with the stable and unstable manifolds of its basic piece.

\begin{lemma}\label{l.conleycasesess}

Let $f\in\mathcal{A}_0(\Sigma)$ and $S$ be a non-trivial Conley surface
associated to a basic piece $\Lambda$. Then for any essential component $b$ of $\partial S$ we have that either $W^s(\Lambda,f)\cap b\neq\emptyset$ or $W^u(\Lambda,f)\cap b\neq\emptyset$.

\end{lemma}

We remark that
\begin{itemize}
\item if $f(b)\cap S =\emptyset$ then $b$ meets $W^u(\Lambda,f)$, while

\item if $f(b)\cap S \neq\emptyset$ (which happens iff $f^{-1}(b)\cap S =\emptyset$), then $b$ meets $W^s(\Lambda,f)$.

\end{itemize}

\begin{proof} 
We shall assume that $f(b)\cap S\neq\emptyset$ and prove that $W^s(\Lambda,f)\cap b\neq\emptyset$. The other case is analogous.  Consider the set
$$A=\{x\in b\,:\, \mbox{there is some }k>0 \mbox{ such that } f^{k}(x)\in\Sigma \setminus S \}\subseteq b$$

We wish to show that $A\neq b$, for if there is a point $z \in b\setminus A$ then all the future iterates of $z$ are contained in $S$, which implies that $z\in W^{s}(\Lambda,f)$ as every point in $\Sigma$ must belong to the stable manifold of some basic piece. Thus we would find $z\in b\cap W^{s}(\Lambda,f)$ as desired.

For $x\in A$ define $$k_0(x)=\min \{k>0:f^k(x)\in\Sigma\setminus S\} $$

\vspace{.1cm}
{\bf Claim:} If $x\in A$ then $f^n(x)\in\Sigma\setminus S $ for every $n\geq k_0(x)$.
\vspace{.1cm}

Recall from Remark \ref{r.discos} that we can write $\text{Fill}(S)\setminus S = \mathcal{D}^+\sqcup \mathcal{D}^- $ where $\mathcal{D}^+$ and $\mathcal{D}^-$ are disjoint unions of disks with $f(\mathcal{D}^+)\subseteq \text{int}(\mathcal{D}^+)$ and $f^{-1}(\mathcal{D}^-)\subseteq \text{int}(\mathcal{D}^-)$. On the other hand we consider the annuli $A_{\gamma}$ associated to the essential components of $\partial S$ as in Lemma \ref{l.conleyesquasi}, namely:
\begin{enumerate}
\item If $f(\gamma)\cap S = \emptyset$ then $A_{\gamma}$ is the annulus between $\gamma$ and $f(\gamma)$,
\item if $f(\gamma)\cap S \neq \emptyset$ then $A_{\gamma}$ is bounded by $\gamma$ and $f^{-1}(\gamma)$
\end{enumerate}

Let $\partial S^+$ be the essential boundary of $S$ that belongs in case (1) above. Then we have that $$f(S)\subseteq S_1:= S\cup \mathcal{D}^+ \cup \left(\bigcup_{\gamma\subset \partial S^+} A_{\gamma}\right) $$
Thus if $x\in A$ and $k_0=k_0(x)$, we have either $f^{k_0}(x)\in \mathcal{D}^+$ or $f^{k_0}(x)\in A_{\gamma}$ for some component $\gamma$ of $\partial S^+$. In the first case, since $\mathcal{D}^+$ is an attractor, we have $f^n(x)\in\mathcal{D}^+$ for all $n>k_0$, so we get the Claim. In the second one, if $n>k_0$ then $f^n(x)\in f^{n-k_0}(A_{\gamma})$, which lies in the annulus between $\gamma$ and $f^{n-k_0}(\gamma)$, and does not meet $\gamma$. Then Lemma \ref{sl.borde} implies $f^n(x)\notin S$, concluding the proof of the Claim.

\vspace{.1cm}

We are ready to prove that $A\neq b$. Suppose they are equal. Then by compactness of $b$ and continuity of $f$ (upper-semicontinuity of $k_0$) there exists $k_1=\max\{k_0(x):x\in b\}$, and by the  Claim we have $f^{k_1}(b)\cap S =\emptyset$. This contradicts point (2) of Lemma \ref{sl.borde}, since $f(b)\cap S \neq\emptyset$ by the hypothesis of the case we are proving. This concludes the proof, by the remarks at the beginning of it.




\end{proof}

\section{Star shape of the rotation set}\label{s.agujas}


We show that for $f\in \mathcal{A}_0(\Sigma)$ the set $\rho(f)\subset H_1(\Sigma_g;\mathbb{R})$ is a union of convex sets all containing the origin.

The proof will be based on our characterization of the rotation set as $$\rho(f) = \bigcup_{\mathcal{C}}\rho_{\mathcal{C}}(f)$$ where $\mathcal{C}$ ranges over the maximal chains of basic pieces of $f$, given by Theorem \ref{t.rhorhochain}. This is a union of finitely many convex sets, but there may be chains $\mathcal{C}$ for which $\rho_{\mathcal{C}}(f)$ does not contain $0$. We shall show that the rotation sets of such chains are {\em radial} segments (i.e. contained in a $1$-dimensional subspace), and that $\rho(f)$ contains the continuation to $0$ of these segments.

\begin{lemma}\label{l.agujas}
Let $f\in \mathcal{A}_0(\Sigma)$ and assume $\Lambda_1\prec \Lambda_2$ are basic pieces such that there is no $\Lambda$ with $\Lambda_1\prec \Lambda \prec \Lambda_2$. Then one of the following holds:
\begin{itemize}
\item[a)] $\Lambda_1\cup\Lambda_2$ contains a zero-rotation periodic point of $f$, or
\item[b)] there is a closed curve $\alpha$ such that $\rho_{\Lambda_i}(f) \subset \langle [\alpha]\rangle $ for $i=1,2$.
\end{itemize}
\end{lemma}

We remark that $\alpha$ in statement (b) might be inessential, in which case $\rho_{\Lambda_i}(f) = \{0\}$.

\begin{proof}
Take $k\in\N$ so that the basic pieces of $f^k$ are topologically mixing. By Remark \ref{r.piecespower}, there exist $\Lambda_1'\subseteq\Lambda_1$ and $\Lambda_2'\subseteq\Lambda_2$ basic pieces of $f^k$ satisfying the same hypothesis on the heteroclinical relation. Notice that it is enough to prove the lemma for $f^k$, $\Lambda'_1$ and $\Lambda'_2$. Thus we shall assume that $\Lambda_1$ and $\Lambda_2$ are topologically mixing for $f$.


Recall that by Proposition \ref{p.propfaxa}, the set $\Lambda^u  =\textrm{cl}\left[\bigcup_{x\in\Lambda}W^u(x,f)\right]$ is an attractor for any basic piece $\Lambda$ of $f$. So we can find a subsurface $U$ containing $\Lambda_1^u$ such that
\begin{itemize}
\item $\textrm{cl}[f(U)]\subseteq U$
\item $\Omega(f)\cap U= \Lambda_1^u\cap\Omega(f)=\cup_{\Lambda\succeq\Lambda_1}\Lambda$
\end{itemize}

Moreover, we can assume that $U$ is W-filled and that all its connected components intersect $\Omega(f)$. 
Define $$\mathcal{F}=\{\Lambda:\Lambda_1\prec\Lambda,\text{ }\Lambda\neq\Lambda_2\}.$$  
The basic pieces of $f$ that are contained in $U$ are exactly $\Lambda_1$, $\Lambda_2$ and those in $\mathcal{F}$. On the other hand, the pieces $\Lambda\in\mathcal{F}$ do not precede $\Lambda_1$ nor $\Lambda_2$, so the closed set $$\bigcup_{\Lambda\in\mathcal{F}}\Lambda^u  $$ does not meet $\Lambda_1\cup\Lambda_2$. We shall construct a curve $K$ joining $\Lambda_1$ and $\Lambda_2$ within $U$ that does not meet $\cup_{\Lambda\in\mathcal{F}}\Lambda^u $.

Let $V_1$ and $V_2$ be neighbourhoods of $\Lambda_1$ and $\Lambda_2$ respectively, that are contained in $U$ and disjoint from $\cup_{\Lambda\in\mathcal{F}}\Lambda^u $.  Since $\Lambda_1\prec\Lambda_2$ there are $x\in\Lambda_1$ and $y\in \Lambda_2$ such that $W^u(x,f)$ meets $W^s(y,f)$ at some point $z$. Let $\beta$ be the arc of $W^u(x,f)$ from $x$ to $z$ and $\gamma$ the arc of $W^s(y,f)$ from $z$ to $y$. We can assume that $\gamma\subset V_2$, otherwise we apply a high enough power of $f$ and redefine $x$, $y$ and $z$. This yields that $K:=\beta\cup\gamma$ is contained in $U$, since $\beta\subset \Lambda_1^u \subset U$. Next we prove that $K$ is disjoint from $\cup_{\Lambda\in\mathcal{F}}\Lambda^u$. That is clear for $\gamma$, since it lies in $V_2$. For each $p\in \beta$ there is some neighbourhood $V_p$ and some $k>0$ with $f^{-k}(V_p)\subset V_1$.  Now if $\Lambda^u$ met $\beta$ for some $\Lambda\in\mathcal{F}$, we consider some point of intersection $p$ and see that $W^u(\Lambda,f)$ must meet $V_p$. Since $W^u(\Lambda,f)$ is $f^{-1}$ invariant, this would mean that $W^u(\Lambda,f)$ meets $V_1$, which contradicts the choice of $V_1$.



We apply Lemma \ref{l.construquasi} to the subsurface $U$, the basic pieces in $\mathcal{F}$, and the continuum $K$. This gives a quasi-invariant subsurface $S$ such that:
\begin{itemize}
\item $S\cap\Omega(f)=\Lambda_1\cup\Lambda_2$
\item $K\subset S$
\end{itemize}

Let $S_0$ be the connected component of $S$ that contains $K$, which is also quasi-invariant by Remark \ref{r.quasiconexas}. By construction of $K$ we get that both $\Lambda_1$ and $\Lambda_2$ meet $S_0$. Therefore Lemma \ref{l.qisurface} item (1) lets us deduce that $\Lambda_1\cup\Lambda_2$ is contained in $S_0$. If $S_0$ is contained in a disk, we can apply Lemma \ref{l.qisurface} point (3) to show that $\rho_{\Lambda_1}(f) = \rho_{\Lambda_2}(f)$ and this is a point $v\in \langle[\alpha]\rangle$ where $\alpha$ is some closed curve in $\Sigma$. Thus statement (b) holds.

If $S_0$ is contained in an essential annulus $A$ (but not in a disk) we take an essential simple closed curve $\alpha$ with $\langle[\alpha]\rangle = H_1(A) = H_1(S_0)$. From Lemma \ref{l.qisurface} point (4) we get statement (b) for this $\alpha$.

In the remaining case, $S_0$ is not contained in a disk or an annulus , so Lemma \ref{l.filling} point (1) gives statement (a).
\end{proof}

We obtained a little more from the proof of Lemma \ref{l.agujas}, that we state below as a remark.

\begin{rem} \label{r.agujas} If statement (b) in Lemma \ref{l.agujas} holds with a curve $\alpha$ that is not essential or not simple, then we can say more: $\rho_{\Lambda_1}(f)=\rho_{\Lambda_2}(f)=\{v\}$ for $v\in \langle [\alpha]\rangle$.

\end{rem}

\begin{lemma} \label{l.cadenas}
Let $f\in \mathcal{A}_0(\Sigma)$ and $\mathcal{C}$ be a maximal chain of basic pieces of $f$. Then exactly one of the following holds.
\begin{enumerate}
\item[A)] $0\in\rho_{\mathcal{C}}(f)$.
\item[B)] There is an essential closed curve $\alpha$ in $\Sigma$ with $\rho_{\mathcal{C}}(f)\subset \langle[\alpha]\rangle - \{0\}$.
\end{enumerate}

\end{lemma}

We observe that in case (B) the curve $\alpha$ must be not only essential, but non-trivial in homology.

\begin{proof}

Write $\mathcal{C} = \{\Lambda_1 \prec \cdots \prec \Lambda_k \}$. If some $\Lambda_i$ has $0\in\rho_{\Lambda_i}(f)$ then $\mathcal{C}$ satisfies (A). We assume that is not the case, in particular that no $\Lambda_i$ contains a zero-rotation periodic point of $f$. By the maximality of $\mathcal{C}$ we can apply Lemma \ref{l.agujas} to $\Lambda_i\prec\Lambda_{i+1}$ for $i=1,\ldots,k-1$, to conclude that there are closed curves $\alpha_i$ in $\Sigma$ such that $\rho_{\Lambda_i}(f)$ and $\rho_{\Lambda_{i+1}}(f)$ are contained in $\langle [\alpha_i]\rangle-\{0\}$. Under our assumptions it is clear that $[\alpha_i]\neq 0$, thus $\alpha_i$ are essential and $\langle [\alpha_i]\rangle$ are one dimensional. For $i=1,\ldots,k-2$ we have that $\left(\langle [\alpha_i]\rangle \cap \langle [\alpha_{i+1}]\rangle\right) -\{0\}$ contains $\rho_{\Lambda_{i+1}}(f)$, thus it is non empty. So the subspaces $\langle [\alpha_i]\rangle$ are all the same, and we can choose $\alpha_1=\cdots=\alpha_{k-1}$ and call this curve $\alpha$. Then we have  $\rho_{\mathcal{C}}(f)\subset \langle[\alpha]\rangle$. Depending on whether $\rho_{\mathcal{C}}(f)$ contains $0$ or not we are in case (A) or (B).

\end{proof}

\begin{rem} If $\mathcal{C}$ is in case (B) of Lemma \ref{l.cadenas} with a curve $\alpha$ that is not simple, then $\rho_{\mathcal{C}}(f)$ is a single point. This follows from the proof of Lemma \ref{l.cadenas} and Remark \ref{r.agujas}.
\end{rem}

\begin{theorem} \label{t.starshape}
If $f\in\mathcal{A}_0(\Sigma_g)$ then $\rho(f)$ is star-shaped about $0$.
\end{theorem}


\begin{proof}
By Theorem \ref{t.rhorhochain} we write $\rho(f) = \bigcup_{\mathcal{C}}\rho_{\mathcal{C}}(f)$ where $\mathcal{C}$ ranges over the maximal chains of basic pieces of $f$. We take an arbitrary $v\in\rho(f)$, aiming to prove that $[0,v]\subset \rho(f)$, and let $\mathcal{C}$ be a maximal chain so that $v\in\rho_{\mathcal{C}}(f)$. By Lefschetz's theorem $f$ has a zero-rotation fixed point, so we can find a maximal chain $\mathcal{C}'$ that has a basic piece containing said zero-rotation fixed point. In particular $\rho_{\mathcal{C}'}(f)$ is a convex set containing $0$.

Recall that the equivalence relation generated by $\prec$ on the set of basic pieces of $f$ has only one class ($\mathcal{G}_f$ is connected). So we can find maximal chains $\mathcal{C}=\mathcal{C}_1,\ldots,\mathcal{C}_k=\mathcal{C}'$ so that $\mathcal{C}_i$ intersects $\mathcal{C}_{i+1}$ in at least one basic piece for all $i=1,\ldots,k-1$. We may assume that for $i<k$ the $\mathcal{C}_i$ are in case (B) of Lemma \ref{l.cadenas}, otherwise we redefine $\mathcal{C}'$ to be the first $\mathcal{C}_i$ in case (A).

If $k=1$ then $\rho_{\mathcal{C}}(f)$ is a convex set containing $0$ and $v$, so we are done. Otherwise, for $i<k$ we have that $\rho_{\mathcal{C}_i}(f)\subset\langle[\alpha_i]\rangle-\{0\}$ for some essential closed curve $\alpha_i$. Since $\mathcal{C}_i$ and $\mathcal{C}_{i+1}$ have a basic piece in common, $\rho_{\mathcal{C}_i}(f)$ intersects $\rho_{\mathcal{C}_{i+1}}(f)$ for all $i=1,\ldots,k-1$. This implies that the $\langle[\alpha_i]\rangle$ for $i<k$ are all the same subspace of dimension $1$. The sets $\rho_{\mathcal{C}_i}(f)$ for $i<k$ are segments in the line $\langle[\alpha_1]\rangle$, each intersecting the next, thus $I=\bigcup_{i=1}^{k-1} \rho_{\mathcal{C}_i}(f)$ is a segment in $\langle[\alpha_1]\rangle$. Now $\rho_{\mathcal{C}_k}(f)$ is a convex set that contains $0$ and intersects $I$, thus $\bigcup_{i=1}^{k} \rho_{\mathcal{C}_i}(f) = I\cup \rho_{\mathcal{C}_k}(f) $ contains $\text{conv}(I\cup\{0\})\subset \langle[\alpha_1]\rangle$. In particular $\bigcup_{i=1}^{k} \rho_{\mathcal{C}_i}(f)$ contains $[0,v]$.
\end{proof}

\section{Bounding the complexity of the rotation set: trivial pieces}\label{s.triviales}

Given $f\in \mathcal{A}_0(\Sigma)$ we consider a Conley decomposition of $\Sigma$ as in section \ref{s.conley}, and for each basic piece $\Lambda$ of $f$ we denote by $S_\Lambda$ the Conley surface associated to $\Lambda$ in this decomposition. Recall that a basic piece $\Lambda$ is called trivial if every component of $S_{\Lambda}$ is trivial. Equivalently, if $\text{Fill}(S_\Lambda)$ is a disjoint union of finitely many disks.


The aim of this section is to prove the following result, that is the first step towards achieving a topological bound for the number of convex sets needed to express $\rho(f)$ for $f\in\mathcal{A}_0(\Sigma)$.

\begin{proposition} \label{sacar.triviales} Let $\mathcal{C}$ be a maximal chain for $f\in\mathcal{A}_0(\Sigma)$. Then there is a chain $\mathcal{C}^*$ that contains no trivial pieces such that $\rho_{\mathcal{C}}(f)\subset \rho_{\mathcal{C}^*}(f)$.
\end{proposition}

Proposition \ref{sacar.triviales} will allow us to write $$\rho(f) = \bigcup_{\mathcal{C}}\rho_{\mathcal{C}}(f)$$ with $\mathcal{C}$ ranging over the chains of basic pieces of $f$ that contain no trivial pieces. Of course, we can restrict this union by asking for $\mathcal{C}$ to be maximal among the chains with no trivial pieces, as a way of avoiding the obvious repetition. In short, this says we can forget about trivial pieces when computing $\rho(f)$.

We recall that if $\Lambda$ is a trivial piece, then $\mathcal{D} = \text{Fill}(S_\Lambda) $ is a disjoint union of disks satisfying that $f(\mathcal{D})$  either contains or is contained in $\mathcal{D}$. Moreover, if we write $\mathcal{D}=D_1\cup\cdots\cup D_k$ as a disjoint union of disks, then $f(D_j)$ either contains or is contained in $D_{\sigma(j)}$, for $\sigma$ a cyclic permutation of $\{1,\ldots,k\}$.

\begin{definition} Let $\mathcal{C}$ be a chain of basic pieces of $f\in\mathcal{A}_0(\Sigma)$. We say that $\mathcal{C}$ is {\em completely trivial} if $$\mathcal{D}:=\text{Fill}\left(\bigcup_{\Lambda\in\mathcal{C}}S_\Lambda\right)$$ is a disjoint union of finitely many disks, so that $f(\mathcal{D})$ either contains or is contained in $\mathcal{D}$.
\end{definition}
If $\mathcal{C}$ is a completely trivial chain, then
every basic piece of $\mathcal{C}$ meets every connected component of $\mathcal{D}$:
Otherwise, we can decompose $\mathcal{D}$ as a union of two disjoint non-empty open sets $U$, $V$ that are both either forward or backward invariant (according to whether $\mathcal{D}$ is forward or backward invariant). Then each basic piece in $\mathcal{C}$ is contained in either $U$ or $V$. On the other hand, by definition every connected component of $\mathcal{D}$ meets some basic piece of $\mathcal{C}$. But this is absurd, for basic pieces contained in $U$ cannot precede basic pieces in $V$ and the other way around, which contradicts the existence of the chain $\mathcal{C}$.

Thus Lemma \ref{l.qisurface} applied to a suitable power of $f$ gives that
$\rho_{\mathcal{C}}(f)$ is a single point whenever $\mathcal{C}$ is completely trivial.
This fact will be relevant in what follows. 

\begin{lemma} \label{l.ctrivial} Let $\mathcal{C}$ be a completely trivial chain of $f\in\mathcal{A}_0(\Sigma)$. Then there exists $\Lambda$ a non-trivial basic piece with $\rho_{\mathcal{C}}(f)\subset \rho_{\Lambda}(f)$.

\end{lemma}

\begin{proof} Let $v\in H_1(\Sigma)$ be the point with $\rho_{\mathcal{C}}(f)=\{v\}$. Denote $\mathcal{D}=\text{Fill}(\bigcup_{\Lambda\in\mathcal{C}}S_\Lambda)$, which by hypothesis is a disjoint union of disks. Recalling Remark \ref{r.discos} for all the non-trivial surfaces in the Conley decomposition, we see that there exists some non-trivial Conley surface $S$ so that $\mathcal{D}\subseteq \text{cl}[\text{Fill}(S)\setminus S]$. Let $\Lambda$ be the (non-trivial) basic piece associated to $S$. By item (1) of Lemma \ref{l.conleyesquasi} there is  $k\in\N$ so that $S$ is quasi-invariant under $f^k$. Let $D_0$ be a connected component of $\text{cl}[\text{Fill}(S)\setminus S]$ that meets $\mathcal{D}$.
Applying Lemma \ref{l.filling} we get $x_0\in S$ a fixed point of $f^k$ so that $\rho(f^k,x_0)$ equals $\rho_{\Lambda'}(f^k)$ for every basic piece $\Lambda'$ of $f^k$ that is contained in $D_0$. Observe that for every such piece $\Lambda'$ we have $\rho_{\Lambda'}(f^k)=\{k\cdot v\}$, thus $\rho(f^k,x_0)=k\cdot v$, and also note that $x_0\in S\cap\Omega(f)=\Lambda$. Therefore $x_0\in \Lambda$ is a periodic point of $f$ whose rotation set equals $\rho_\mathcal{C}(f)$, so $\Lambda$ is the basic piece that satisfies the statement.


\end{proof}

Clearly Lemma \ref{l.ctrivial} proves the simplest case of Proposition \ref{sacar.triviales}, that is when the maximal chain $\mathcal{C}$ is completely trivial. Next we shall study the structure of a chain in terms of the trivial pieces it contains.

\begin{lemma} \label{l.trivialsegment} Let $\mathcal{C} =\{ \Lambda_0\prec \cdots \prec \Lambda_k \}$ be a chain of $f\in\mathcal{A}_0(\Sigma)$. If $\Lambda_j$ is a trivial piece then either the initial segment $\Lambda_0\prec \cdots \prec \Lambda_j$ or the final segment $\Lambda_j\prec \cdots \prec \Lambda_k$ is completely trivial (or both).
\end{lemma}

\begin{proof}
Let $\mathcal{D}=\text{Fill}(S_{\Lambda_j})$, which is a union of disks since $\Lambda_j$ is a trivial piece. If $f(\mathcal{D})\subset \mathcal{D}$ then $\Lambda_j^u$ is contained in $\mathcal{D}$. Recalling that $\Lambda_j^u$ contains every piece that follows $\Lambda_j$, we obtain that the final segment of $\mathcal{C}$ starting at $\Lambda_j$ is completely trivial. On the other hand, if $\mathcal{D}\subset f(\mathcal{D})$ we can give the same argument with $\Lambda_j^s$, and obtain that the initial segment is completely trivial.

\end{proof}

Lemma \ref{l.trivialsegment} allows us to write any chain $\mathcal{C}$ of $f\in\mathcal{A}_0(\Sigma)$ as a union of three segments, some of them possibly empty, in the following way: First we have an initial segment $\mathcal{C}_0$ that is completely trivial, all of its pieces contained in a disjoint union of disks $\mathcal{D}_0$ with $f(\mathcal{D}_0) \supset \mathcal{D}_0$. Next we have the middle segment $\mathcal{C}_{NT}$ containing the non-trivial pieces, and finally the completely trivial segment $\mathcal{C}_1$ with its pieces contained in a disjoint union of disks $\mathcal{D}_1$ such that $f(\mathcal{D}_1)\subset \mathcal{D}_1$.

When $\mathcal{C}$ contains non-trivial pieces this decomposition is unique. On the other hand, $\mathcal{C}$ is completely trivial if it can be written as either $\mathcal{C}=\mathcal{C}_0$ or $\mathcal{C}=\mathcal{C}_1$. We remark that a chain might contain only trivial pieces without being completely trivial. The next step towards Proposition \ref{sacar.triviales} is to exclude that case.

\begin{lemma} \label{l.nctrivial} Let $\mathcal{C}$ be a chain of $f\in\mathcal{A}_0(\Sigma)$. Then there is a chain $\mathcal{C}'$ of $f$ that contains a non-trivial piece and such that $\rho_{\mathcal{C}}(f)\subset \rho_{\mathcal{C}'}(f)$.

\end{lemma}

\begin{proof} By Lemma \ref{l.ctrivial} the only case we need to consider is when $\mathcal{C}$ has trivial pieces only, but is not completely trivial. We take a decomposition $\mathcal{C}=\mathcal{C}_0\cup\mathcal{C}_1$ as in the discussion above. Since $\mathcal{C}_0$ and $\mathcal{C}_1$ are completely trivial chains, there must exist $\mathcal{D}_0$ and $\mathcal{D}_1$ disjoint unions of disks such that:
\begin{itemize}
\item $\mathcal{D}_j$ contains all the basic pieces of $\mathcal{C}_j$ for $j=0,1$,
\item There exist non-trivial Conley surfaces $S_j$ such that $\mathcal{D}_j$ consists of a disjoint union of connected components of $\text{cl}[\text{Fill}(S_j)\setminus S_j]$ for $j=0,1$,
\item $f(\mathcal{D}_1)\subseteq \mathcal{D}_1$ and $f(\mathcal{D}_0)\supseteq \mathcal{D}_0$.
\end{itemize}
Notice that by this construction, the trivial basic pieces that are not contained in $\mathcal{D}_0\cup \mathcal{D}_1$ have associated Conley surfaces that are disjoint from $\mathcal{D}_0\cup \mathcal{D}_1$. This is the key for the following claim:

\vspace{.1cm}
{\bf Claim:} { If we have a chain $\Lambda_0\prec\Lambda\prec\Lambda_1$ with $\Lambda_j\subset \mathcal{D}_j$ for $j=0,1$ and $\Lambda$ disjoint from $\mathcal{D}_0\cup \mathcal{D}_1$, then $\Lambda$ must be non-trivial.}

\vspace{.1cm}
To prove the claim by contradiction, assume that $\Lambda$ is trivial. Then it is contained in a disjoint union of disks $\mathcal{D}$ that is disjoint from $\mathcal{D}_0\cup \mathcal{D}_1$. But this is absurd: If $\mathcal{D}$ is attracting, every piece that follows $\Lambda$ is contained in $\mathcal{D}$ and we cannot have $\Lambda\prec\Lambda_1\subset \mathcal{D}_1$. If $\mathcal{D}$ is repelling, we argue similarly with $\Lambda_0$.

\vspace{.1cm}
Taking a suitable power of $f$ and applying Lemma \ref{l.qisurface} point (3) to it, we see that there exist $v_j\in H_1(\Sigma)$ such that $\rho_{\Lambda_j}(f)=\{v_j\}$ for every piece $\Lambda_j\subset \mathcal{D}_j$ and $j=0,1$. In particular $\rho_{\mathcal{C}}(f)=\text{conv}(v_0,v_1)$. Together with the Claim, this fact means we can prove the lemma by finding a chain of the form $$\mathcal{C}'=\{\Lambda_0\prec\Lambda\prec\Lambda_1\}$$ where $\Lambda_j\subset \mathcal{D}_j$ for $j=0,1$, and $\Lambda$ is disjoint from $\mathcal{D}_0\cup \mathcal{D}_1$.

Define $U_0 := \mathcal{D}_1\cup\bigcup_{\Lambda_0\subset \mathcal{D}_0} \Lambda_0^u$. Since it is an attractor, taking a neighbourhood of it we an find a subsurface $U$ such that: \begin{itemize}
\item $\text{cl}[f(U)]\subseteq\text{int}(U)$
\item $U\cap \Omega(f)=U_0\cap\Omega(f)$, which is the union of the following set of basic pieces: $$\{\Lambda : \Lambda\subset\mathcal{D}_1 \}\cup\{\Lambda:
\Lambda\succeq\Lambda_0\mbox{ for some }\Lambda_0\subset\mathcal{D}_0\}$$
\end{itemize}
Moreover, we can assume that every connected component of $U$ meets $\Omega(f)$ and that $U$ is W-filled. Notice that this surface contains $\mathcal{D}_0$, since every $x\in \mathcal{D}_0$ has $\alpha$-limit in $\mathcal{D}_0$ and therefore $x$ must belong to $\Lambda_0^u$ for some basic piece $\Lambda_0\subset \mathcal{D}_0$. The existence of the chain $\mathcal{C}$ implies that we can find a basic piece $\Lambda'\subseteq \mathcal{D}_0$, an arc $\gamma$ contained in $W^u(\Lambda',f)$ and disks $D'_0$, $D'_1$ which are connected components of $\mathcal{D}_0$ and $\mathcal{D}_1$ respectively, so that $\gamma$ joins $D'_0$ with $D'_1$. Denote by $K$ the union of $\gamma$ and these disks, i.e. $K:=\gamma\cup D'_0\cup D'_1$, and notice that $K$ is contained in $U$.

Let $\mathcal{F}$ be the family of basic pieces $\Lambda$ that are not contained in $\mathcal{D}_0$, that follow some piece in $\mathcal{D}_0$, but precede no piece in $\mathcal{D}_1$. In other words, recalling the definition of $U$ we can write $$\mathcal{F} = \{\Lambda: \Lambda\subset U\setminus\mathcal{D}_0,  \Lambda^u\cap \mathcal{D}_1 = \emptyset \}.$$ If $\Lambda\in\mathcal{F}$ then $\Lambda^u$ must be disjoint from $\mathcal{D}_0\cup\mathcal{D}_1$, and arguing as in Lemma \ref{l.agujas} we can show that $\Lambda^u$ is disjoint from $\gamma$ as well. Therefore $K$ does not meet $\bigcup_{\Lambda\in \mathcal{F}}\Lambda^u$.


Consider $k\in\N$ so that the basic pieces of $f^k$ are topologically mixing, and recall Remark \ref{r.piecespower}, stating that each basic piece of $f$ decomposes as a finite and disjoint union of basic pieces of $f^k$. We are in conditions to apply Lemma \ref{l.construquasi} to $f^k$, with the subsurface $U$, the family of basic pieces of $f^k$ that arises from $\mathcal{F}$, and the continuum $K$.  Thus we obtain a subsurface $S$ with the following properties:
\begin{itemize}
\item $S$ is quasi-invariant under $f^k$,
\item $S\cap\Omega(f^k)=(\Omega(f)\cap U)\setminus \bigcup_{\Lambda\in\mathcal{F}}\Lambda^u$
\item $K\subset S$
\end{itemize}

Denote by $S'$ the connected component of $S$ containing $K$, which is quasi-invariant under $f^k$ by Remark \ref{r.quasiconexas}. We consider  $\mathcal{D}''_0:=f^{-k}(\mathcal{D}_0)\subset \textrm{int}(\mathcal{D}_0)$ and $\mathcal{D}''_1:=f^{k}(\mathcal{D}_1)\subset \textrm{int}(\mathcal{D}_1)$, whose components are either contained in $\textrm{int}(S')$ or disjoint from $S'$. This allows us to ensure that
 $$S'' = \textrm{cl}[S'-(\mathcal{D}''_0\cup \mathcal{D}''_1)]$$  is a subsurface, that is quasi-invariant for $f^k$. Also notice that $D''_0:=f^{-k}(D'_0)$ and $D''_1:=f^{k}(D'_1)$ are components of $\mathcal{D}''_0$ and $\mathcal{D}''_1$ that do meet $S'$, thus give $S''$ two boundary components. Then $S''$ cannot be an annulus, otherwise $\Sigma$ would be $S^2$, so $f^k$ must  have some fixed point in $S''$ by Remark \ref{r.filling}. This fixed point of $f^k$ must belong to some basic piece $\Lambda$ of $f$. We see  that $\Lambda\subset S$, since $\Omega(f)\cap S$ is $f$-invariant, and that $\Lambda$ must be disjoint from $\mathcal{D}_0\cup\mathcal{D}_1$, for otherwise it would be contained in $\mathcal{D}''_0\cup\mathcal{D}''_1$ and could not meet $S''$.

This gives the chain we were looking for: Since $\Lambda\subset S\setminus(\mathcal{D}_0\cup \mathcal{D}_1)$ we get, on one hand, that $\Lambda\subset U\setminus\mathcal{D}_1$ so there is a basic piece $\Lambda_0\subset\mathcal{D}_0$ with $\Lambda_0\prec\Lambda$. On the other hand, $\Lambda\notin\mathcal{F}$ while following $\Lambda_0\subset\mathcal{D}_0$, so there must be a basic piece $\Lambda_1\subset\mathcal{D}_1$ with $\Lambda\prec\Lambda_1$. Then the chain $\mathcal{C}'= \{\Lambda_0\prec\Lambda\prec\Lambda_1\} $ verifies the lemma, recalling the Claim and the fact that $\Lambda$ is disjoint from $\mathcal{D}_0\cup \mathcal{D}_1$.




\end{proof}

We are ready to provide a proof or Proposition \ref{sacar.triviales}. The proof will be
supported on the previous lemmas, using some of the constructions introduced
in their proofs.

\smallskip

{\bf Proof of Proposition \ref{sacar.triviales}:}

Let $\mathcal{C}$ be a chain for $f\in\mathcal{A}_0(\Sigma)$. Write $\mathcal{C} = \mathcal{C}_0\cup \mathcal{C}_{NT} \cup \mathcal{C}_1$ where $\mathcal{C}_{NT}$ consists of all the non-trivial pieces of $\mathcal{C}$. By Lemma \ref{l.nctrivial} we can assume $\mathcal{C}_{NT}$ is not empty, and write $\mathcal{C}_{NT} = \{ \Lambda_1\prec\cdots\prec\Lambda_n \}$. As we observed before, Lemma \ref{l.qisurface} gives that for $j=0,1$ there are $v_j\in H_1(\Sigma)$ such that $\rho_{\mathcal{C}_j}(f)=\{v_j\}$. We finish the proof by the following claim.

\vspace{.1cm}
{\bf Claim:} {\em There are non-trivial pieces $\Lambda'_j$, for $j=0,1$, such that $\Lambda'_0\preceq \Lambda_1$, $\Lambda_n\preceq \Lambda'_1$, and $v_j\in\rho_{\Lambda'_j}(f)$.}
\vspace{.1cm}

This claim completes the proof, taking $\mathcal{C}^*=\{\Lambda'_0\preceq \Lambda_1\prec\cdots\prec\Lambda_n\preceq \Lambda'_1 \}$. What follows is the proof of the claim.

\smallskip

We prove the existence of $\Lambda'_0$, the argument for $\Lambda'_1$ being analogous. As in the proof of Lemma \ref{l.nctrivial}, we can find $\mathcal{D}_0$ a disjoint union of disks so that:
\begin{itemize}
\item $\mathcal{D}_0\subseteq f(\mathcal{D}_0)$,
\item $\mathcal{D}_0$ contains all the basic pieces of $\mathcal{C}_0$
\item $\mathcal{D}_0$ is a union of connected components of $\text{cl}[\text{Fill}(S_0)\setminus S_0]$ for some non-trivial Conley surface $S_0$.
\end{itemize}

Therefore, every trivial piece not contained in $\mathcal{D}_0$ is contained in some attracting or repelling union of disks that is disjoint from $\mathcal{D}_0$.

Let $\mathcal{F}$ be the family of basic pieces $\Lambda$ that satisfy the following:
\begin{itemize}
\item $\Lambda\preceq\Lambda_1$,
\item $\Lambda$ is not contained in $\mathcal{D}_0$,
\item there is some basic piece $\Lambda_0\subset \mathcal{D}_0$ with $\Lambda_0\prec\Lambda$.
\end{itemize}

Notice that $\Lambda_1\in\mathcal{F}$. We also see that the pieces in $\mathcal{F}$ are non-trivial: Take a trivial piece $\Lambda$ that is not contained in $\mathcal{D}_0$. Then it is contained in a quasi-invariant union of disks $\mathcal{D}$ that is disjoint from $\mathcal{D}_0$. If $\mathcal{D}$ is repelling then there is no $\Lambda_0\subset \mathcal{D}_0$ that precedes $\Lambda$. If $\mathcal{D}$ is attracting, then $\Lambda_1$ cannot follow $\Lambda$.

Let $\Lambda'$ be a minimal element of $\mathcal{F}$ with respect to the partial order $\prec$. We will show that $v_0\in\rho_{\Lambda'}(f)$, thus $\Lambda'$ works as the basic piece $\Lambda'_0$ in the statement of this claim.

Define $R=\cup_{\Lambda\subseteq \mathcal{D}_0}\Lambda^u$. Since $R$ is an attractor we can find a subsurface $U$ containing $R$ so that:
\begin{itemize}
\item $U\cap\Omega(f)=R\cap\Omega(f)$
\item $\textrm{cl}[f(U)]\subseteq\text{int}(U)$
\end{itemize}
We can also assume that every connected component of $U$ meets $\Omega(f)$. Notice that $U$ contains $\mathcal{D}_0$ by definition of $R$ and the fact that $\mathcal{D}_0$ is a repellor (we gave the explicit argument in the proof of Lemma \ref{l.nctrivial}). 

Recall that $\Lambda'$ is the minimal element of $\mathcal{F}$, so there is some basic piece $\Lambda_0 \subset \mathcal{D}_0$ with $\Lambda_0\prec\Lambda'$. Then there is an arc $\gamma$ joining $\Lambda_0$ and $\Lambda'$ that is a concatenation of $\gamma^u\subset W^u(\Lambda_0,f)$ and $\gamma^s\subset W^s(\Lambda',f)$. Moreover, we can assume $\gamma^s$ is within a small neighbourhood of $\Lambda'$, which gives us that $\gamma\subset U$, since $\gamma^u\subset \Lambda_0^u\subset U$ and $\Lambda'\subset \textrm{int}(U)$. Let $D'_0$ be the component of $\mathcal{D}_0$ with the starting point of $\gamma$, and let $K:=D'_0\cup\gamma$, observing that it is contained in $U$.

Let $\mathcal{U}$ be the family of basic pieces $\Lambda$ of $f$ that are contained in $U
\setminus \mathcal{D}_0$ but are different from $\Lambda'$, i.e.: $$\mathcal{U} = \{\Lambda \subset U\setminus \mathcal{D}_0:\Lambda\neq\Lambda'\}. $$ Notice that the pieces of $\mathcal{U}$ cannot precede $\Lambda'$, for if there was $\Lambda\in \mathcal{U}$ with $\Lambda\prec\Lambda'$ then $\Lambda'$ would not be minimal in $\mathcal{F}$. Thus

\begin{equation}\label{eqloca}(U\cap\Omega(f))\setminus \left(\bigcup_{\Lambda\in\mathcal{U}}\Lambda^u\right)=\Lambda'\cup \left( \mathcal{D}_0\cap\Omega(f) \right) \end{equation}

Using equation \ref{eqloca} we can argue as in Lemmas \ref{l.agujas} and \ref{l.nctrivial} to show that $K$ is disjoint from $\cup_{\Lambda\in\mathcal{U}}\Lambda^u$. Take $k\in\N $ so that the basic pieces of $f^k$ are topologically mixing, and apply Lemma \ref{l.construquasi} to $f^k$, the subsurface $U$, the family of basic pieces of $f^k$ that arises from $\mathcal{U}$, and the continuum $K$. This gives a surface $S$ that is quasi-invariant under $f^k$ and satisfies the following:





\begin{itemize}
\item $S\cap\Omega(f^k)=\Lambda'\cup (\mathcal{D}_0\cap\Omega(f))$, by equation \ref{eqloca},
\item $K\subset S$
\end{itemize}

Let $S'$ be the connected component of $S$ containing $K$, which is quasi-invariant under $f^k$ by Remark \ref{r.quasiconexas}. We take $\mathcal{D}''_0:=f^{-k}(\mathcal{D}_0)$, so we have that the components of $\mathcal{D}''_0$ that meet $S'$ are contained in $\text{int}(S')$, and $$S''=S'\setminus \mathcal{D}''_0 $$ is a surface that is quasi-invariant under $f^k$. This construction gives that $\Omega(f^k)\cap S''\subset \Lambda'$. We see that this set is non empty: The endpoint of $\gamma^u$ is in $K\setminus \mathcal{D}''_0$, so it belongs to $S''\cap \Lambda'$. Thus there exists $\Lambda''\subset\Lambda'$ a basic piece of $f^k$ that is contained in $S''$. On the other hand, notice that $D''_0 := f^{-k}(D'_0)$ is a component of $\text{Fill}(S'')\setminus S''$, that it is a repellor for $f^k$ and that every basic piece of $f^k$ in $D''_0$ has rotation vector $k\cdot v_0$.


If $S''$ is trivial, Lemma \ref{l.qisurface} implies that every basic piece in $\text{Fill}(S'')$ has the same rotation vector. In particular, since $D_0''\subseteq \text{Fill}(S'')$, we have that $\rho_{\Lambda''}(f^k)=k\cdot v_0$ and therefore $v_0\in\rho_{\Lambda'}(f)$ as desired.

In case that $S''$ is non-trivial we use Lemma \ref{l.filling} point (2) to deduce that there exists $p$ a fixed point of $f^k$ in $S''$ with rotation $k\cdot v_0$. Since $p\in \Omega(f)\cap S'' \subset \Lambda'$, this again shows that $v_0\in\rho_{\Lambda'}(f)$.

\section{Bounding the complexity of the rotation set: Annular pieces}\label{s.annular}

In this section we provide the final results which allow us to prove the main theorem \ref{mainthm}.
Recall that from sections \ref{s.rotsetbp} and \ref{s.agujas}, for $f\in\mathcal{A}_0(\Sigma)$ we have that
$$\rho(f)=\bigcup_{\mathcal{C}}\rho_{\mathcal{C}}(f) $$ where $\mathcal{C}$ ranges over all maximal chains of $\mathcal{G}_f$ containing no trivial basic pieces. Our goal is to rewrite this union so that it involves at most $\textrm{c(g)}$ convex sets containing zero, where $c(g)$ is some constant depending only on the genus $g$ of $\Sigma$. If we consider a Conley decomposition for $f$ as in section \ref{s.conley}, we see that there are at most $2g-2$ curved pieces (pair of pants decomposition), but there is no bound for the number of annular pieces. Unlike the case for trivial pieces, these annular pieces do participate in the rotation set. Of course, it will be crucial that there are at most $3g-3$ homotopy classes of annuli in the Conley decomposition, but this fact alone does not give the desired bound, for we must consider the heteroclinic relations that may exist among such annular pieces, giving rise to chains in $\mathcal{G}_f$. This section is thus devoted to control the rotation sets of the chains involving annular pieces, and the study of the heteroclinic relation among such pieces will be at the center of this problem.


Throughout this section we consider some fixed Conley decomposition of $f\in\mathcal{A}_0(\Sigma)$. If $\Lambda$ is a basic piece of $f$, let $S_{\Lambda}$ stand for the surface of our Conley decomposition that contains $\Lambda$, which we shall call {\em the Conley surface of} $\Lambda$. Let us recall the notations from section \ref{s.conley}, in particular the Poincar\'{e} cover $\pi:\D\to\Sigma$ and $F$ the canonical lift of $f$ to said cover. If $S\subset\Sigma$ is a connected subsurface, a {\em lift} $\tilde{S}\subset\D$ of $S$ is a connected component of $\pi^{-1}(S)$. In that case, $\textrm{Fill}(\tilde{S})$ is the union of $\tilde{S}$ with the components of the complement that are bounded disks (in the hyperbolic metric of $\D$). Clearly $\textrm{Fill}(\tilde{S})$ is a lift of $\textrm{Fill}(S)$.


\smallskip


\smallskip
\subsection{Coarse geometry of sub-surface lifts}

Here we prove a technical lemma, which can be interpreted in terms of the coarse geometry of a lift $\tilde{S}$ of a sub-surface $S\subset \Sigma$ near some component $\tilde{b}$ of $\partial\tilde{S}$. In the sequel we will apply it to the dynamics of $F$. This result is very specific, so we include a proof here, but in general a lot is known about the coarse geometry of hyperbolic surfaces, see for instance \cite{bridson}.

Given a metric space $(E,d)$ and a real number $R>0$, we say that a sequence $X=\{x_0,\ldots,x_k\}\subseteq E$ is a an $R$-{\em path} if $d(x_i,x_{i+1})<R$ for $i=0,\ldots,k-1$. This can be interpreted as a coarse version of a continuous path, where $R$ is the scale of the error.
This notion will be of interest because the sequences of iterates $x_k=F^k(x)$ give $R$-paths for some uniform $R>0$, since $F$ is the canonical lift.

We are ready to state the mentioned technical lemma.

\begin{lem}\label{lema geometrico} Let $S\subseteq \Sigma$ be a sub-surface with an essential boundary component $b$, and $R$ be a positive real number. Consider:
\begin{itemize}
\item $\tilde{S}\subseteq\D$ a lift of $S$, and $\tilde{b}\subseteq \partial\tilde{S}$ a lift of $b$.
\item $\xi\in\partial\D$ one of the endpoints of $b$.
\item $T\in\pi_1(\Sigma)$ a deck transformation that generates $\textrm{Stab}(\tilde{b})$.
\end{itemize}

Then there exists a compact set $\mathcal{K}\subseteq \D$ with the following property: For any $n\in\N$  there exists $W_n$ a neighbourhood of $\xi\in \textrm{cl}[\D]$ and integers $m_1,\ldots,m_n$ such that
\begin{itemize}
\item $T^{m_i}(\mathcal{K})\cap T^{m_j}(\mathcal{K})=\emptyset$ for $ i \neq j$,
\item for every $R$-path $X=\{x_0,\ldots,x_m\}\subseteq \tilde{S}$ with $x_0\in\mathcal{K}$ and $x_m\in W_n$ we have  $T^{m_i}(\mathcal{K})\cap X\neq\emptyset$ for every $i=1,\ldots,n$
\end{itemize}

Moreover, given a compact set $D\subseteq \tilde{S}$ we can suppose that $D\subseteq \mathcal{K}$.
\end{lem}

Dropping some precision, this property basically says that any infinite $R$-path that converges to $\xi$ must meet infinitely many disjoint translates of $\mathcal{K}$ under $T$. Also, that this is a property that holds when $\mathcal{K}$ is sufficiently large. The compactness of $\mathcal{K}$ says that every such $R$-path must pass uniformly close to $\tilde{b}$ infinitely often as it converges to $\xi$.

We now turn to the constructions at play in the proof of Lemma \ref{lema geometrico}. The key idea is standard in coarse geometry: to prove an analogous property for an infinite tree, and then promote it to $\tilde{S}$ via {\em quasi-isometry}. We will assume throughout the proof that $S$ is filled, as that causes no loss of generality.

Since $S$ is a surface of finite type with boundary we have $\pi_1(S)\cong \mathbb{F}_l$, the free group on $l$ generators, for some $l\in\N$. Let $\mathcal{T}$ be the Cayley graph of $\mathbb{F}_l$ with respect to the standard generating set, so $\mathcal{T}$ is the regular tree of valence $2l$ on each vertex. We consider the standard graph metric on $\mathcal{T}$, that assigns length $1$ to every edge. For $g\in\mathbb{F}_l$ let $L_g$ be the action of $g$ on $\mathcal{T}$ by left multiplication, which preserves the graph structure and thus is an isometry.

We assumed $S$ is filled, so $\tilde{S}$ is a universal cover for $S$, and $\mathbb{F}_l\cong\pi_1(S)$ injects into $\pi_1(\Sigma)$. For $g\in\mathbb{F}_l$ let $T_g$ be the action of $g$ on $\tilde{S}$ by the corresponding deck transformation, which is an isometry of the Poincar\'{e} metric $d_{\D}$.

This gives rise to an {\em equivariant quasi-isometry} between $\mathcal{T}$ and $\tilde{S}$, which is a map $\varphi:\mathcal{T}\to\tilde{S}$ satisfying the following:

\begin{itemize}
\item $\varphi$ is a {\em quasi-isometrc embedding}: There are constants $C\geq 0$ and $\lambda\geq 1$ such that for all $x,y\in\mathcal{T}$ we have $$ \lambda^{-1}d_{\mathcal{T}}(x,y) - C< d_{\D}(\varphi(x),\varphi(y))< \lambda d_{\mathcal{T}}(x,y) + C$$
\item The image of $\varphi$ is coarsely dense: There is $C_0>0$ such that for every $x\in\tilde{S}$ there is some $y\in\mathcal{T}$ with $d_{\D}(x,\varphi(y))<C_0$.
\item $\varphi$ is $\mathbb{F}_l$-equivariant:  For every $g\in\mathbb{F}_l$ we have $T_g\circ\varphi=\varphi\circ L_g$.
\end{itemize}

Such a map can be obtained as follows: Let $\mathcal{R}_l$ be the wedge of $l$ circles and $i:\mathcal{R}_l\to S$ an embedding such that $S$ retracts by deformation onto $i(\mathcal{R}_l)$. Then $\mathcal{T}$ is a universal cover for $\mathcal{R}_l$ and we can take $\varphi:\mathcal{T}\to\tilde{S}$  as some lift of $i$. Though $\varphi$ is not bijective, the notion of quasi-isometry is symmetrical: there also exists an equivariant quasi-isometry $\psi:\tilde{S}\to\mathcal{T}$. In this case we can obtain such a map from the retraction $r:S\to i(\mathcal{R}_l)$, by letting $\psi$ be a lift of $i^{-1}\circ r:S\to\mathcal{R}_l$. Choosing the lifts appropriately we can assume $\psi\circ\varphi = Id_{\mathcal{T}}$. On the other hand, $\varphi\circ\psi = \tilde{r}$, a lift of the retraction $r$, which is at a uniformly bounded distance from $Id_{\tilde{S}}$.

It is a basic result \cite{bridson} that $\varphi$ induces an homeomorphism $\varphi_{\infty}:\partial\mathcal{T}\to\partial\tilde{S}$, where $\partial\mathcal{T}$ is the natural Cantor set boundary of $\mathcal{T}$ (this is Gromov's visual boundary, which agrees with the topological ends in the case of a tree). We also have that $\psi_{\infty} = \varphi_{\infty}^{-1}$, and if $\nu\in\partial{T}$, then $\varphi$ maps a neighbourhood of $\nu$ in $\textrm{cl}[\mathcal{T}]:=\mathcal{T}\cup\partial\mathcal{T}$ into a neighbourhood of $\varphi_{\infty}(\nu)$ in $\textrm{cl}[\tilde{S}]$. (A neighbourhood of $\nu$ in $cl[\mathcal{T}]$ is defined in a natural way, as the component containing $\nu$ of the complement of a large compact set in $\mathcal{T}$). The same holds for $\psi$ and $\psi_{\infty}$.

As we mentioned earlier, we shall prove a version of Lemma \ref{lema geometrico} for the tree $\mathcal{T}$. First we need some background on the behaviour of $L_g$ for a non trivial $g\in\mathbb{F}_l$. The {\em translation length} of $g$ is the minimum $$l_g=\min\{d_{\mathcal{T}}(x,L_g(x)): x\in\mathcal{T} \} $$ which is achieved on a line, isometrically embedded in $\mathcal{T}$, where $L_g$ is conjugated to a translation by the distance $l_g$. This line is called the {\em axis} $A_g\subset\mathcal{T}$ of $g$. The line $A_g$ has two endpoints in $\partial\mathcal{T}$, which are also the attracting and repelling points of the dynamics of $L_g$ extended by continuity to $\textrm{cl}[\mathcal{T}]$. We shall call these points $g^+$ and $g^-$, respectively attracting and repelling.

The following is the tree version of Lemma \ref{lema geometrico}.

\begin{lem}\label{sublema} Let $g\in\mathbb{F}_l$ be a non trivial element, and $R_0$ a positive real constant.

There exists a compact set $\mathcal{K}_0\subseteq \mathcal{T}$ with the following property: For any $n\in\N$  there exists $W'_n$ a neighbourhood of $g^+\in \textrm{cl}[\mathcal{T}]$ and integers $m_1,\ldots,m_n$ such that

\begin{itemize}

\item $L_g^{m_i}(\mathcal{K}_0)\cap L_g^{m_j}(\mathcal{K}_0) = \emptyset$ if $i\neq j$,
\item for every $R_0$-path $X'=\{x'_0,\ldots,x'_m\}\subseteq \mathcal{T}$ with $x'_0\in\mathcal{K}_0$ and $x'_m\in W'_n$ we have $L_g^{m_i}(\mathcal{K}_0)\cap X'\neq\emptyset$ for $i=1,\ldots n$.
\end{itemize}

Moreover, given a compact set $D'\subseteq \mathcal{T}$ we can suppose that $D'\subseteq \mathcal{K}_0$.

\end{lem}

\begin{proof} Take $x\in A_g \subset \mathcal{T}$ and set $\mathcal{K}_0 = B(x,R_0)$, the closed ball around $x$ with radius $R_0$ in the distance $d_{\mathcal{T}}$. Given $n\in\N$, we take $m_1,\ldots,m_n$ positive integers so that $|m_i-m_j| > 2R_0$ for $i\neq j$. For instance, we could set $m_i = i(2R_0+1)$ if $R_0$ is an integer. Let us introduce the notation $\mathcal{K}_i=L_g^{m_i}(\mathcal{K}_0)$ for $i=1,\ldots,n$.

First we show that $\mathcal{K}_i\cap \mathcal{K}_j =\emptyset$ if $i\neq j$, by showing that $d_{\mathcal{T}}(\mathcal{K}_i,\mathcal{K}_j)> 0$ . Since $L_g$ is an isometry of $\mathcal{T}$ we have $$\mathcal{K}_i = L_g^{m_i}(B(x,R_0)) = B(L_g^{m_i}(x),R_0) $$ and since $x\in A_g$ we get that  $d_{\mathcal{T}}(L_g^{m_i}(x),L_g^{m_j}(x)) = |m_i-m_j|l_g$. Notice that $l_g\geq 1$ by our definition of $d_{\mathcal{T}}$, that assigns length $1$ to the edges. Thus
$$d_{\mathcal{T}}(\mathcal{K}_i,\mathcal{K}_j) \geq |m_i-m_j|l_g - 2R_0 > 0$$

For $W'_n$ we can take any neighbourhood of $g^+$ that does not intersect $\mathcal{K}_i$ for $i=0,\ldots,n$. One such neighbourhood is the component containing $g^+$ of the complement of $\bigcup_{i=0}^n\mathcal{K}_i$ in $\textrm{cl}[\mathcal{T}]$. To see that this works, take $X'$ an $R_0$-path from $\mathcal{K}_0$ to $W'_n$, and pick some $i\in\{1,\ldots,n\}$.

If we remove the point $L_g^{m_i}(x)$ from $\mathcal{T}$ we obtain $2l$ connected components $C_1,\ldots,C_{2l}$. Let $C'_r=C_r\setminus \mathcal{K}_i$ for $r=1,\ldots,2l$, i.e. the intersection of $C_r$ with the complement of $B(L_g^{m_i}(x),R_0)$. Since $\mathcal{K}_0$ and $W'_n$ are connected and disjoint from $\mathcal{K}_i$, they are contained in two such sets, say $C'_1$ and $C'_2$ respectively. First we show that $C'_1\neq C'_2$: Since we chose $m_i>0$, we have that $L_g^{m_i}(x)$ lies between $x$ and $g^+$ in the line $A_g$ (when extended to $\textrm{cl}[\mathcal{T}]$). By our definition of $W'_n$, we then have that $x$ and $W'_n\cap A_g$ are on different sides of $L_g^{m_i}(x)$ on that line, implying that $C_1\neq C_2$.

The key property of the sets we defined is that $d_{\mathcal{T}}(C'_r,C'_s)=2R_0$ if $r\neq s$, which is a direct consequence of $\mathcal{T}$ being a tree. Thus, if $X'$ does not intersect $\mathcal{K}_i$, it must remain in $C'_1$, and can not end in $W'_n\subset C'_2$, a contradiction.

Finally, notice that the same proof works if we change $R_0$ for any $R_1 \geq R_0$. In particular, given a compact set $D'\subseteq \mathcal{T}$ we can choose $R_1$ so that $D'\subseteq B(x,R_1)$.

\end{proof}

Now we turn to the final considerations before the proof of Lemma \ref{lema geometrico}. It is a standard fact that the equivariance of the quasi-isometries $\varphi$ and $\psi$ passes to the boundary, i.e. $\varphi_{\infty}$ and $\psi_{\infty}=\varphi_{\infty}^{-1}$ are also equivariant. As such, for $g\in\mathbb{F}_l$ we have that $\varphi_{\infty}(g^+)$ and $\varphi_{\infty}(g^-)$ are the limit points of $T_g$ in $\partial\tilde{S}\subset \partial \D$.

Recall that we are assuming $S$ is filled, otherwise we may work with $\textrm{Fill}(S)$.

 \vspace{.2mm}
{\flushleft \bf Proof of Lemma \ref{lema geometrico}:}

Let $g\in\mathbb{F}_l\cong\pi_1(S)$ be the element corresponding to the deck transformation $T$, i.e. such that $T_g=T$. We can assume $\psi_{\infty}(\xi)=g^+$, otherwise we work with $g^{-1}$ and would obtain $m_1,\ldots,m_n$ negative. Let $\lambda\geq 1$ and $C\geq 0$ be the constants that satisfy

 $$ \lambda^{-1}d_{\D}(x,y) - C< d_{\mathcal{T}}(\psi(x),\psi(y))< \lambda d_{\D}(x,y) + C$$

Then $\psi$ takes an $R$-path in $\tilde{S}$ to an $R_0$-path in $\mathcal{T}$ for $R_0 = \lambda R + C$. Let $\mathcal{K}_0\subset \mathcal{T}$ be the compact set that satisfies Lemma \ref{sublema} for $R_0$. We let $\mathcal{K} = \psi^{-1}(\mathcal{K}_0)$, the pre-image of $\mathcal{K}_0$ under $\psi$. This set is compact, as $\psi$ is continuous and a quasi-isometry.

For $n\in\N$ we take $W'_n$ and $m_1,\ldots,m_n$ as in Lemma \ref{sublema}. By the boundary properties of the quasi-isometries, discussed before Lemma \ref{sublema}, we know that $W_n = \psi^{-1}(W'_n)$ is a neighbourhood of $\xi$. Let us verify that this setting satisfies the statement:

Since $\psi$ is equivariant, $T^{m_i}(\mathcal{K}) = \psi^{-1}(L_g^{m_i}(\mathcal{K}_0))$, and pre-images of disjoint sets are disjoint. On the other hand, if $X=\{x_0,\ldots,x_m\}$ is an $R$-path with $x_0\in\mathcal{K}$ and $x_m\in W_n$, then $X'=\{\psi(x_0),\ldots,\psi(x_m)\}$ is an $R_0$-path that begins in $\mathcal{K}_0$ and ends in $W'_n$. Thus $X'$ meets $L_g^{m_i}(\mathcal{K}_0)$ for all $i=1,\ldots,n$. We use again that $T^{m_i}(\mathcal{K}) = \psi^{-1}(L_g^{m_i}(\mathcal{K}_0))$, and we see that $X$ must meet $T^{m_i}(\mathcal{K})$ for all $i$.

Finally, if $D\subset\tilde{S}$ is compact, we can ensure that $D\subseteq\mathcal{K}$ by making $\mathcal{K}_0$ contain $D'=\psi(D)$.

$\hfill\square$

\subsection{Dichotomy for adjacent Conley surfaces}

We now focus on basic pieces that have adjacent Conley surfaces for a prescribed Conley decomposition, studying whether or not they are heteroclinically related. Our main result in this direction is Proposition \ref{dicotomia}, which constitutes a particular case of this problem.

For a periodic point $q\in \textrm{Per}_m(f)$ of period $m$, let $(q)$ be the  loop obtained by evaluating the isotopy between the identity and $f^m$ in $q$. If $b\subset\Sigma$ is a simple closed curve, we say that $q$ {\em is in the class of} $b$ if $(q)$ is freely homotopic to $b^k$ for some non zero $k\in\Z$. This can also be described in terms of the canonical lift: If $\tilde{b}$ is a lift of $b$ to $\D$ and $T$ is a generator of $\textrm{Stab}(\tilde{b})\subseteq\pi_1(\Sigma)$, then $q\in \textrm{Per}_m(f)$ is in the class of $b$ if there is a lift $\tilde{q}$ of $q$ such that $F^m(\tilde{q})=T^k(\tilde{q})$ for some $k\in\Z$.


\begin{proposition}\label{dicotomia} Consider $\Lambda_1$, $\Lambda_2$ saddle type basic pieces such that:
\begin{itemize}
\item[a)] $\Lambda_1$ has a contractible fixed point.
\item[b)] $\Lambda_2$ is annular and $\rho_{\Lambda_2}(f)=\{v\}$ where $v\neq 0$.
\item[c)] $S_{\Lambda_1}\cap S_{\Lambda_2}:=b$ is an essential simple closed curve.
\item[d)] $f(S_{\Lambda_1})\cap S_{\Lambda_2}\neq\emptyset$.
\end{itemize}
Then one of the following holds:
\begin{enumerate}
\item There exists a periodic point $q\in\Lambda_1$ in the class of $b$, or
\item $\Lambda_1\prec\Lambda_2$.
\end{enumerate}
\end{proposition}

The conclusion can be interpreted as saying that an heteroclinical connection $\Lambda_1\prec\Lambda_2$ can only be avoided if some of the rotation of $\Lambda_1$ mirrors that of $\Lambda_2$. That is to be taken in the sense that $v$ is a non zero multiple of $[b]\in H_1(\Sigma;\R)$, since $\Lambda_2$ is annular, and option (1) implies a non zero rotation vector for $\Lambda_1$ also in the direction of $[b]$. It should be noted that the two options in the conclusion are not exclusive.

Before the proof of Proposition \ref{dicotomia} we need to recall some general facts about topological dynamics. We shall state these results in the particular case that concerns us, i.e. that of $f\in \mathcal{A}_0(\Sigma)$. For $n\in\N$ and $x\in\Sigma$ we write $$\mathcal{O}_n(x,f) = \{x,f(x),\ldots,f^n(x) \} $$ the $n$-{\em partial orbit} of $x$ by $f$. We restate Lemma \ref{l.saltos2} using this notation:

\begin{lem}\label{wandering} (Lemma \ref{l.saltos2}) Let $U$ be a neighbourhood of $\Omega(f)$. Then there exists $\kappa\in\N$ such that for every $x\in\Sigma$ and $n\geq \kappa$, we have $\mathcal{O}_n(x,f)\cap U\neq\emptyset$.
\end{lem}


We will need a more refined statement about the behaviour of partial orbits near a basic piece.

\begin{lem}\label{lema dinamico} Let $\Lambda$ be a basic piece of $f$ and $U$, $S$ be open sets such that $\Lambda\subseteq U\subseteq S$ and $\textrm{cl}[S]\cap \Omega(f)=\Lambda$. Then there exists $n_0\in\N$ and an open set $V$ with $\Lambda\subseteq V\subseteq U$ that satisfy the following property:
If a partial orbit $\mathcal{O}_m(x,f)$ is included in $S$ with $x\in V$ and $m>n_0$,  then $\mathcal{O}_{m-n_0}(x,f)$ is included in $U$.
\end{lem}
\begin{proof}
First we see that there exists a neighbourhood $V$ of $\Lambda$, included in $U$, with the property that if $x\in V$ and $f^i(x)\notin U$ for some $i>0$ then $f^j(x)\notin V$ for every $j>i$. Otherwise we could construct a sequence $x_k$ in $\Sigma\setminus U$ and sequences of integers $m_k<0<n_k$  such that $f^{m_k}(x_k)$ and $f^{n_k}(x_k)$ converge to points in $\Lambda$. Taking an accumulation point of $x_k$ we would obtain a point outside of $U$ that should belong to $\Lambda$,
as such a point should be in a recurrent class of $f$.

Consider the neighbourhood of $\Omega(f)$ defined as $V_0:=V\cup \textrm{int}[\Sigma\setminus S]$. Then, by Lemma \ref{wandering}, there exists $\kappa\in\N$ such that any orbit of length $n\geq\kappa$ must intersect $V_0$. We claim that this $V$ and $n_0=\kappa$ satisfy the conclusion.

Consider a partial orbit $\mathcal{O}_m(x,f)$ included in $S$ with $x\in V$, and suppose that $f^i(x)\notin U$ for some $i>0$. By construction of $V$, and the assumption that $\mathcal{O}_m(x,f)\subset S$, we get that $\{f^i(x),\ldots,f^m(x)\} \cap V_0=\emptyset$. Thus we obtain from Lemma \ref{wandering} that $m-i\leq n_0$, implying $i\geq m-n_0+1$. This means that $\mathcal{O}_{m-n_0}(x,f)\subset U$, by the definition of $i$.
\end{proof}






We are ready to prove Proposition \ref{dicotomia}. As well as the previous results of this section, the reader may want to recall Lemma \ref{l.entornos1}.

{\flushleft \bf Proof of Proposition \ref{dicotomia}:}





Take $\tilde{S}_1\subset\D$ any lift of $S_{\Lambda_1}$. Condition (a) in the hypotheses gives us some $p\in \tilde{S}_1 $ that is fixed by $F$. We can apply Lemma \ref{l.conleycasesess}, taking condition (d) into account, to find some lift $\tilde{b}$ of $b$ in $\tilde{S}_1$ that intersects $W^u(p,F)$ non trivially. By Theorem \ref{t.stamani}, using that $\Lambda_1$ is of saddle type, we have that $W^u(p,F)$ is an immersed manifold of dimension 1, which $p$ divides into two branches. Let $B$ be a branch of $W^u(p,F)$ that meets $\tilde{b}$.


We shall introduce some orientation conventions: As we remarked before, condition (b) gives us that $v$ is a non zero multiple of $[b]\in H_1(\Sigma;\R)$. We pick an orientation for $b$ so that $v$ is a positive multiple of $[b]$. Lifting induces an orientation for $\tilde{b}$, which in turn defines a left and a right side of $\tilde{b}$ in $\D$, where the orientation of $\D$ is induced by lifting that of $\Sigma$. We assume that $\tilde{S}_1$ lies at the left of $\tilde{b}$. The other case admits a symmetrical proof.

Consider $T$ the generator of $\textrm{Stab}(\tilde{b})$ that moves the points of $\tilde{b}$ in the positive direction, according to the orientation we just defined on $\tilde{b}$. Let $\xi_{+},\xi_{-}\in\partial\D$ be the endpoints of $\tilde{b}$, which are also the attracting and repelling points of $T$ respectively. We shall split the proof into two cases, according to whether or not  $B\cap\tilde{S}_1$ accumulates on $\{\xi_{-},\xi_{+}\}$. When it does we will show that there is a periodic point in $\Lambda_1$ in the class of $b$, i.e. option (1) in the conclusion. In the complementary case we will show that $\Lambda_1\prec\Lambda_2$. Let $D_1$ be a fundamental domain for the covering $\pi:\tilde{S}_1\to S_{\Lambda_1}$  that contains $p$. 






\smallskip

{\bf Case I:} $B\cap\tilde{S}_1$ accumulates on either $\xi_{+}$ or $\xi_{-}$.

We shall assume $B\cap\tilde{S}_1$ accumulates on $\xi_{+}$, the other case is analogous. We aim to find a periodic point in $\Lambda_1$ that is in the class of $b$. Consider $\tilde{\Lambda}_1=\pi^{-1}(\Lambda_1)\cap\tilde{S}_1$, which is an invariant hyperbolic set for $F$. An adaptation of the shadowing lemma (Theorem \ref{t.shadchain}) to $F$ and $\tilde{\Lambda}_1$ shows that there is $\epsilon_0>0$ and $U$ a neighbourhood of $\tilde{\Lambda}_1$ such that any $\epsilon_0$-pseudo-orbit for $F$ contained in $U$ is shadowed by an orbit of $F$ in $\tilde{\Lambda}_1$. Moreover, if the pseudo-orbit projects by $\pi$ to a periodic pseudo-orbit for $f$, then the projection of the shadowing orbit is also periodic.

This reduces the problem to find $\{x_j\}$ an $\epsilon_0$-pseudo-orbit for $F$, contained in $U$, such that $x_m = T^k(x_0)$ for some $m,k>0$: If the orbit of $x\in\tilde{\Lambda}_1$ shadows $\{x_j\}$, then $F^m(x)=T^k(x)$, and so $\pi(x)\in\Lambda_1$ will be a periodic point of $f$ in the class of $b$.


If we take a neighbourhood $W$ of $\xi_{+}$, the assumption of Case I implies that $B\cap\tilde{S}_1\cap W$ is non empty. According to the definition of unstable manifold, any point in $B\cap\tilde{S}_1\cap W$ can be written as $F^n(y)$ for some $n>0$ and $y\in D_1$, where $y$ can be chosen as close as we want to $p$ as long as $n$ is large enough. Moreover, by the properties of Conley surfaces we have that the partial orbit $\mathcal{O}_n(y,F)$ is contained in $\tilde{S}_1$. Thus we get $y\in D_1$ with $\mathcal{O}_n(y,F)\subset \tilde{S}_1$ and $F^n(y)\in W$.

We shall find a suitable neighbourhood $W$ of $\xi_{+}$ so that the $F$-orbit of the point $y$ obtained from this argument can be perturbed to a pseudo-orbit with the desired properties.



Let $R=\text{max}\{d(x,F(x)):x\in\D\}$, and $\mathcal{K}$ a compact set given by applying Lemma \ref{lema geometrico} to $S_{\Lambda_1}$, $b$ and $R$. We assume $D_1\subset\mathcal{K}$, as Lemma \ref{lema geometrico} says we can. By compactness there exists $N>0$ such that any subset of $\mathcal{K}$ with $N$ or more points must have two points that are closer than $\epsilon_0$ to each other. Let $W_N$ be a neighbourhood of $\xi_{+}$ and $m_1,\ldots,m_N\in\Z$ given by Lemma \ref{lema geometrico} for this $N$. Notice that the choice of $R$ implies that any $F$-orbit is an $R$-path. If we take $W\subset W_N$, then we will have that $\mathcal{O}_n(y,F)$ meets $N$ disjoint translates of $\mathcal{K}$ by $T$, as stated in Lemma \ref{lema geometrico}. More precisely, there are $n_i\in\N$ with $F^{n_i}(y)\in T^{m_i}(\mathcal{K})$ for $i=1,\ldots,N$. Then $y_i=T^{-m_i}(F^n_i(y))$ are in $\mathcal{K}$ for $i=1,\ldots,N$, and by the choice of $N$ we have two of them that are closer than $\epsilon_0$ to each other. Let $n_i<n_j$ be such that $y_i$ and $y_j$ are closer than $\epsilon_0$. We define $x_j=F^{n_i+j}(y)$ for $j=0,\ldots,n_j-n_i$ and $x_m = T^{m_j-m_i}(x_0)$ for $m=n_j-n_i+1$. Then $\{x_j\}_{j=0}^m$ is an $\epsilon_0$-pseudo-orbit contained in $\tilde{S}_1$, with $x_m = T^k(x_0)$ for $k=m_j-m_i$.

We only have to adjust $W\subset W_N$ so that we can ensure the pseudo-orbit $\{x_j\}$ arising from our construction is contained in $U$. Since $F$ is a lift of $f$, we can assume $U$ is the lift of a neighbourhood $U_0$ of $\Lambda_1$, i.e. $U = \pi^{-1}(U_0)\cap\tilde{S}_1$. Let $V_0\subset U_0$ be a neighbourhood of $\Lambda_1$, and $n_0>0$, as given by Lemma \ref{lema dinamico} for the basic piece $\Lambda_1$ and the open sets $U_0$ and $\text{int}(S_{\Lambda_1})$. We consider $V = \pi^{-1}(V_0)\cap\tilde{S}_1$ and $W$ a neighbourhood of $\xi_{+}$ such that $F^{-i}(W)\subset W_N$ for $i=0,\ldots,n_0$, which is possible since $F$ extends continuously as the identity on $\partial \D$. As we remarked before, we can assume $y\in V$ and $n>n_0$ so Lemma \ref{lema dinamico} gives us that $\mathcal{O}_{n-n_0}(y,F)\subset U$. By the choice of $W$ we have $F^{n-n_0}(y)\in W_N$, thus we can apply the previous argument to $\mathcal{O}_{n-n_0}(y,F)$, which also meets $N$ disjoint translates of $\mathcal{K}$ by $T$. This gives us an $\epsilon_0$-pseudo-orbit $\{x_j\}$ contained in $U$ with $x_m = T^k(x_0)$ for some $m,k>0$, and so concludes the proof in Case I.

\smallskip

{\bf Case II:} $B\cap\tilde{S}_1$ does not accumulate on $\{\xi_{+},\xi_{-}\}$.


In this case we want to deduce that $\Lambda_1\prec\Lambda_2$. Let $b'$ be the other essential component of $\partial S_{\Lambda_2}$, so that $\textrm{Fill}(S_{\Lambda_2})$ is the annulus between $b$ and $b'$. We consider $\tilde{S}_2$ the lift of $S_{\Lambda_2}$ that is adjacent to $\tilde{S}_1$ at $\tilde{b}$, and $\tilde{b}'$ the lift of $b'$ to this cover. Notice that $\textrm{Stab}(\tilde{b}) = \textrm{Stab}(\tilde{S}_2) = \textrm{Stab}(\tilde{b}')$, and that $\textrm{Fill}(\tilde{S}_2)$ is the band between $\tilde{b}$ and $\tilde{b}'$, which has $\xi_{+}$ and $\xi_{-}$ as limit points in $\partial\D$. We orient $b'$ so that $T$ moves the points of $\tilde{b}'$ in the positive direction, thus $\textrm{Fill}(\tilde{S}_2)$ is the region at the right of $b$ and the left of $b'$.

Notice that if $\Lambda$ is a basic piece  contained in $\textrm{Fill}(S_{\Lambda_2})$, then either  $\Lambda=\Lambda_2$ or $\Lambda$ is trivial and is contained in $\mathcal{D}=\textrm{Fill}(S_{\Lambda_2})\setminus S_{\Lambda_2}$ which is a union of disks that are attracting or repelling. Either way we have $\rho_{\Lambda}(f)=\{v\}$ by condition (b) and Lemma \ref{l.conleycasesess}. The shadowing lemma for basic pieces gives that $\Lambda$ contains some periodic point with rotation $v$, and thus a lift $q\in\textrm{Fill}(\tilde{S}_2)$ of such a periodic point will satisfy $F^m(q)=T^k(q)$ for some $m,k>0$. 

We shall develop the proof in a sequence of claims, which we state below and will prove afterwards.

\smallskip
{\bf Claim 1:} $B\cap \tilde{S}_2$ accumulates on $\xi_{+}$.

\smallskip
{\bf Claim 2:} If $\Lambda_1\not\prec\Lambda_2$ then $B\cap \tilde{b}'$ accumulates on $\xi_{+}$.

\smallskip
{\bf Claim 3:} If $B\cap \tilde{b}'$ accumulates on $\xi_{+}$ then $\pi(B)$ accumulates on $\Lambda_2$.

\smallskip
Claim 1 serves as a lemma to prove claim 2. From claims 2 and 3 we derive a proof that $\Lambda_1\prec\Lambda_2$ by contradiction: If we suppose the contrary we get that $\pi(B)$, which is contained in $W^u(\Lambda_1,f)$, accumulates on $\Lambda_2$. Then Lemma \ref{l.entornos1} would imply that $\Lambda_1\prec\Lambda_2$, against what we supposed. Thus all that remains is to prove the claims.


\smallskip
{\bf Proof of claim 1:}

Consider first the case when $B\cap \tilde{b}'=\emptyset$. Applying Lemma \ref{l.saltos2} to some point in  $\pi(B)\cap b$ we see that in this case $\pi(B)$ must accumulate on some basic piece $\Lambda\subset \textrm{Fill}(S_{\Lambda_2})$. So there are $q_0\in\Lambda$ and $z_0\in\pi(B)$ with $z_0\in W^s(q_0,f)$, and then we can take lifts $q\in\tilde{\Lambda} = \pi^{-1}(\Lambda)\cap\textrm{Fill}(\tilde{S}_2)$ and $z\in B$ with $\pi(q)=q_0$, $\pi(z)=z_0$, and $z\in W^s(q,F)$. Since $\rho(f,q_0) = v$ we have that $F^n(q)$ tends to $\xi_{+}$ as $n\to+\infty$, and then so does $F^n(z)$, implying that $B\cap\textrm{Fill}(\tilde{S}_2)$ accumulates on $\xi_{+}$. Since $\textrm{Fill}(\tilde{S}_2)\setminus \tilde{S}_2$ is a disjoint union of bounded disks and $B$ is connected, it follows that $B\cap \tilde{S}_2$ accumulates on $\xi_{+}$.





We consider now the situation when $B\cap \tilde{b}'\neq\emptyset$. This implies in particular that $F(\tilde{b}')$ lies at the right side of $\tilde{b}'$, for otherwise we would have that $\textrm{Fill}(\tilde{S}_2)$ is an attracting region for $F$, which falls under the previous case. Consider an orientation (and therefore an order) on $B$ so that $F$ move points of $B$ in the positive direction. With respect to this order, let $\alpha'$ be the minimum in $B$ of the set $B\cap \tilde{b}'$ and let $[p,\alpha']$ denote the arc of $B$ from $p$ to $\alpha'$. Recalling the orientation (and order) of $\tilde{b}$, let $\alpha$ be the minimum in $\tilde{b}$ of the set $\tilde{b}\cap[p,\alpha']$. Take $\gamma=[\alpha,\alpha']$ as the arc in $B$ between $\alpha$ and $\alpha'$. Also define $h\subset\tilde{b}$ as the half-line joining $\xi_{-}$ with $\alpha$ and $h'\subset\tilde{b}'$ as the half-line joining $\xi_{-}$ with $\alpha'$. The concatenation of $h$, $\gamma$ and $h'$ gives a Jordan curve $J\subset \textrm{cl}[\D]$, and we consider the region $D$ bounded by it, i.e. the component of $\textrm{cl}[\D]\setminus J$ that is disjoint from $\partial\D$. (Figure \ref{inter} depicts $D$, among other elements of this proof).


Clearly $D$ contains fundamental domains of the cover $\pi:\tilde{S}_2\to S_{\Lambda_2}$ so there is $q\in D$ a lift of a periodic point in $\Lambda_2$ as in the remark previous to the statements of the claims. Thus given an arbitrary neighbourhood $W$ of $\xi_{+}$ there is $n\in\N$ with $F^n(q)=T^k(q)\in W$ for some $k>0$. We choose $W$ so that $\tilde{S}_2\cap W$ is connected and consider a ray $r$ in $\tilde{S}_2\cap W$ joining $F^n(q)$ with $\xi_{+}$. These are all the objects shown in Figure \ref{inter}. We shall show that $B$ intersects $r$, which implies the claim.


\tiny
\begin{figure}[ht]\begin{center}

\centerline{\includegraphics[height=6.5cm]{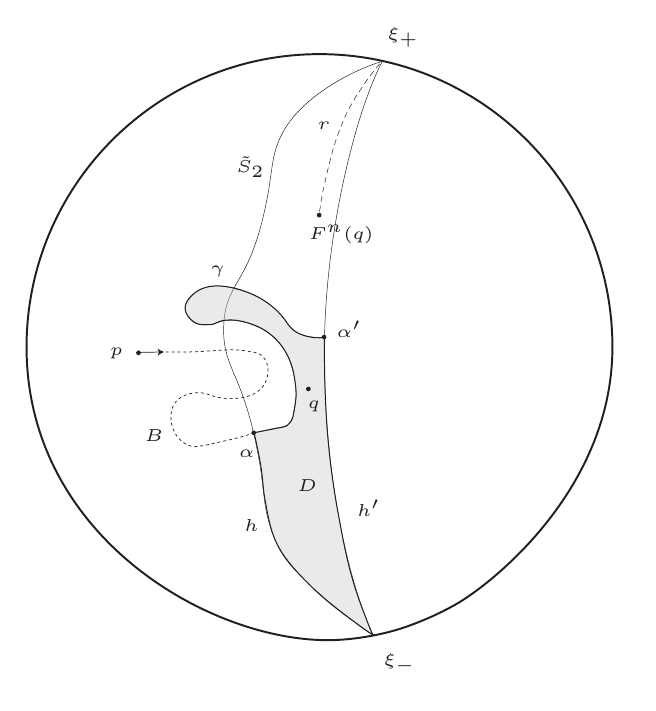}}
\caption{Objects considered in case II of the proof or Proposition \ref{dicotomia}.}
\label{inter}
\end{center}\end{figure}
\normalsize

Recall that $F$ extends continuously to $\textrm{cl}[\D]$ as the identity on $\partial\D$, so $F^n(J)$ is also a Jordan curve in $\textrm{cl}[\D]$ that meets $\partial\D$ exactly on $\xi_{-}$, and bounds the region $F^n(D)$. Since $r$ joins $F^n(q)\in F^n(D)$ with $\xi_{+}$, it must intersect $F^n(J)$. If $r$ meets $F^n(\gamma)$, which is an arc of $B$, we are already done. On the other hand $r$ cannot intersect $F^n(h')$ since it is disjoint from $\tilde{S}_1$, as we have $F(\tilde{b}')$ at the right side of $\tilde{b}'$. Thus let us assume $F^n(h)$ meets $r$ and let $z$ be the minimum of $r\cap F^n(h)$ in the line $F^n(\tilde{b})$, with the orientation induced by that of $\tilde{b}$, i.e. going from $\xi_{-}$ to $\xi_{+}$.

Let $t$ be the concatenation of $t_1$, the arc of $F^n(h)$ from $\xi_{-}$ to $z$, and $t_2$, the arc of $r$ from $z$ to $\xi_{+}$. Notice that $t_1$ and $t_2$ only meet at $z$, so $t$ is a line, and that it separates $\tilde{b}$ from $F^n(\tilde{b}')$, since both $F^n(h)$ and $r$ lie at the right of $\tilde{b}$ and the left of $F^n(h')$. Thus $F^n([p,\alpha'])$ must intersect $t$, since it is an arc from $p$, at the left of $\tilde{b}$, to $F^n(\alpha')\in F^n(\tilde{b}')$. We seek to prove that $F^n([p,\alpha'])\subset B$ intersects $t_2\subset r$, as that would give the claim. By construction of $h$ we have $[p,\alpha']\cap h = \{\alpha\}$ and so we get $F^n([p,\alpha'])\cap F^n(h) = \{F^n(\alpha)\}$. Observing that $F^n(\alpha)$ is the maximum of $F^n(h)$ in the line $F^n(\tilde{b})$, the last fact shows that $F^n([p,\alpha'])$ can only meet $t_1$ in case that $z=F^n(\alpha)$, and then only at that point. In any case $F^n([p,\alpha'])$ cannot meet $t_1$ in its interior, so it must meet $t_2$, and that finishes the proof of this claim.


\smallskip
{\bf Proof of claim 2:}

Consider $q\in \tilde{\Lambda}_2:=\pi^{-1}(\Lambda_2)\cap \tilde{S}_2$. By Lemma \ref{l.conleycasesess} we know that $W^s(q,F)$ intersects $\tilde{b}$. On the other hand, depending on whether $F(\tilde{b}')$ lies at the right or the left of $\tilde{b}$, we have that either $W^s(q,F)$ or $W^u(q,F)$ intersects $\tilde{b}'$, this also coming from Lemma \ref{l.conleycasesess}. In any case, we have an arc $\delta\subset\tilde{S}_2$, starting at a point $\beta\in\tilde{b}$ and ending at a point $\beta'\in\tilde{b}'$, that is contained in $W^s(q,F)\cup W^u(q,F)$. The key point is that $\pi(B)$ cannot meet $W^u(\pi(q),f)$ by definition, nor $W^s(\pi(q),f)$ by the assumption that $\Lambda_1\not\prec\Lambda_2$, thus $B$ cannot meet $T^n(\delta)$ for any $n\in\Z$.

Let $h_n$ be the half-line in $\tilde{b}$ from $T^n(\beta)$ to $\xi_{+}$, and $h'_n$ the half-line in $\tilde{b}'$ from $T^n(\beta')$ to $\xi_{+}$. Let $J_n$ be the Jordan curve in $\textrm{cl}[\D]$ defined by concatenation of $h_n$, $T^n(\delta)$ and $h'_n$, and $D_n$ the disk bounded by $J_n$. Given a neighbourhood $W$ of $\xi_{+}$ we have that $\textrm{cl}[D_n] \subset W$ for $n$ large enough. We aim to show that $B$ intersects $h'_n$, which proves this claim.

By the definition of $h_n$ and $h'_n$, and the compactness of $T^n(\delta)$, there is some neighbourhood $W_n$ of $\xi_{+}$ such that $W_n\cap \tilde{S}_2 \subseteq \textrm{cl}[D_n]$. So Claim 1 gives us that $B$ intersects $\textrm{cl}[D_n]$, thus $B$ meets $J_n$, as the Jordan curve that bounds $D_n$. We recall that $B$ cannot meet $T^n(\delta)$ by the assumption of this claim. On the other hand, $B$ cannot meet $h_n$ if $n$ is large enough by the hypothesis of Case II. So $B$ must meet $h'_n$, and we get the claim.


\smallskip
{\bf Proof of claim 3:}





Take a sequence $x_k \in B\cap \tilde{b}'$ that converges to $\xi_{+}$. We can write $x_k = F^{n_k}(y_k)$ for some $n_k>0$ and $y_k\in D_1$. We see that there is $m_k>0$ so that $F^j(y_k) \in \tilde{S}_1$ for $j=0,\ldots,m_k-1$ and $F^j(y_k) \in \tilde{S}_2$ for $j=m_k,\ldots, n_k$. That is by the properties of Conley surfaces: If $\mathcal{O}_{n_k}(y_k,F)$ meets any lift of any other Conley surface it would not reach $\tilde{b}'$, and once in $\tilde{S}_2$ it cannot return to $\tilde{S}_1$. We remark that this reasoning includes the lifts of the Conley surfaces that form $\mathcal{D}$. In order to ease notation let us introduce $z_k = F^{m_k}(y_k)$.

We consider the partial orbits $\mathcal{O}_{n_k-m_k}(\pi(z_k),f)$, which are contained in $\pi(B)\cap S_{\Lambda_2}$, and we shall prove the claim by showing that their union accumulates in $\Lambda_2$.

Let $W_+$, $W_-$ be neighbourhoods of $\xi_{+}$, $\xi_{-}$ such that $B\cap\tilde{S}_1$ lies outside $W_+\cup W_-$, as given by the assumption of Case II. Let $C$ be the band that lies between the lines $F^{-1}(\tilde{b})$ and $\tilde{b}$, and $K = C \setminus (W_+\cup W_-)$. Notice that $K$ is compact. We see that the sequence $F^{-1}(z_k)$ belongs to $K$: a point $F^{-1}(z_k)$ is outside $W_+\cup W_-$ since it belongs to $\tilde{S}_1\cap B$, and $F^{-1}(z_k)\in C$ since $z_k = F(F^{-1}(z_k))\in\tilde{S}_2$.

Since $x_k = F^{n_k-m_k}(z_k)$ tends to $\xi_{+}$ while $F^{-1}(z_k)$ remains in a compact set, we deduce that $n_k-m_k \to +\infty$ recalling that $d(x,F(x))$ is uniformly bounded on $\D$. Thus if we take $U$ an arbitrary neighbourhood of $\Lambda_2$, by Lemma \ref{wandering} there is $k$ such that $\mathcal{O}_{n_k-m_k}(\pi(z_k),f)$ must intersect $U$, and that concludes the proof.

$\hfill\square$

\subsection{Unstable-stable paths and walls}

We turn now to introduce some key objects for our study of the heteroclinical relationship. They  will consist on continuous paths that can be written as a concatenation $\gamma = \gamma^u\gamma^s$ where $\gamma^u$ and $\gamma^s$ are contained respectively in the unstable and stable manifold of some saddle type basic pieces. Among such paths, the ones that will be of interest come in two versions, namely unstable-stable {\em paths} and {\em walls}, which will play complementary roles in order to help us deduce some forcing conditions for the heteroclinical relationship between annular pieces. This forcing result is Theorem \ref{conexion cruzada} at the end of this section, and we shall need the unstable-stable paths not only to prove it, but also to give the definitions that appear in its statement.

For convenience, if $\gamma$ is a path we will denote by $\gamma^-$ and $\gamma^+$ the initial and final endpoints of $\gamma$. Let $\Lambda_1$ and $\Lambda_2$ be basic pieces of $f$ (not necessarily different). We say that $\gamma$ is a $(\Lambda_1,\Lambda_2)$-{\em path} if it can be written as a concatenation $\gamma = \gamma^u\gamma^s$ such that

\begin{itemize}
\item $\gamma^u\subset W^u(\Lambda_1,f)$ and $\gamma^s\subset W^s(\Lambda_2,f)$, and
\item $\gamma^-\in\Lambda_1$ and $\gamma^+\in\Lambda_2$.

\end{itemize}

On the other hand, $\gamma$ is a $(\Lambda_1,\Lambda_2)$-{\em wall} if we have $\gamma = \gamma^u\gamma^s$ with

\begin{itemize}
\item $\gamma^u\subset W^u(\Lambda_1,f)$ and $\gamma^s\subset W^s(\Lambda_2,f)$, and
\item $\gamma$ is a loop (i.e. $\gamma^-=\gamma^+$) that is non trivial in homotopy.
\end{itemize}

Notice that the existence of a $(\Lambda_1,\Lambda_2)$-path or wall implies that $\Lambda_1\preceq\Lambda_2$. On the other hand, if $\Lambda_1\prec\Lambda_2$ we can obtain the existence of a $(\Lambda_1,\Lambda_2)$-path from the definition of heteroclinical connection and Theorem \ref{t.stamani}. An important difference between paths and walls is that walls do not need to intersect the basic pieces, they may be contained in the wandering set. Also, it is not always the case that a $(\Lambda_1,\Lambda_2)$-wall exists for $\Lambda_1\prec\Lambda_2$.

The following is a basic property for any curve of the form $\gamma^u\gamma^s$ as above, in particular for unstable-stable paths and walls, which are the cases we will be using.
Recall the notation introduced in section \ref{s.conley}: $S^{(n)}=\textrm{Fill}\left(\bigcup_{|j|\leq n} f^j(S)\right)$.

\begin{lem}\label{casi-soporte} Let $\Lambda_1$ and $\Lambda_2$  be basic pieces of $f$ contained in the interior of a sub-surface $S$ and let $\gamma$ be a $(\Lambda_1,\Lambda_2)$-path or wall. Then there exists $n\in\N$ such that $\gamma\subset S^{(n)}$.
\end{lem}
\begin{proof} Since $\gamma^u$ has compact image and $S$ is a neighbourhood of $\Lambda_1$, there exists $n\in\N$ such that the image of $\gamma^u$ is contained in $f^{n}(S)$. The same works for $\gamma^s$.
\end{proof}

The next lemmas give us constructions of unstable-stable walls. More precisely, they tell us that some conditions on the rotation set imply the existence of walls.

\begin{lem}\label{pared} Suppose $\Lambda_1\prec\Lambda_2$ are basic pieces of $f$ and assume that $\Lambda_1$ contains a contractible fixed point. Then for every non-contractible periodic point $q\in\Lambda_2$ there exists a $(\Lambda_1,\Lambda_2)$-wall freely homotopic to $(q)$.

\end{lem}

We point out that $\Lambda_1$ and $\Lambda_2$ need not be different.

\begin{proof} Let $p\in \Lambda_1$ be a contractible fixed point. Since $\Lambda_1\prec\Lambda_2$ there exists a $(\Lambda_1,\Lambda_2)$-path $\gamma=\gamma^u\gamma^s$ with $\gamma^-=p$ and $\gamma^+=q$. Let $\tilde{\gamma}$ be a lift of $\gamma$, where $\tilde{\gamma}^-=\tilde{p}$ is fixed by $F$ and $\tilde{\gamma}^+=\tilde{q}$ satisfies $F^m(\tilde{q})=T^k(\tilde{q})$ for some deck transformation $T$ and $m,k>0$. We consider $x\in\tilde{\gamma}$ that belongs to $W^{u}(\tilde{p},F)\cap W^s(\tilde{q},F)$, for instance we can take $x=\tilde{\gamma}^u{}^+=\tilde{\gamma}^s{}^-$. Since $F(\tilde{p})=\tilde{p}$ we can take an arc $\delta^u\subset W^u(\tilde{p},F)$ joining $x$ and $F^m(x)$. On the other hand $F^m(x)\in W^s(F^m(\tilde{q}),F)= W^s(T^k(\tilde{q}),F)$, and by equivariance we have $W^s(T^k(\tilde{q}),F)= T^k(W^s(\tilde{q},F))$ which contains $T^k(x)$, so we can consider $\delta^u\subset W^s(T^k(\tilde{q}),F)$ joining  $F^m(x)$ and $T^k(x)$. Then we can concatenate $\delta^u$ with $\delta^s$ and $\pi(\delta^u\delta^s)$ is a loop with homotopy class corresponding to $T^k$, which is the class of $(q)$. Thus $\pi(\delta^u\delta^s)$ is the desired $(\Lambda_1,\Lambda_2)$-wall.

\end{proof}

If the Conley surfaces of $\Lambda_1$ and $\Lambda_2$ are adjacent we can get something a bit stronger.

\begin{rem} \label{casi-soporte2} If $S_{\Lambda_1}$ and $S_{\Lambda_2}$ are adjacent, so $S_{\ast}=\textrm{Fill}(S_{\Lambda_1}\cup S_{\Lambda_2})$ is connected, we see from Lemma \ref{casi-soporte} that the $(\Lambda_1,\Lambda_2)$-wall $\gamma$ of Lemma \ref{pared} is contained in $S^{(n)}$ for some $n$. We remark that $\gamma$ and $(q)$ are freely homotopic inside $S^{(n)}$, since they are both represented by the same deck transformation $T^k$, not merely by conjugate deck transformations.
\end{rem}

The need for the last remark only comes up when some component of $\Sigma\setminus S_{\ast}$ is an annulus, so some classes in $\pi_1(S_{\ast})$ are conjugate in $\pi_1(\Sigma)$ but not in $\pi_1(S_{\ast})$.

\begin{lem}\label{rota-intervalo} Let $\Lambda$ be an annular basic piece of $f$ such that $\rho_{\Lambda}(f)\subset H_1(S_{\Lambda};\R)$ is a non-trivial interval.
Then there exists a $(\Lambda,\Lambda)$-wall with homotopy type in $S_{\Lambda}$.
\end{lem}

\begin{proof} Recall that the shadowing lemma for basic pieces implies that rational points in $\rho_{\Lambda}(f)$ are realized by periodic orbits. Thus if $0\in\rho_{\Lambda}(f)$ we can conclude directly from the previous Lemma \ref{pared}. Otherwise, we use a minor adaptation: Take a lift $\tilde{S}$ of $S_{\Lambda}$ and $T$ a generator of $\textrm{Stab}(S)$. If $p\in\tilde{S}$ projects to a periodic orbit realizing some rational rotation vector, we have $F^m(p)=T^k(p)$ for some $m,k\in\Z$. We consider the map $G=F\circ T^{-k}$, which is another lift of $f$ and has a fixed point inside $\tilde{S}$. Since $\rho_{\Lambda}(f)$ is a non-trivial interval, there exist $\tilde{q}\in \tilde{S}$ that projects to a periodic orbit with different rotation than that of $\pi(p)$. Then there are $n,l\in\Z^{*}$ such that $G^{n}(\tilde{q})=T^l(\tilde{q})$. The rest of the proof follows arguing as in Lemma \ref{pared} with the map $G$.
\end{proof}

A similar fact to Remark \ref{casi-soporte2} holds for Lemma \ref{rota-intervalo}.

\subsection{Annular packages and orientation conventions} \label{sec.pack}

Let us introduce some notations that will be important for the rest of this section as well as the next one. Some of them will even take part in the statement of our main theorem \ref{mainthm}. If $\Lambda$ is an annular basic piece of $f$ we define its {\em annular package} $[\Lambda]$ as the set of all annular pieces of $f$ whose filled in Conley surfaces are in the homotopy class of $\textrm{Fill}(S_{\Lambda})$. Notice that the annular packages are equivalence classes on the set of annular basic pieces of $f$, and that there are at most $3g-3$ such classes. Further, if $\Lambda$ is annular we define the {\em essential Conley surface} of the package $[\Lambda]$ as $$S_{[\Lambda]}=\bigcup_{\Lambda'\in[\Lambda]}\textrm{Fill}(S_{\Lambda'})$$ and notice that each $S_{[\Lambda]}$ is an embedded annulus and that different annular packs give rise to disjoint surfaces that are not equivalent modulo homotopy, thus there are at most $3g-3$ of them.

If $\gamma$ is a (parametrized) essential closed curve on $\Sigma$ and $\tilde{\gamma}$ is a lift of $\gamma$ to the universal cover $\D$, then we can define a right and a left hand side of $\tilde{\gamma}$ according to the orientation of $\tilde{\gamma}$ and the orientation of $\D$ lifted from that of $\Sigma$. Let $R_{\tilde{\gamma}}$ and $L_{\tilde{\gamma}}$ be the left and right components of $\D\setminus \tilde{\gamma}$. Of course $R_{\tilde{\gamma}}$ and  $L_{\tilde{\gamma}}$ depend on the parametrization of $\gamma$, or more precisely, on the orientation of $\gamma$.

For each package $[\Lambda]$ we choose an essential simple closed curve $c_{[\Lambda]}$ in $S_{[\Lambda]}$, which gives a generator for $\pi_1(S_{[\Lambda]})$. This choice induces the following orientation conventions: If $\gamma$ is an essential simple closed curve that can be isotoped into $S_{[\Lambda]}$, it is positively oriented if it is homotopic to $c_{[\Lambda]}$ and negatively oriented if it is homotopic to $c_{[\Lambda]}^{-1}$. Choosing always the positive orientation we can define a local left and right hand side for such curves in a way that is coherent among those in the same class: if $\gamma_1$ and $\gamma_2$ are disjoint simple closed curves in the class of $S_{[\Lambda]}$, then the annulus that they bound is on the right side of one of those curves and on the left side of the other (recall that $\Sigma$ is not a torus). Notice that the orientation convention is such that a small neighbourhood of $\tilde{\gamma}$ in $L_{\tilde{\gamma}}$ projects to the local left hand side of $\gamma$ and conversely for the right hand sides. This also allows us to say whether $S_{\Lambda_1}$ is on the right or the left of $S_{\Lambda_2}$ for $\Lambda_1,\Lambda_2 \in [\Lambda]$. If $P$ is an annulus on the isotopy class of $S_{[\Lambda]}$ and $\tilde{P}$ is a lift of $P$, we define $L_{\tilde{P}}$ as $L_{\tilde{b}_0}$ where $b_0$ is the left boundary of $P$. We define $R_{\tilde{P}}$ analogously for $b_1$ the right boundary of $P$.

Next we prove a basic lemma that follows section \ref{s.conley}, but was postponed until now since it is only used in this section and requires some of the definitions above.

\begin{lem}\label{homotopia} Let $S$ be the fill of a non-trivial Conley surface and $\tilde{S}\subseteq\D$ a connected lift of $S$. Consider $h:S^1\times [0,1]\to\Sigma$ a free homotopy of essential closed curves and let $H:\R\times [0,1]\to\D$ be a lift of $h$. Denote by $h_t = h(S^1\times\{t\})$  and $H_t = H(\R\times\{t\})$ for $t\in [0,1]$. Assume that

\begin{itemize}
\item $h_0\subset S$,
\item $h_t\subset S^{(m)}$ for some $m\in\N$ and all $t\in [0,1]$,
\item $H_0 \subset \tilde{S}$.
\end{itemize}

Then
$$ (L_{H_0}\setminus \tilde{S})\cap \pi^{-1}(\Omega(f))=(L_{H_1}\setminus \tilde{S})\cap\pi^{-1}(\Omega(f)).$$
The same works interchanging $L_{H_0}$ with $R_{H_0}$ and $L_{H_1}$ with $R_{H_1}$.
\end{lem}




\begin{proof}

Clearly $H_t^+, H_t^-\in\partial\D$ are constant and $L_{H_t}$ accumulates on one side of $\partial\D\setminus \{H_t^+, H_t^-\}$, which is the same side for every $t$. Because of that it suffices to show that $H(\R\times [0,1])$ does not intersect $\pi^{-1}(\Omega(f))\setminus \tilde{S}$. That happens since $h(S^1\times [0,1])\subset S^{(m)}$ which does not meet $\Omega(f)\cap S$ by Lemma \ref{l.conleycasesess0}.



The proof for the other case is analogous.

\end{proof}

We remark that we can replace $S$ in  Lemma \ref{homotopia} with $S^{(m_0)}$ for $m_0\leq m$. This is a corollary of Lemmma \ref{homotopia} and the fact that  $S^{(n)}$ retracts by deformation to $S$ rel $\Omega(f)$, which is part of Lemma \ref{l.conleycasesess0}.

\subsection{Oriented heteroclinical connections} \label{sec.oriprec}

The unstable-stable paths will allow us to introduce an oriented version of the heteroclinical relationship for annular pieces. Let $\Lambda$ be an annular basic piece and $\Lambda_1$ any non-trivial basic piece with $\Lambda_1\prec\Lambda$. We say that the heteroclinical connection is {\em on the left} of $\Lambda$, and write $\Lambda_1\prec^L \Lambda$ if there are lifts $\tilde{S}$ and $\tilde{S}_1$ of $\textrm{Fill}(S_{\Lambda})$ and $\textrm{Fill}(S_{\Lambda_1})$ respectively and $\tilde{\delta}$ a lift of a $(\Lambda_1,\Lambda)$-path $\delta$ that satisfy:

\begin{itemize}
\item $\tilde{S}_1\subseteq L_{\tilde{S}}$,
\item $\tilde{\delta}^-\in \tilde{S}_1$, and $\tilde{\delta}^+\in \tilde{S}$.
\end{itemize}

We define $\Lambda\prec^R \Lambda_1$ analogously, only changing the first point for $\tilde{S}_1\subseteq R_{\tilde{S}}$. There is also a symmetric version for a connection of the type $\Lambda\prec\Lambda_2$ where $\Lambda$ is annular and $\Lambda_2$ non-trivial. We spell it out for clarity: The heteroclinical connection is on the left of $\Lambda$, which we write $\Lambda\prec_L \Lambda_2$, if there are lifts $\tilde{S}$ and $\tilde{S}_2$ of $\textrm{Fill}(S_{\Lambda})$ and $\textrm{Fill}(S_{\Lambda_2})$ respectively and $\tilde{\delta}$ a lift of a $(\Lambda,\Lambda_2)$-path $\delta$ that satisfy:

\begin{itemize}
\item $\tilde{S}_2\subseteq L_{\tilde{S}}$,
\item $\tilde{\delta}^-\in \tilde{S}$, and $\tilde{\delta}^+\in \tilde{S}_2$.
\end{itemize}

The definition of $\Lambda\prec_R \Lambda_2$ is analogous. There are some important observations about this concept: First, these definitions make sense only for a fixed Conley decomposition and an orientation of the annular classes as given in the previous sub-section \ref{sec.pack}. If $\Lambda_1\prec\Lambda$ and $\Lambda$ is annular, it is easy to show that either $\Lambda_1\prec^L\Lambda$ or $\Lambda_1\prec^R\Lambda$. Both could happen at the same time, but only if the annulus $\textrm{Fill}(S_{\Lambda})$ is attracting. The analogous property holds for $\Lambda\prec\Lambda_2$ where $\Lambda$ is annular, where one of $\Lambda\prec_L\Lambda_2$ or $\Lambda\prec_R\Lambda_2$ takes place, and they are not exclusive if $\textrm{Fill}(S_{\Lambda})$ is repelling. On the other hand, if $\textrm{Fill}(S_{\Lambda})$ is neither attracting nor repelling the orientation of the heteroclinical connections at $\Lambda$ is forced: for instance, if $f$ moves both boundaries of $\textrm{Fill}(S_{\Lambda})$ to the right, we have that $\Lambda_1\prec\Lambda\prec\Lambda_2$ implies $\Lambda_1\prec^L\Lambda\prec_R\Lambda_2$.

If $\Lambda_1\prec\Lambda_2$ are annular pieces in the same annular package, we have $\Lambda_1\prec_L\Lambda_2$ if and only if $\Lambda_1\prec^R\Lambda_2$, but if they are in different packages there is no such relation.

If $\Lambda$ is annular and $\Lambda_0$, $\Lambda_1$ are non-trivial then $\Lambda_0\prec\Lambda_1\prec^L\Lambda$ implies $\Lambda_0\prec^L\Lambda$, which can be obtained as a consequence of the $\lambda$-lemma and the properties of Conley surfaces lifted to $\D$. The analogous statements for $\prec^R$, $\prec_L$ and $\prec_R$ also hold.


\subsection{Walls and forcing of heteroclinical connections}

As we mentioned before, unstable-stable paths and walls will be used to deduce heteroclinical connections between pieces. The main tool in that direction is Lemma \ref{cruce} below. First let us set up the situation: We consider $\Lambda_1\prec\Lambda_2$ and $\Lambda\prec\Lambda'$ non-trivial basic pieces of $f$, and let $\tilde{S}$ and $\tilde{S}'$ be lifts of $S = \textrm{Fill}(S_{\Lambda})$ and $S' = \textrm{Fill}(S_{\Lambda'})$ respectively. Denote by $\tilde{\Lambda}:=\tilde{S}\cap\pi^{-1}(\Lambda)$ and $\tilde{\Lambda}':=\tilde{S}\cap\pi^{-1}(\Lambda')$, which are $F$-invariant sets by Lemma \ref{l.conleycasesess0}.


\begin{lem}\label{cruce} In the situation above, suppose there is a $(\Lambda_1,\Lambda_2)$-wall $\gamma$ disjoint from $\Lambda\cup\Lambda'$, and a $(\Lambda,\Lambda')$-path $\delta$, that admit lifts $\tilde{\gamma}$ and $\tilde{\delta}$ with the following properties:

\begin{itemize}
\item $\tilde{\Lambda}\subset L_{\tilde{\gamma}}$,
\item $\tilde{\Lambda}'\subset R_{\tilde{\gamma}}$,
\item $\tilde{\delta}^-\in \tilde{\Lambda}$ and $\tilde{\delta}^+\in \tilde{\Lambda}'$.
\end{itemize}

Then $\Lambda_1\prec\Lambda'$ and $\Lambda\prec\Lambda_2$.

\end{lem}

We remark that we can interchange $L_{\tilde{\gamma}}$ with $R_{\tilde{\gamma}}$ in the statement.





\begin{proof} Let us write $\delta = \delta^u\delta^s$ and $\gamma = \gamma^u\gamma^s$. Consider the lifts $\tilde{\delta}^u$ and $\tilde{\delta}^s$ so that $\tilde{\delta} = \tilde{\delta}^u\tilde{\delta}^s$. Notice that $\tilde{\gamma}$ is a complete lift of a loop, so it is an infinite concatenation of alternating lifts of $\gamma^u$ and $\gamma^s$.

Since $\tilde{\gamma}$ separates $\D$ into $L_{\tilde{\gamma}}$ and $R_{\tilde{\gamma}}$ and $\tilde{\delta}$ begins on one of these sides and ends on the other, we see that $\tilde{\gamma}$ and $\tilde{\delta}$ must intersect. Using the fact that stable (or unstable) manifolds of different pieces cannot intersect, we have that either $\tilde{\delta}^u$ meets $\tilde{\gamma}$ in some lift of $\gamma^s$, implying that $\Lambda\prec \Lambda_2$, or that $\tilde{\delta}^s$ meets $\tilde{\gamma}$ in some lift of $\gamma^u$ and so $\Lambda_1\prec\Lambda'$. We want to see that both relations take place, and we will achieve it by changing $\tilde{\delta}$ for other convenient lifts of $(\Lambda,\Lambda')$-paths.

Let $\alpha = \text{dist}(\tilde{\gamma},\tilde{\Lambda}\cup\tilde{\Lambda}')$, which is strictly positive since  $\gamma$ and $\Lambda\cup\Lambda'$ are disjoint compact sets. For $n\in\Z$ we consider $\tilde{\delta}_n = F^n(\tilde{\delta})$. The fact that $\tilde{\Lambda}$ and $\tilde{\Lambda}'$ are $F$-invariant implies that $\tilde{\delta}_n^-\in\tilde{\Lambda}$ and $\tilde{\delta}_n^+\in\tilde{\Lambda}'$, so $\tilde{\delta}_n$ projects to a $(\Lambda,\Lambda')$-path $\delta_n$ which also satisfies the hypothesis. Since $\tilde{\delta}^s$ is an arc in a stable manifold, we have that for $n$ large enough the diameter of $\tilde{\delta}_n^s = F(\tilde{\delta}^s) $ is less than $\alpha$. This implies that $\tilde{\delta}_n^s$ cannot meet $\tilde{\gamma}$, recalling that $\tilde{\delta}_n^s{}^+\in\tilde{\Lambda}'$. So $\tilde{\delta}_n^u = F(\tilde{\delta}^u)$ must intersect $\tilde{\gamma}$, which may happen only in a lift of $\gamma^s$ and implies $\Lambda\prec \Lambda_2$.

The other relation can be obtained by the same argument using $\delta_{-n}$ for $n$ large enough.




\end{proof}


With stronger assumptions we can construct walls for the newly deduced heteroclinical connections. Let $S_i = \textrm{Fill}(S_{\Lambda_i})$ for $i=1,2$ and $S_{\ast} = S_1 \cup S_2$. In the first result of this nature, stated below, we need to assume that $\Lambda$ and $\Lambda'$ are different form $\Lambda_1$ and $\Lambda_2$, so they are disjoint from $S_{\ast}$.

\begin{lem}\label{espiral}
Assume the situation above, including the hypotheses of Lemma \ref{cruce}. Suppose also that
\begin{itemize}

\item There is a contractible fixed point $p\in\Lambda_1$ and a non-contractible periodic point $q\in\Lambda_2$.

\item $\gamma=\gamma^u\gamma^s$ with $\gamma^u\subset W^u(p,f)$ and $\gamma^s\subset W^s(q,f)$, and $\gamma$ is freely homotopic to $(q)$.

\item $S_1$ and $S_2$ are adjacent, i.e. $S_{\ast} = S_1 \cup S_2$ is connected.


\end{itemize}

Then there exists a $(\Lambda,\Lambda_2)$-wall $\eta$ contained in $S_{\ast}^{(n)}$ for some $n>0$. Moreover, we can take $\eta$ homotopic to $\gamma$ inside $S_{\ast}^{(n)}$.

There also exists a $(\Lambda_1,\Lambda')$-wall with the same conditions.

\end{lem}






\begin{proof}

By Lemma \ref{casi-soporte} there is $n_0$ with $\gamma\subset S^{(n_0)}$. Let $\tilde{S}_{\ast}$ be the lift of $S_{\ast}$ so that $\tilde{\gamma}\subset \tilde{S}_{\ast}^{(n_0)} $. Then we have

\begin{equation}\label{espiral1}
\tilde{\Lambda}\subset L_{\tilde{\gamma}}\setminus \tilde{S}^{(n_0)} \mbox{ and } \tilde{\Lambda}'\subset R_{\tilde{\gamma}}\setminus \tilde{S}^{(n_0)}
\end{equation}

and this also holds true if we increase $n_0$ or change $\gamma$ by a homotopic loop in $S^{(n_0)}$. Thus it can be assumed that $\gamma$ is as constructed in Lemma \ref{pared}. Recall the notations of that lemma, specifically that $q$ has period $m$ and $T$ is a deck transformation so that $T^k$ corresponds to the homotopy class of $(q)$. That construction lets us assume that $\gamma^u$ and $\gamma^s$ are arcs joining $x_0$ and $f^m(x_0)$ for some $x_0\in W^u(p,f)\cap W^s(q,f)$ (in the notation of Lemma \ref{pared}, we have $x_0 = \pi(x)$).

Consider the loops $\gamma_i = f^{mi}(\gamma)$ for $i\geq 0$. They are all $(\Lambda_1,\Lambda_2)$-walls homotopic to $\gamma$, and $\gamma_{i-1}$ intersects $\gamma_{i}$ in at least the point $x_i=f^{mi}(x_0)$. Write $\gamma_i = \gamma_i^u\gamma_i^s$ where $\gamma_i^u=f^{mi}(\gamma^u)\subset W^u(p,f)$ and $\gamma_i^s=f^{mi}(\gamma^s)\subset W^s(q,f)$. Notice that we can concatenate $\gamma_{i-1}^u$ with $\gamma_i^u$ at the point $x_i$, obtaining an arc of $W^u(p,f)$ that goes from $x_{i-1}$ to $x_{i+1}$. On the other hand we can concatenate $\gamma_i^s$ with $\gamma_{i-1}^s$ at $x_i$, which gives us an arc of $W^s(q,f)$ going from $x_{i+1}$ to $x_{i-1}$.

We define $\Gamma = \bigcup_{i=0}^{i=4} \gamma_i$ and the arcs $\Gamma^u = \gamma^u_0\cdots\gamma^u_4$ and $\Gamma^s = \gamma^s_4\cdots\gamma^s_0$ (we will see the reason to use five of the $\gamma_i$ later on).  Notice that $\Gamma = \Gamma^u\cup\Gamma^s$ and that $\Gamma^u$ and $\Gamma^s$ are sub arcs of $W^u(p,f)$ and $W^s(q,f)$ respectively. Since stable and unstable manifolds are transversal to each other we see that $\Gamma$ is a CW-complex of dimension 1, which is finite and connected. We can define $\textrm{Fill}(\Gamma)$ by adding the components of $\Sigma\setminus\Gamma$ that are disks, so $\textrm{Fill}(\Gamma)$ retracts by deformation to an essential simple closed curve with homotopy class corresponding to the deck transformation $T$. Observe as well that $\textrm{Fill}(\Gamma)\subset S^{(n_1)}$ for $n_1=n_0+4$.

Since $\Gamma$ has no vertices of valence 1, there are simple closed curves $b_0,b_1\subset\Gamma$ in the homotopy class of $T$ such that $\textrm{Fill}(\Gamma)$ is the closed region at the right of $b_0$ and the left of $b_1$, according to the orientation given by $\gamma$.

Let $\tilde{\Gamma}$ be the connected component of $\pi^{-1}(\Gamma)$ that contains $\tilde{\gamma}$, and $\tilde{b}_0$, $\tilde{b}_1$ the lifts of $b_0$ and $b_1$ in $\tilde{\Gamma}$. We then have that $\tilde{\Gamma}$, $\tilde{b}_0$ and $\tilde{b}_1$ are invariant by $T$, and accumulate in $\partial\D$ exactly at $\tilde{\gamma}^+$ and $\tilde{\gamma}^-$. Also notice that $B=\textrm{cl}[R_{\tilde{b}_0}]\cap \textrm{cl}[L_{\tilde{b}_1}]$ is the connected lift of $\textrm{Fill}(\Gamma)$ that contains $\tilde{\Gamma}$.

We are going to show now that $\textrm{Fill}(\Gamma)$ is an annulus, i.e. that $b_0\cap b_1 = \emptyset$. First we see that $b_0\cap b_1\subset\bigcap_{i=0}^4\gamma_i$: If $y\in b_0\cap b_1$ we consider $\tilde{y}\in \tilde{b}_0\cap\tilde{b}_1$ with $\pi(\tilde{y})=y$, which exists since $\tilde{b}_0$ and $\tilde{b}_1$ are $T$-invariant, and notice that $B\setminus\{\tilde{y}\}$ has two connected components which accumulate on $\partial\D$ in $\tilde{\gamma}^+$ and $\tilde{\gamma}^-$ respectively. For every $i=0,\ldots,4$ we have that $\tilde{\gamma}_i\subset B$ goes from $\tilde{\gamma}^-$ to $\tilde{\gamma}^+$, so they all must pass through $\tilde{y}$. By projecting we conclude $y\in\bigcap_{i=0}^4\gamma_i$. Next we show that $\bigcap_{i=0}^4\gamma_i=\emptyset$: Writing $\gamma_i = \gamma_i^u\cup\gamma^s_i$ we see that $\bigcap_{i=0}^4\gamma_i$ is union of the sets $\bigcap_{i=0}^4\gamma_i^{\mu_i}$ for each choice of $\mu_0,\ldots,\mu_4 \in \{u,s\}$. On the other hand, if $i_1,i_2,i_3 \in \{0,\ldots,4\}$ are different we have $\gamma_{i_1}^u\cap\gamma_{i_2}^u\cap\gamma_{i_3}^u=\emptyset$ and $\gamma_{i_1}^s\cap\gamma_{i_2}^s\cap\gamma_{i_3}^s=\emptyset$ by construction, since they are non-trivial intervals with disjoint interiors in $W^u(p,f)$ and $W^s(q,f)$ respectively. Thus all the sets of the form $\bigcap_{i=0}^4\gamma_i^{\mu_i}$ are empty, since there always must be $i_1,i_2,i_3$ different with $\mu_{i_1}=\mu_{i_2}=\mu_{i_3}$. This is the point of using five of the curves $\gamma_i$.

We turn now to show that $W^u(\tilde{\Lambda},F)$ crosses $B$. Since $b_0^k$ and $b_1^k$ are homotopic to $\gamma$ in $S^{(n_1)}$, we can deduce from Lemma \ref{homotopia} and equation \ref{espiral1} that $\tilde{\Lambda}\subset L_{\tilde{b}_0}$ and $\tilde{\Lambda}'\subset R_{\tilde{b}_1}$. Recall that $\tilde{\delta}$, as mentioned in Lemma \ref{cruce}, goes from $\tilde{\Lambda}$ to
$\tilde{\Lambda}'$. We proceed as in Lemma \ref{cruce} by taking $\tilde{\delta}_n = F^n(\tilde{\delta})$ for $n$ large enough so that $\tilde{\delta}_n^s = F^n(\tilde{\delta}^s)$ has less diameter than $$\alpha = \textrm{dist}(B,\tilde{\Lambda}\cup\tilde{\Lambda}') = \textrm{dist}(\textrm{Fill}(\Gamma),\Lambda\cup\Lambda')>0$$
So $\tilde{\delta}_n^u = F^n(\tilde{\delta}^u)$ goes from $\tilde{\Lambda}\subset L_{\tilde{b}_0}$ to a point in the $\alpha$-neighbourhood of $\tilde{\Lambda}'$, that is contained in $R_{\tilde{b}_1}$. Thus $\tilde{\delta}_n^u$ is an arc of $W^u(\tilde{\Lambda},F)$ that crosses $B$.

We can then consider a sub-arc $\mu$ of $W^u(\tilde{\Lambda},F)$ such that $\mu\subset B$ with $\mu^-\in\tilde{b}_0$ and $\mu^+\in\tilde{b}_1$. Then $B\setminus \mu$ consists on two components $B^+$ and $B^-$ that accumulate on $\partial\D$ in $\tilde{\gamma}^+$ and $\tilde{\gamma^-}$ respectively. Since $\Gamma^u\subset W^u(\Lambda_1,f)$ and unstable manifolds of different pieces do not meet, we see that $\mu$ cannot intersect $\pi^{-1}(\Gamma^u)$. On the other hand $\mu^-\in\tilde{b}_0 \subset \tilde{\Gamma}$ and $\Gamma = \Gamma^s\cup\Gamma^u$. So there is some connected lift $\tilde{\Gamma}^s_0$ of $\Gamma^s$ that contains $\mu^-$. We see that $$\pi^{-1}(\Gamma^s)\cap\tilde{\Gamma} = \bigcup_{n\in\Z} T^n(\tilde{\Gamma}^s_0) $$ since $T$ generates $\textrm{Stab}(\tilde{\Gamma})=\textrm{Stab}(B)$. The sets in the union, namely $\tilde{\Gamma}^s_n=T^n(\tilde{\Gamma}^s_0)$, are pairwise disjoint since $\Gamma^s$ is an arc with no self-intersection.

We aim to show that $\mu$ intersects either $\tilde{\Gamma}^s_k$ or $\tilde{\Gamma}^s_{-k}$. That is sufficient to finish the proof: Say that $\mu$ meets $\tilde{\Gamma}^s_k$ in a point $z$. Then we can concatenate the sub-arc of $\mu$ from $\mu^-$ to $z$ with the sub-arc of $\tilde{\Gamma}^s_k$ joining $z$ with $T^k(\mu^-)$, and let $\eta$ be the projection of that curve. It is straightforward to verify that $\eta$ is a $(\Lambda,\Lambda_2)$-wall, recalling that $\Gamma^s\subset W^s(\Lambda_2,f)$. It is homotopic to $\gamma$ since the endpoints of the lift differ by translation by $T^k$, and the homotopy can be made inside $\textrm{Fill}(\Gamma)\subset S_{\ast}^{(n_1)}$. If $\mu$ meets $\tilde{\Gamma}^s_{-k}$ a similar argument applies.

We now prove that $\mu$ intersects $\tilde{\Gamma}^s_{k}\cup\tilde{\Gamma}^s_{-k}$. We will show it by contradiction, so assume that $\mu\cap(\tilde{\Gamma}^s_{k}\cup\tilde{\Gamma}^s_{-k})=\emptyset$. Since $\mu$ intersects $\tilde{\Gamma}^s_{0}$, we have $\tilde{\Gamma}^s_{k}\subset B^+$ and $\tilde{\Gamma}^s_{-k}\subset B^-$. Consider the loop $\gamma_{\ast} = \gamma_0^u\gamma_1^u\gamma_1^s\gamma_0^s$, recalling the definition of the curves $\gamma_i$ to see this concatenation is well defined and that $\gamma_{\ast}$ begins and ends at $x_0\in\Gamma^s\cap\Gamma^u$. Notice as well that $\gamma_{\ast}$ is freely homotopic to $\gamma^2$. Let $\nu=\nu^u\nu^s$ be the lift of $\gamma_{\ast}$ that begins at $\tilde{x}_0\in \tilde{\Gamma}^s_{-k}$, where $\nu^u$ lifts $\gamma_0^u\gamma_1^u$ and $\nu^s$ lifts $\gamma_1^s\gamma_0^s$. (Remark: $\nu$ is a lift of $\gamma_{\ast}$ as a path, i.e. not a complete lift, as we are accustomed to consider for essential loops). Since $\gamma_{\ast}$ is in the homotopy class of $\gamma^2$, we get that $\nu$ ends at $T^{2k}(\tilde{x}_0)$ that belongs in $\tilde{\Gamma}^s_{k} = T^{2k}(\tilde{\Gamma}^s_{-k})$. Thus $\nu^s$ is contained in $\tilde{\Gamma}^s_{k}$ and $\nu^u$ goes from $\tilde{x}_0\in \tilde{\Gamma}^s_{-k}$ to some point in $\tilde{\Gamma}^s_{k}$. Recalling that by our assumption we had $\tilde{\Gamma}^s_{-k}\subset B^-$ and $\tilde{\Gamma}^s_{k}\subset B^+$, we have an arc $\nu^u\subset B$ joining a point in $B^-$ with one in $B^+$, implying that $\nu^u$ must meet $\mu$. This is absurd, since they are sub-arcs of unstable manifolds of different pieces, namely $\mu\subset W^u(\tilde{\Lambda},F)$ and $\nu^u\subset W^u(\tilde{\Lambda}_1,F)$.

This finishes the proof, the construction of the $(\Lambda_1,\Lambda')$-wall being analogous.

\end{proof}

We would like to improve Lemma \ref{espiral} to include the case when $\Lambda'$ agrees with $\Lambda_2$. In order to achieve our result in that direction we need $\Lambda_1$ and $\Lambda_2$ to be in the hypotheses of Proposition \ref{dicotomia}. That resembles the conditions in Lemma \ref{espiral} in that $S_1$ and $S_2$ are adjacent, $\Lambda_1$ has a contractible fixed point and $\Lambda_2$ a non-contractible periodic point. They are stronger in the sense that $S_2$ must be an annulus and $\rho_{\Lambda_2}(f)$ must be a point, but we do not need to assume $\Lambda_1\prec\Lambda_2$. That is because we will arrive at a weaker conclusion than that of Lemma \ref{espiral}, which will also be true under option 1 in Proposition \ref{dicotomia}.
We assume that $b=S_1\cap S_2$ is the left boundary of $S_2$, according to the orientation of the annular class of $[\Lambda_2]$, so condition (d) in Proposition \ref{dicotomia} says that $f(b)$ lies at the left of $b$.

\begin{lem} \label{espiral2} Assume $\Lambda_1$ and $\Lambda_2$ are as in Proposition \ref{dicotomia}, where $b$ is the left boundary of $S_{\Lambda_2}$. Let $\Lambda$ be a non-trivial basic piece with $\Lambda\prec^L \Lambda_2$. Then either $\Lambda\preceq\Lambda_1$ or there is a $(\Lambda,\Lambda_2)$-wall in the homotopy class of $b$.
\end{lem}

Clearly we could interchange the left with the right in the statement if $b$ was the right boundary of $S_2$.

\begin{proof}

We assume $\Lambda\neq \Lambda_1$, otherwise the result is trivial. First apply Proposition \ref{dicotomia} to $\Lambda_1$ and $\Lambda_2$. If option 1 holds then Lemma \ref{pared} gives us a $(\Lambda_1,\Lambda_1)$-wall in the homotopy class of $b$, which has $\Lambda_2$ to its right since it is contained in some $S_1^{(n)}$. Since $\Lambda\prec^L\Lambda_2$  we get a $(\Lambda,\Lambda_2)$-path that enters $S_2$ through $b$, and we can apply Lemma \ref{cruce} to get that $\Lambda\prec\Lambda_1$. So from now on we can assume $\Lambda_1\prec\Lambda_2$, that is option 2 in Proposition \ref{dicotomia}. Moreover, we can assume that option 1 does not hold, i.e. that $\Lambda_1$ has no rotation in the direction of $[b]$.

Assuming that $\Lambda_1\prec\Lambda_2$ puts us in the hypotheses of Lemma \ref{pared} with $S_1$ and $S_2$ adjacent, and this allows us to make the constructions for the proof of Lemma \ref{espiral}. In particular, recall the annulus $\textrm{Fill}(\Gamma)$ and its lift $B$, and let $n_0$ be such that $\textrm{Fill}(\Gamma)$ is contained in $S_{\ast}^{(n_0)}$. Let $\tilde{b}$ be the lift of $b$ in $\tilde{S}_{\ast}$. Since $\Lambda\prec^L\Lambda_2$ we can choose a lift $\tilde{\Lambda}$ of $\Lambda$ so that $W^u(\tilde{\Lambda},F)$ meets $W^s(\tilde{\Lambda}_2,F)$  and $\tilde{\Lambda}\subset L_{\tilde{b}}$. Notice that then $\tilde{\Lambda}\subset L_{\tilde{b}}\setminus S_{\ast}^{(n_0)}$, which is at the left of $B\subset \tilde{S}_{\ast}^{(n_0)}$. If $W^u(\tilde{\Lambda},F)$ meets $R_{\tilde{b}}\setminus \tilde{S}_{\ast}^{(n_0)}$, we see that $W^u(\tilde{\Lambda},F)$ must cross $B$ and so we can use the argument of Lemma \ref{espiral} to get a $(\Lambda,\Lambda_2)$-wall in the class of $b$.

In the rest of the proof we assume that $W^u(\tilde{\Lambda},F)\subset L_{\tilde{b}}\cup\tilde{S}_2^{(n_0)}$ and also that $\Lambda\not\prec\Lambda_1$, as this is the situation we have not covered yet. Let $\tilde{q}\in \tilde{S}_2$ be a lift of a periodic point $q\in\Lambda_2$, such that $F^m(\tilde{q}) = T^k(\tilde{q})$ where $T$ is a deck transformation generating $\textrm{Stab}(\tilde{S}_2)$ as usual. Consider also $b'$ the right boundary of the annulus $S_2^{(n_0)}$ and $\tilde{b}'$ the corresponding lift in $\tilde{S}_2^{(n_0)}$. Let $C$ be the band between $\tilde{b}$ and $\tilde{b}'$, i.e. $C=\textrm{cl}[R_{\tilde{b}}\cap L_{\tilde{b}'}]$, and observe it contains $\tilde{S}_2$ and is contained in $\tilde{S}_2^{(n_0)}$. We can construct a curve $\alpha$ contained in $W^s(\tilde{q},F)\cup W^u(\tilde{q},F)$ that crosses $C$, just as we did when proving Claim 2 in the proof of Proposition \ref{dicotomia}. Also recall that the beginning arc of $\alpha$, going from $\tilde{b}$ to $\tilde{q}$, must belong to $W^s(\tilde{q},F)$, while the arc of $\alpha$ form $\tilde{q}$ to $\tilde{b}'$ may belong either to $W^s(\tilde{q},F)$ or to $W^u(\tilde{q},F)$. In any case, any two points in $\alpha\setminus W^u(\tilde{q},F)$ are connected by an arc in $\alpha\cap W^s(\tilde{q},F)$. Since unstable manifolds of different pieces are disjoint, we see that $W^u(\tilde{\Lambda},F)$ can only meet $T^n(\alpha)$ in $W^s(T^n(\tilde{q}),F)$ for any $n$. Because of this, our objective is to find an arc in  $W^u(\tilde{\Lambda},F)$ that meets two disjoint translates of $\alpha$ by $T$, as that would give us a $(\Lambda,\Lambda_2)$-wall in the class of $b$.

By compactness of $\alpha$, there is $k_0$ such that the translates $\alpha_j=T^{jk_0}(\alpha)$ are pairwise disjoint for $j\in\Z$. Each $\alpha_j$ splits $C$ into two components, that accumulate in $\partial\D$ respectively at the points $\xi_+$ and $\xi_-$, the attracting and repelling point of $T$. Choose an arc $\nu\subset W^u(\tilde{\Lambda},F)$ that begins in $\tilde{\Lambda}$ and ends at a point $x\in W^u(\tilde{\Lambda},F)\cap W^s(\tilde{q},F)$ that is very close to $\tilde{q}$, in particular with $x\in \tilde{S}_2$. Then for $n$ large we have that $F^n(x)$ tends to $\xi_{+}$. Notice that if $F^n(\nu)\cap \tilde{b}$ remains in a fixed compact set for all $n$, then we can conclude as follows: We recall that $$F^n(\nu)\subset W^u(\tilde{\Lambda},F) \subset L_{\tilde{b}}\cup\tilde{S}_2^{(n_0)}=L_{\tilde{b}}\cup C$$ for all $n$, so there is an arc in $F^n(\nu)\cap C$ joining $F^n(x)$ with a point in $F^n(\nu)\cap \tilde{b}$. If $n$ is large enough, this arc in $C$ is going from a compact set to an arbitrary neighbourhood of $\xi_+$, thus it must intersect many of the curves $\alpha_j$ (as many as we want, shrinking the neighbourhood of $\xi_+$ and thus increasing $n$).

Thus it suffices to show that $F^n(\nu)\cap \tilde{b}$ is contained in some compact set for all $n$. This is what we do in case $\Lambda_1$ is curved. When $\Lambda_1$ is annular we shall prove a little less: we will find sub-arcs $\nu_n\subset\nu$ such that $F^n(\nu_n)\cap C$ contains an arc joining  a point in a fixed compact subset of $\tilde{b}$ with $F^n(x)$. Provided that said compact subset of $\tilde{b}$ is independent of $n$, this is also sufficient to conclude.

Let us assume first that $\Lambda_1$ is annular, thus in the same package as $\Lambda_2$ (by adjacency) and with $\rho_{\Lambda_1}(f)=\{0\}$ (as we assumed it has no rotation in the direction of $[b]$). Let $\tilde{p}$ be a lift in $\tilde{S}_1$ of a contractible fixed point $p\in \Lambda_1$. Denote by $c$ the left boundary of $S_1$, and $\tilde{c}$ the lift of $c$ contained in $\tilde{S}_1$. Notice that $S_1$ cannot be a repelling annulus, as $W^u(\Lambda,f)$ goes through it, so $f^{-n}(c)$ is the left boundary of $S_1^{(n)}$. We will do a similar construction to that of $\alpha$ but for the point $\tilde{p}$, and this time it will depend on $n$: On one hand let $\beta$ be the segment of $W^u(\tilde{p},F)$ that begins at $\tilde{p}$, ends in $\beta^+\in\tilde{b}$, and is contained in $\tilde{S}_1$. On the other hand, let $\beta_n$ be the segment of $W^s(\tilde{p},F)$ going from $\tilde{p}$ to some point in $F^{-n}(\tilde{c})$, with  $\beta_n\subset\tilde{S}_1^{(n)}$. We set $\gamma_n = \overline{\beta}_n\beta$, and notice that the band $B_n=S_1^{n}\cap L_{\tilde{b}}$, that is the band between  $F^{-n}(\tilde{c})$ and $\tilde{b}$, is divided into two components by $\gamma_n$ each accumulating in one of the points $\xi_+$ and $\xi_-$.  By compactness of $\nu$ we see that there is $k_1$ so that $\nu\cap\tilde{b}$ is in the compact interval of $\tilde{b}$ between $T^{-k_1}(\beta^+)$ and $T^{k_1}(\beta^+)$, let us call this interval $K$, and notice it does not depend on $n$. Let $\nu_n$ be the sub-arc of $\nu$ that joins $x$ and a point in $F^{-n}(\tilde{c})$ and is contained in $B_n\cup C$ (if we orient $\nu$ so it starts at $\tilde{\Lambda}$, then $\nu_n^-$ is the last point of $\nu$ in $\nu\cap F^{-n}(\tilde{c})$). Since we are assuming $\Lambda\not\prec \Lambda_1$ we have that $W^u(\tilde{\Lambda},F)$ must be disjoint from every translate of $\gamma_n$ for all $n$. Thus $\nu_n\cap B_n$ is in the compact subset of $B_n$ that lies between $T^{-k_1}(\gamma_n)$ and $T^{k_1}(\gamma_n)$, that we call $D_n$.  Since $p$ is contractible, we have that $T^j(\tilde{p})$ is fixed by $F$ for all $j$, and so $F(T^j(\gamma_n))\cap L_{\tilde{b}} = T^j(\gamma_{n-1})$ for all $j$ and $n$. Applying this for $j=\pm k_1$ we get that $F^n(D_n)\cap L_{\tilde{b}} = D_0$ and in particular $F^n(D_n)\cap\tilde{b} = K$. This implies $F^n(\nu_n)\cap\tilde{b} \subset K$ for all $n$. Then $F^n(\nu_n)$ contains an arc in $C$ that joins a point in $K$ and $F^n(x)$, as we desired. This finishes the proof for $\Lambda_1$ annular.

Next we suppose $\Lambda_1$ is curved. Since $\tilde{\Lambda}$ is $F$-invariant and $F$ is the canonical lift, we have that $\tilde{\Lambda}$ is contained in exactly one component of $\D\setminus\tilde{S}_1$ that we call $V$. Let $\tilde{c}$ be the boundary component of $\tilde{S}_1$ that is adjacent to $V$. Observe that $\tilde{c}$ is different from $\tilde{b}$, and moreover $V$ and $\tilde{c}$ are contained in $L_{\tilde{b}}$. Then, since $S_1$ is not a disk or annulus, we have $\tilde{c}^-$ and $\tilde{c}^+$ disjoint from $\xi_-$ and $\xi_+$. This implies that for any $R>0$ there are neighbourhoods $W^-$ and $W^+$ of $\xi_-$ and $\xi_+$ such that $$\textrm{dist}(V,W^-\cup W^+)> R $$
Let $\kappa$ be as in Lemma \ref{l.saltos2} and $R_0=\max\{d(y,F(y)):y\in\D\}$. We take $R = (\kappa+1)R_0$ and set $W^-$ and $W^-$ accordingly. Consider $K=\tilde{b}\setminus (W^-\cup W^+)$, which is a compact interval of $\tilde{b}$. We want to prove that $F^n(\nu)\cap\tilde{b}$ is contained in $K$ for all $n$, noticing that $K$ does not depend on $n$. So we will show that $W^u(\tilde{\Lambda},F)\cap \tilde{b}$ is disjoint from $W^-\cup W^+$. To do so, we consider $y\in W^u(\tilde{\Lambda},F)\cap \tilde{b}$, and see that the backward orbit $\{F^{-j}(y): j\geq 0 \}$ can only meet $\tilde{S}_1$ and $V$, and moreover there is some $j_1$ so that $F^{-j}(y)\in \tilde{S}_1$ for $j=0,\ldots,j_1$ and in $F^{-j}(y)\in V$ for $j>j_1$. This is because of the properties of Conley surfaces: The backward orbit of $y$ cannot enter any attracting component of $\D\setminus\tilde{S}_1$, and from the repelling ones it cannot return, so it can only enter $V$. Also notice that $\mathcal{O}_{j_1}(y,F^{-1})$ is disjoint from the Conley surfaces of the trivial pieces in $\tilde{S}_1$, for otherwise it would not be possible for $y$ to belong in $\tilde{b}$ (for the attracting disks) or in $W^u(\tilde{\Lambda},F)$ (for the repelling ones). Thus by our assumption that $\Lambda\not\prec\Lambda_1$ and Lemma \ref{l.saltos2} we get that $j_1\leq\kappa$. On the other hand $\mathcal{O}_{j_1+1}(y,F^{-1})$ is an $R_0$-path between $y$ and $F^{j_1+1}(y)\in V$, thus $\textrm{dist}(y,V)\leq (j_1+1) R_0\leq R$ and so $y \notin W^-\cup W^+$ as we desired.

\end{proof}

We remark that the $(\Lambda,\Lambda_2)$-wall in Lemma \ref{espiral2} has the same properties as those in Lemma \ref{espiral}: it is also contained in $S_{\ast}^{(n)}$ for some $n$, and freely homotopic to $b^k$ inside $S_{\ast}^{(n)}$ for some $k\neq 0$.

\subsection{Heteroclinical connections between annular pieces}

The following is the main result of this section, namely the forcing result for heteroclinical connections we have been aiming for. We shall focus on what we need for Theorem \ref{mainthm}, that is the heteroclinical relations among annular pieces with non-trivial rotation set. For short, we say that a basic piece $\Lambda$ has {\em non-trivial rotation} if $\rho_{\Lambda}(f)\neq\{0\}$.


\begin{thm}\label{conexion cruzada}
Let $\Lambda$ and $\Lambda_{\ast}$ be annular basic pieces of $f$ with non-trivial rotation and belonging to different annular packages. Assume that
\begin{itemize}

\item $\Lambda\prec^L\Lambda'_{\ast}$ for some $\Lambda'_{\ast}\in[\Lambda_{\ast}]$ with non-trivial rotation, and
\item $\Lambda'\prec^L\Lambda_{\ast}$ for some $\Lambda'\in[\Lambda]$ with non-trivial rotation.
\end{itemize}
Then $\Lambda\prec\Lambda_{\ast}$.
The analogous results for connections of types $\prec^R$, $\prec_L$ and $\prec_R$ also hold.
\end{thm}


\begin{proof}

Let $\Lambda_1$ be the leftmost piece in $[\Lambda_{\ast}]$ that has non-trivial rotation. We denote $S_1=\textrm{Fill}(S_{\Lambda_1})$, which is an essential annulus contained in $S_{[\Lambda_{\ast}]}$, and let $b_0$ and $b_1$ be respectively the left and right boundaries of $S_1$. Notice that either $\Lambda_{\ast}$ agrees with $\Lambda_1$, or it lies to the right of $b_1$ in $S_{[\Lambda_{\ast}]}$, and the same reasoning applies to $\Lambda'_{\ast}$, since we are assuming they have non-trivial rotation. In any case both $\Lambda_{\ast}$ and $\Lambda'_{\ast}$ lie at the right of $b_0$ in $S_{[\Lambda_{\ast}]}$.

Let $\Lambda_0$ be the basic piece such that $S_{\Lambda_0}$ is adjacent to $S_1$ at the boundary $b_0$. It is certainly not trivial since $b_0$ is essential. Moreover, there are only two possibilities for $\Lambda_0$: Either it is a curved piece, which happens if $b_0$ is the left boundary of $S_{[\Lambda_{\ast}]}$, or it is an annular piece in the package $[\Lambda_{\ast}]$, which is at the left of $\Lambda_1$ and thus has trivial rotation. In any case we see that $\Lambda_0$ contains a contractible fixed point: if it is curved that is by Lefschetz index, and if it is annular because we have $\rho_{\Lambda_0}(f)=\{0\}$ in that case.

Let $S_0 = \textrm{Fill}(S_{\Lambda_0})$ and notice that $S_0\cup S_{[\Lambda_{\ast}]}$ is always connected and filled: If $\Lambda_0$ is annular we have $S_0\subset S_{[\Lambda_{\ast}]}$, and if $\Lambda_0$ is curved then $S_0$ and $S_{[\Lambda_{\ast}]}$ are adjacent at their common boundary $b_0$.

We take some connected lift $\tilde{P}_{\ast}\subset\D$ of $S_{[\Lambda_{\ast}]}$. We shall lift the objects we defined in the following manner: We consider $\tilde{\Lambda}_{\ast} = \pi^{-1}(\Lambda_{\ast})\cap \tilde{P}_{\ast}$ and $\tilde{\Lambda}'_{\ast} = \pi^{-1}(\Lambda'_{\ast})\cap \tilde{P}_{\ast}$. Take $\tilde{b}_0$ and $\tilde{b}_1$ as the lifts of $b_0$ and $b_1$ contained in $\tilde{P}_{\ast}$, and also $\tilde{S}_1$ as the lift of $S_1$ that is contained in $\tilde{P}_{\ast}$, which is the band between $\tilde{b}_0$ and $\tilde{b}_1$. Let $\tilde{S}_0$ be the lift of $S_0$ that meets $\tilde{b}_0$.

Since $[\Lambda]$ and $[\Lambda_{\ast}]$ are different annular packages, we have that $S_{[\Lambda]}$ is disjoint from $S_{[\Lambda_{\ast}]}$. The hypothesis that $\Lambda\prec^L \Lambda'_{\ast}$ gives us a $(\Lambda,\Lambda'_{\ast})$-path $\delta$ with a lift $\tilde{\delta}$ such that
\begin{equation}
\delta^-\in L_{\tilde{b}_0} \mbox{ and } \delta^+\in\tilde{\Lambda}'_{\ast}\subset R_{\tilde{b}_0}
\end{equation}
Analogously, $\Lambda'\prec^L \Lambda_{\ast}$ gives us a $(\Lambda',\Lambda_{\ast})$-path $\delta'$ with a lift $\tilde{\delta}'$ satisfying
\begin{equation}
\delta'^-\in L_{\tilde{b}_0} \mbox{ and } \delta'^+\in\tilde{\Lambda}_{\ast}\subset R_{\tilde{b}_0}
\end{equation}

We let $\tilde{P}$ and $\tilde{P}'$ be the lifts of $S_{[\Lambda]}$ containing $\delta^-$ and $\delta'^-$ respectively (notice that $\tilde{P}$ and $\tilde{P}'$ could agree or not). Consider also  $\tilde{\Lambda} = \pi^{-1}(\Lambda)\cap \tilde{P}$ and  $\tilde{\Lambda}' = \pi^{-1}(\Lambda')\cap \tilde{P}'$. Since $[\Lambda]\neq[\Lambda_{\ast}]$ and $\Lambda_0$ is either curved or in $[\Lambda_{\ast}]$, we get that $S_{[\Lambda]}$ and $S_0\cup S_{[\Lambda_{\ast}]}$ cannot meet in their interior, so the interior of $\tilde{P}\cup\tilde{P}'$ is disjoint from $\tilde{S}_0\cup \tilde{P}_{\ast}$. This gives us that

\begin{equation}
\tilde{\Lambda}\cup\tilde{\Lambda}'\subset L_{\tilde{b}_0}\setminus (\tilde{S}_0\cup \tilde{P}_{\ast})
\end{equation}

On the other hand, we have that either $\Lambda_{\ast}=\Lambda_1$ or $\tilde{\Lambda}_{\ast}\subset R_{\tilde{b}_1}$, and the same applies to $\Lambda'_{\ast}$.

\tiny
\begin{figure}[ht]\begin{center}

\centerline{\includegraphics[height=6.5cm]{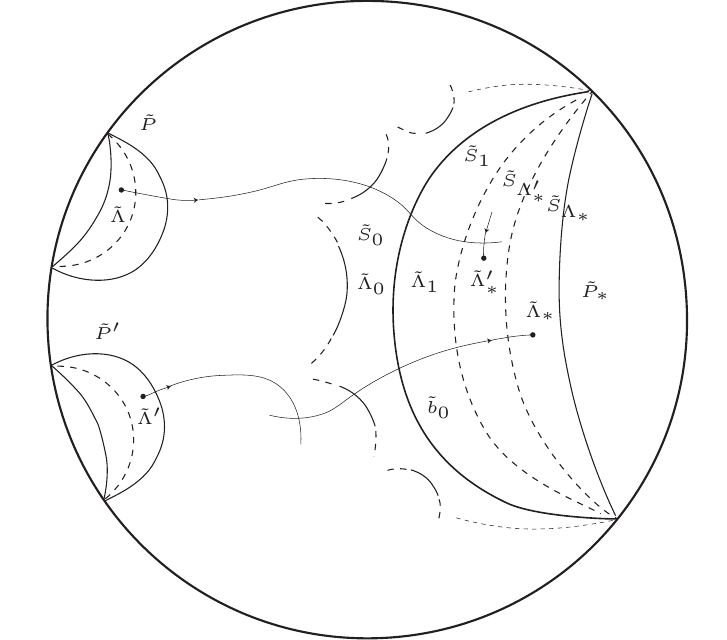}}
\caption{Elements considered for the proof of Theorem \ref{conexion cruzada}. We depict the situation when $\Lambda_0$ is curved. We want to deduce that $W^u(\tilde{\Lambda},F)\cap W^s(\tilde{\Lambda}_{*},F)$.}
\label{inter1}
\end{center}\end{figure}
\normalsize

The objective is to construct an unstable-stable wall that let us apply Lemma \ref{cruce} to obtain  the desired connection, i.e. that $\Lambda\prec\Lambda_{\ast}$. We shall split the proof into several cases, each one with a different method for constructing this wall. The main split follows the behaviour of $\rho_{\Lambda_1}(f) \neq \{0\}$, according to whether it is a point or an interval. The case when $\rho_{\Lambda_1}(f)$ is a point will split further into two sub-cases.

\smallskip
{\bf\flushleft Case I:} $\rho_{\Lambda_1}(f)$ is a non-trivial interval.

Use Lemma \ref{rota-intervalo} to obtain a $(\Lambda_1,\Lambda_1)$-wall $\gamma$ that is freely homotopic to a power of $b_0$, say $b_0^k$. By Lemma \ref{casi-soporte} we get that $\gamma\subset S_1^{(n)}$ for some $n$. Since $S^{(n)}$ is an annulus that retracts by deformation to $S_1$ we see that the free homotopy from $\gamma$ to $b_0^k$ can be made inside $S_1^{(n)}$ (which can also be obtained by Remark \ref{casi-soporte2}). By Lemma \ref{homotopia} we see that $\tilde{\Lambda}\subset L_{\tilde{\gamma}}$, where $\tilde{\gamma}$ is the lift of $\gamma$ that we get by lifting the homotopy from $b_0^k$ to $\gamma$ beginning at $\tilde{b}_0$. On the other hand, we either have $\Lambda'_{\ast} = \Lambda_1$ or $\tilde{\Lambda}'_{\ast} \subset L_{\tilde{\gamma}}$, also by Lemma \ref{homotopia}. So we obtain from Lemma \ref{cruce} that $$\Lambda\prec\Lambda_1\preceq \Lambda'_{\ast}$$

Reasoning in the same way with $\Lambda'$ and $\Lambda_{\ast}$ we get that $$\Lambda'\prec\Lambda_1\preceq \Lambda_{\ast}$$

Thus by transitivity we get $\Lambda\prec\Lambda_{\ast}$ as desired.

\smallskip
{\bf\flushleft Case II:} $\rho_{\Lambda_1}(f)=\{v\}$ for $v\neq 0$.

We see that $\Lambda_0$ and $\Lambda_1$ are in the conditions of Proposition \ref{dicotomia}: For condition (a), we already mentioned that $\Lambda_0$ contains a contractible fixed point, and condition (b) is the assumption of Case II. Their Conley surfaces are adjacent at $b_0$ by construction, and $f(b_0)$ lies at the right of $b_0$ for otherwise the connection realized by $\delta$ (or $\delta'$) would be impossible, thus we get conditions (c) and (b). Thus Proposition \ref{dicotomia} says that either $\Lambda_0$ has a periodic point in the class of $b_0$, or that we have $\Lambda_0\prec\Lambda_1$. We split the proof accordingly.

\smallskip
{\bf\flushleft Sub-case II A:} $\Lambda_0$ has a periodic point in the class of $b_0$.

In this case $\Lambda_0$ has non-trivial rotation. This may only happen if $\Lambda_0$ is curved, because of our construction, though we do not need this for the argument. Apply Lemma \ref{pared} to get a $(\Lambda_0,\Lambda_0)$-wall $\gamma$ that is freely homotopic to some $b_0^k$ inside $S_0^{(n)}$ for some $n$ (the last part comes from Remark \ref{casi-soporte2}. We proceed as in Case I, using $S_0^{(n)}$ instead of $S_1^{(n)}$ when using Lemma \ref{homotopia}, and we obtain
\begin{equation}
\Lambda\prec\Lambda_0\prec \Lambda'_{\ast} \mbox{ and } \Lambda'\prec\Lambda_0\prec \Lambda_{\ast}
\end{equation}
Which gives $\Lambda\prec\Lambda_{\ast}$ as well.

\smallskip
{\bf\flushleft Sub-case II B:} $\Lambda_0\prec\Lambda_1$.

Let $S_{\ast} = S_0\cup S_1$, and notice that $\Lambda_0$ and $\Lambda_1$ are in the conditions of Lemma \ref{pared} and Remark \ref{casi-soporte2}. Let $\gamma$ be the $(\Lambda_0,\Lambda_1)$-wall given by Lemma \ref{pared}, freely homotopic to $b_0^k$ in some $S_{\ast}^{(n)}$. We consider the lift $\tilde{\gamma}$ of $\gamma$ coming from lifting the homotopy beginning at $\tilde{b}_0$.

Now if $\Lambda_{\ast}=\Lambda'_{\ast}$ we are clearly finished, so we may assume one of them is not $\Lambda_1$. Let us assume $\Lambda'_{\ast}\neq \Lambda_1$, thus we have $$\tilde{\Lambda}'_{\ast}\subset R_{\tilde{b}_0}\setminus \tilde{S}_{\ast}$$ where $\tilde{S}_{\ast}=\tilde{S}_0\cup \tilde{S}_1$. Recalling that $\tilde{\delta}^-\in \tilde{\Lambda}\subset L_{\tilde{b}_0}\setminus \tilde{S}_{\ast}$ and $\tilde{\delta}^+\in\tilde{\Lambda}'_{\ast}$, we use Lemma \ref{homotopia} to see that $\tilde{\gamma}$ and $\tilde{\delta}$ are in the conditions of Lemmas \ref{cruce} and \ref{espiral}. Thus we get that $\Lambda\prec\Lambda_1$ and that there is a  $(\Lambda,\Lambda_1)$-wall $\nu$ freely homotopic to $b_0^k$ in some $S_{\ast}^{(n_1)}$. If $\Lambda_1 = \Lambda_{\ast}$ we already finished, so suppose this is not the case, and so we have $\tilde{\Lambda}_{\ast}\subset R_{\tilde{b}_0}\setminus \tilde{S}_{\ast}$ . Again we consider the appropriate lift $\tilde{\nu}$ of $\nu$, obtained from lifting the homotopy from $\tilde{b}_0$. Recalling that $\tilde{\delta}'^+\in \tilde{\Lambda}_{\ast}$, we see from Lemma \ref{homotopia} that $\tilde{\delta}'$ and $\tilde{\nu}$ are in the hypotheses of Lemma \ref{cruce}. Since $\delta'$ realizes the connection $\Lambda'\prec\Lambda_{\ast}$ and $\nu$ is a $(\Lambda,\Lambda_1)$-wall, we see that Lemma \ref{cruce} gives us $\Lambda\prec\Lambda_{\ast}$ as desired. (As well as $\Lambda'\prec\Lambda_1$, that we do not need).

If, on the other hand, we had $\Lambda_{\ast}\neq \Lambda_1$ and $\Lambda'_{\ast} = \Lambda_1$, we get that $\Lambda_0\prec\Lambda_{\ast}$ from Lemma \ref{cruce} applied to $\tilde{\delta}'$ and $\gamma$. Since $\Lambda_0$ and $\Lambda_1$ are in the conditions of Proposition \ref{dicotomia} and $\Lambda\prec^L \Lambda_1 = \Lambda'_{\ast}$, we can apply Lemma \ref{espiral2} to these pieces. In case $\Lambda\prec\Lambda_0$ we conclude by transitivity, since we also have $\Lambda_0\prec\Lambda_{\ast}$. Otherwise there is a $(\Lambda,\Lambda_1)$-wall freely homotopic to $b_0^k$ in some $S_{\ast}^{(n_2)}$. Applying Lemma \ref{cruce} to this wall and $\tilde{\delta}'$ gives us the desired relation $\Lambda\prec\Lambda_{\ast}$.

\end{proof}


\section{Bounding the complexity of the rotation set: Main theorem}\label{s.prin}

Here we state and prove our characterization of the rotation set $\rho(f)$ for $f\in\mathcal{A}_0(\Sigma)$ as a union of at most $c(g)=4\cdot 2^{5g-5}$ convex sets with some additional properties that further describe $\rho(f)$. An {\em essential decomposition} of $\Sigma$ will be a collection of sub-surfaces $\mathcal{S}=\{\Sigma_1,\ldots,\Sigma_k\}$ such that:
\begin{itemize}
\item Each $\Sigma_i$ is essential and filled,
\item $\Sigma=\Sigma_1\cup\cdots\cup\Sigma_k$,
\item If $i\neq j$ then $\Sigma_i$ and $\Sigma_j$ have disjoint interiors and are not isotopic to each other.
\end{itemize}
Under such conditions, different sub-surfaces in $\mathcal{S}$ can either be disjoint or adjacent at some of their boundary components. Notice that $k\leq 5g-5$, the maximum obtained from a pair of pants decomposition by fattening the dividing curves into annuli. For a subset $A\subseteq\mathcal{S}$ define $$\Sigma_A = \bigcup_{\Sigma_j\in A}\Sigma_j$$ and notice that $A$ can be recovered just by looking at $\Sigma_A$. Let us also define $$H_1(\Sigma_A|\mathcal{S};\R) = \sum_{\Sigma_i\in A} H_1(\Sigma_i;\R) $$ as a subspace of $H_1(\Sigma;\R)$. We will write $H_1(\Sigma|\mathcal{S};\R)$ instead of $H_1(\Sigma_{\mathcal{S}}|\mathcal{S};\R)$.

For $f\in\mathcal{A}_0(\Sigma)$, a Conley decomposition induces an essential decomposition as follows: For each annular package $[\Lambda]$ let $\Sigma_{[\Lambda]} = S_{[\Lambda]}$ as defined in sub-section \ref{sec.pack}, and for each curved piece $\Lambda'$ let $\Sigma_{[\Lambda']}=\textrm{Fill}(S_{\Lambda'})$, where the square brackets are there to unify notation. An essential decomposition obtained in that manner will be called an {\em essential Conley decomposition} for $f$.

If $\mathcal{S}=\{\Sigma_{[\Lambda]}\}_{\Lambda}$ is an essential Conley decomposition for $f\in\mathcal{A}_0(\Sigma)$ and $\mathcal{C}=(\Lambda_0\prec\cdots\prec\Lambda_n)$ is a chain of non-trivial basic pieces of $f$ we define the {\em support} of $\mathcal{C}$ as $$\textrm{supp}(\mathcal{C}) = \{\Sigma_{[\Lambda_j]}: j=1,\ldots,n \} \subseteq \mathcal{S} $$
and denote $\Sigma_{\mathcal{C}} = \Sigma_{\textrm{supp}(\mathcal{C})}$. For $A\subseteq\mathcal{S}$ we define $$C_A = \bigcup_{\textrm{supp}(\mathcal{C})=A}\textrm{conv}\left(\rho_{\mathcal{C}}(f),0\right) $$ where the union ranges over the chains $\mathcal{C}$ of non-trivial pieces that have support $A$. Notice that we ask the support of $\mathcal{C}$ to be exactly $A$, {\em not} contained in $A$. If there is no such chain we get $C_A = \emptyset$. Observe that for any $A\subseteq\mathcal{S}$ we have $C_A\subseteq \rho(f)$, since $\rho(f)$ is star-shaped at $0$ by Theorem \ref{t.starshape}. We also remark that if $\mathcal{C}$ contains curved pieces we already have $0\in \rho_{\mathcal{C}}(f)$, so $\rho_{\mathcal{C}}(f)=\textrm{conv}\left(\rho_{\mathcal{C}}(f),0\right)$.

Thus we have $$\rho(f) = \bigcup_{A\subseteq \mathcal{S}} C_A $$

We will show that for each $A\subseteq \mathcal{S}$ the set $C_A$ is a union of at most four convex sets, thus obtaining our main objective. We are now ready to state the main theorem.

\begin{theorem}\label{mainthm} Let $f\in\mathcal{A}_0(\Sigma)$ for $\Sigma$ an orientable closed surface of genus $g>1$. Then $\rho(f)\subset H_1(\Sigma;\R)$ is a union of at most $4\cdot 2^{5g-5}$ convex sets containing $0$.

Moreover, if $\mathcal{S}$ is an essential Conley decomposition for $f$ we can obtain such union as $$\rho(f) = \bigcup_{A\subseteq \mathcal{S}} C_A $$ where we have:
\begin{enumerate}
\item For $A\subset\mathcal{S}$, either $C_A=\emptyset$ or $C_A$ is a union of at most four convex sets containing $0$.
\item For $A\subset\mathcal{S}$, we have $C_A\subset H_1(\Sigma_A|\mathcal{S};\R)$. In particular $\rho(f)\subset H_1(\Sigma|\mathcal{S};\R)$. 
\item If $\mathcal{C}$ is a chain of basic pieces with $\textrm{supp}(\mathcal{C})=A$, then $\rho_{\mathcal{C}}(f)\subseteq C_A$.

\end{enumerate}

\end{theorem}
\smallskip

Clearly the bound  $4\cdot 2^{5g-5}$ is not sharp: There will be many subsets $A\subset\mathcal{S}$ with $C_A=\emptyset$. Also, if $\mathcal{C}'$ is a sub-chain of $\mathcal{C}$ we have $\rho_{\mathcal{C}'}(f)\subseteq \rho_{\mathcal{C}}(f)$, and pieces with rotation $0$ add nothing, so we could restrict the sets $A\subset\mathcal{S}$ to the supports of maximal chains of non-trivial pieces with non-zero rotation (in case $\rho(f)\neq\{0\}$). We will see that the number of convex sets in our writing of $C_A$ is at most four, but it could be less. On the other hand, a sharp bound must be exponential on $g$, as we see in the examples from section \ref{s.ejemploexp}.  

We turn now to the proof of Theorem \ref{mainthm}. Observe that the only statement that remains to be proved is point (1), since point (2) follows from Lemma \ref{l.conleycasesess} and point (3) comes directly from the definition of $C_A$. Thus we will devote the rest of this section to prove point (1), namely that when $A\subset\mathcal{S}$ is the support of some chain, then we can write $C_A$ as a union of at most four convex sets. First let us see that if $A=\{\Sigma_{[\Lambda]}\}$, i.e. $A$ consists of just one surface, then $C_A$ is convex: If $\Lambda$ is curved we have $C_A = \rho_{\Lambda}(f)$ and if $\Lambda$ is annular, then $C_A$ is a union of intervals containing $0$ in the $1$-dimensional subspace $H_1(\Sigma_{[\Lambda]};\R)$. Another easy case is when $A$ contains no annuli, for then the chain $\mathcal{C}$ with $\textrm{supp}(\mathcal{C})=A$ is unique and $C_A = \rho_{\mathcal{C}}(f)$. To deal with the general case we will need the following basic results about convex hulls.

\begin{lem} \label{join1} Let $K,C_1,\ldots,C_r \subset \R^N$ be convex sets so that $\bigcup_{i=1}^r C_i$ is convex. Then $$ \textrm{conv}\left(K\cup \bigcup_{i=1}^r C_i \right) = \bigcup_{i=1}^r \textrm{conv}(K\cup C_i). $$

\end{lem}

\begin{lem} \label{join2} Let $V = V_1\oplus\cdots\oplus V_k$ be a direct sum decomposition of a subspace of $\R^N$. For each $i=1,\ldots,k$, let $C^i_1,\ldots,C^i_{r_i}\subset V_i$ be convex sets so that $\bigcup_{j=1}^{r_i} C^i_j \subset V_i$ is convex. Then $$ \textrm{conv}\left(\bigcup_{i=1}^k\bigcup_{j=1}^{r_i} C^i_j \right) = \bigcup_{(j_1,\ldots,j_k)} \textrm{conv}\left(C^1_{j_1}\cup\cdots\cup C^k_{j_k} \right) $$ where the index $(j_1,\ldots,j_k)$ ranges over all the possible values of $j_i=1,\ldots,r_i$.
\end{lem}

The proof of these results is elementary. Our aim is to use Theorem \ref{conexion cruzada} in conjunction with these lemmas, so we need a definition of {\em marked} support of a chain, that accounts for the orientation of the heteroclinical connections at the annular pieces. As we remarked in sub-section \ref{sec.oriprec}, such orientations are forced except at the first and last pieces, in case they are annular and have filled in Conley surfaces that are repelling and attracting respectively. With this in mind, we define the {\em initial marking} of a chain $\mathcal{C}$ as follows: If $\mathcal{C}=(\Lambda_0\prec\cdots\prec\Lambda_n)$ is a chain we let $X_{\mathcal{C}}=L$ if
\begin{itemize}
\item[I1.] $\Sigma_{[\Lambda_0]}$ is a repelling annulus.
\item[I2.] $\mathcal{C}$ is not contained in $\Sigma_{[\Lambda_0]}$.
\item[I3.] $\Lambda_0\prec_L\Lambda_1$.
\end{itemize}
Define $X_{\mathcal{C}}=R$ as the same with $\prec_{R}$. Set $X_{\mathcal{C}}=0$ for any other case. We define the {\em final marking} $Y_{\mathcal{C}}$ analogously: Put $Y_{\mathcal{C}}=L$ if
\begin{itemize}
\item[F1.] $\Sigma_{[\Lambda_n]}$ is an attracting annulus.
\item[F2.] $\mathcal{C}$ is not contained in $\Sigma_{[\Lambda_n]}$.
\item[F3.] $\Lambda_{n-1}\prec^L\Lambda_n$.
\end{itemize}
and let $Y_{\mathcal{C}}=R$ if the same conditions hold but with $\prec^R$, and $Y_{\mathcal{C}}=0$ if  any of those conditions fail. The {\em marked support} of $\mathcal{C}$ is then the triple $$\textrm{Msupp}(\mathcal{C}) = (\textrm{supp}(\mathcal{C}),X_{\mathcal{C}},Y_{\mathcal{C}}) $$
For $A\subseteq\mathcal{S}$ and $X,Y\in \{L,0,R\}$ we define $$C_{A,X,Y} = \bigcup_{\textrm{Msupp}(\mathcal{C})=(A,X,Y)} \textrm{conv}(\rho_{\mathcal{C}}(f),0)$$

Where $C_{A,X,Y}=\emptyset$ if there is no chain with $(A,X,Y)$ as marked support. This restricts the possible combinations: Notice that if $A=\textrm{supp}(\mathcal{C})$ for a chain $\mathcal{C}=(\Lambda_0\prec\cdots\prec\Lambda_n)$ then $A$ can contain at most one repelling and one attracting annuli, and in that case they must be $\Sigma_{[\Lambda_0]}$ and $\Sigma_{[\Lambda_n]}$ respectively. Thus if $A\subset\mathcal{S}$ contains no repelling annulus, or is just one surface, it must be $X=0$ in order to have a non-empty $C_{A,X,Y}$. On the other hand, if $A$ contains a repelling annulus and is not just one surface, any chain with support $A$ verifies conditions I1 and I2 above, thus $X$ can be $L$ or $R$, but not $0$. The same reasoning applies to the marking $Y$ with regard to a repelling annulus in $A$. Therefore $C_A$ is a union of either one, two or four of the sets $C_{A,X,Y}$.

It only remains to show that $C_{A,X,Y}$ is convex when non-empty. Assume there exists a chain $\mathcal{C}$ with $\textrm{Msupp}(\mathcal{C})=(A,X,Y)$. Write $A = A_0\cup A_1$  where $A_0$ contains the annuli and $A_1$ the surfaces that are not annuli. The sub-chain $\mathcal{C}_1\subset\mathcal{C}$ that contains all the curved pieces of $\mathcal{C}$ is contained in every chain with support $A$. Let $K = C_{A_1} = \rho_{\mathcal{C}_1}(f)$.

On the other hand, write $A_0 = \{\Sigma_{[\Lambda_1]},\ldots,\Sigma_{[\Lambda_k]}\}$. For each $i=1,\ldots,k$ we consider the set $\mathcal{L}_i = \{\Lambda^i_1,\ldots,\Lambda^i_{r_i}\}$ of annular basic pieces that belong to the package $[\Lambda_i]$, have non-trivial rotation, and appear in some chain with marked support $(A,X,Y)$. Observe that every such chain has the same orientations in the heteroclinical relations among annular pieces of different packages, so we can apply Theorem \ref{conexion cruzada} to get that for every choice of $(j_1,\ldots,j_k)$ where $j_i\in\{1,\ldots,r_i\}$ we have $$\Lambda^1_{j_1} \prec \cdots \prec \Lambda^k_{j_k} $$ thus we can write \begin{equation} \label{eqmain1}
C_{A,X,Y} = \bigcup_{(j_1,\ldots,j_k)} \textrm{conv}\left(K\cup \rho_{\Lambda^1_{j_1}}(f) \cup \cdots \cup \rho_{\Lambda^k_{j_k}}(f)\right)
\end{equation}

where the union is taken over {\em all} the possible choices of $(j_1,\ldots,j_k)$. That allows us to apply Lemma \ref{join2} where $V_i = H_1(\Sigma_{[\Lambda_i]};\R)$, that are in direct sum because the annular packages $[\Lambda_i]$ are different, and $C^i_j = \textrm{conv}\left( \rho_{\Lambda^i_j}(f),0\right)$, that have convex union over the same annular package. Then Lemma \ref{join2} gives us that $$\bigcup_{(j_1,\ldots,j_k)} \textrm{conv}\left(C^1_{j_1}\cup\cdots\cup C^k_{j_k} \right) $$ is convex, and we may apply Lemma \ref{join1} to conclude that the union in equation \ref{eqmain1} is convex. This concludes the proof of Theorem \ref{mainthm}.

\section{Rotation sets with non-empty interior}\label{s.prinint}

In this section we provide an application of our results to the geometry of the rotation set.

\begin{thm}\label{t.thmintsec}

Assume that $f\in\mathcal{A}_0(\Sigma)$ has $\textrm{int}[\rho(f)]\neq\emptyset$. Then $\rho(f)$ is convex.

\end{thm}

\begin{proof}

Fix a Conley decomposition for $f$, from which we classify the basic pieces as trivial, annular,
or curved as done before, and let $\mathcal{S}$ be the essential decomposition associated to this Conley decomposition. Recall that an essential curve $\gamma$ in $\Sigma$ induces a functional $i_{\gamma}: H_1(\Sigma;\R)\to \R$ by linear extension of the intersection number with $\gamma$, and that $i_{\gamma}$ is not trivial if $[\gamma]\neq 0$ in homology. If $\Lambda$ is a non-trivial basic piece and $\gamma$ is any boundary curve of $\Sigma_{[\Lambda]}$, we have that $$\rho(f)\subset H_1(\Sigma|\mathcal{S};\R)\subseteq \ker(i_{\gamma})$$ so we infer that $\gamma$ must be non-separating, i.e. that $[\gamma]=0$, for otherwise $\ker(i_{\gamma})$ would have codimension 1, thus empty interior. We see then that if $\Sigma_{[\Lambda]}$ has no genus, i.e. is homeomorphic to a sphere with disks removed, then $H_1(\Sigma_{[\Lambda]}) = H_1(\partial\Sigma_{[\Lambda]}) = \{0\}$, so $\Lambda$ does not contribute towards the rotation set. This includes all the annular pieces. Recalling Proposition \ref{sacar.triviales}, we may write

\begin{equation}\label{e.rhochain}\rho(f)=\bigcup_{\mathcal{C}}\rho_{\mathcal{C}}(f)\end{equation}
where $\mathcal{C}$ ranges over all maximal chains of $\mathcal{G}_f$ containing
only curved basic pieces whose Conley surfaces have genus. Since $\rho(f)$ has non-empty interior, Baire's Theorem asserts that for some maximal chain $\mathcal{C}_1$ of curved basic pieces we must have $\textrm{int}[\rho_{\mathcal{C}_1}(f)]\neq\emptyset$. We will show that $\rho(f) = \rho_{\mathcal{C}_1}(f)$, which is convex, so that will conclude the proof.



In fact, we prove that $\mathcal{C}_1$ contains all the curved pieces with non-trivial rotation. We proceed by contradiction, supposing there is a curved piece $\Lambda$ which is not contained in $\mathcal{C}_1$ and has non-trivial rotation, in particular $\Sigma_{[\Lambda]}$ must have genus. On the other hand, $\Sigma_{[\Lambda]}$ and $\Sigma_{\mathcal{C}_1}$ can only intersect at their boundary, which is trivial in homology, so $H_1(\Sigma_{\mathcal{C}_1}) \cap H_1(\Sigma_{[\Lambda]}) = \{0\}$. But we had $H_1(\Sigma_{[\Lambda]})\neq\{0\}$, so $H_1(\Sigma_{\mathcal{C}_1})$ cannot have full dimension, and then $\rho_{\mathcal{C}_1}(f)\subset H_1(\Sigma_{\mathcal{C}_1})$ cannot have interior, which is absurd.

\end{proof} 
\section{Examples}\label{s.ej}

We provide here a list of relevant examples concerning our main result Theorem \ref{mainthm}.
They are based on J. Kwapisz constructions \cite{Kw1}, where axiom A diffeomorphisms on
$\mathbb{T}^2$ having prescribed rational polygons as rotation sets are obtained. The key point in these examples is to construct horseshoe-like invariant sets whose rotation set realizes a prescribed polygon, and then extend the dynamics in the complement avoiding new rotational information. In picture \ref{toro}
we represent this situation, where the horseshoe-like basic piece named by $\Lambda$ contains three points  $a,b,\ c$ having rotation vectors $(0,0),(1,0),\ (1,1)$ respectively for some lift $F$ of $f$, which implies that the triangle generated by these vectors is contained in the possibly bigger convex set $\rho(F)$. As shown in the picture, the basic piece $\Lambda$ is always related to a sink $p$ (and the same holds for some source). Moreover, the example can be constructed
so that $\Lambda$ is the unique saddle basic piece. The disk $D$ in the picture is assumed
to be mapped into its own interior.

\tiny
\begin{figure}[ht]\begin{center}
\centerline{\includegraphics[height=5cm]{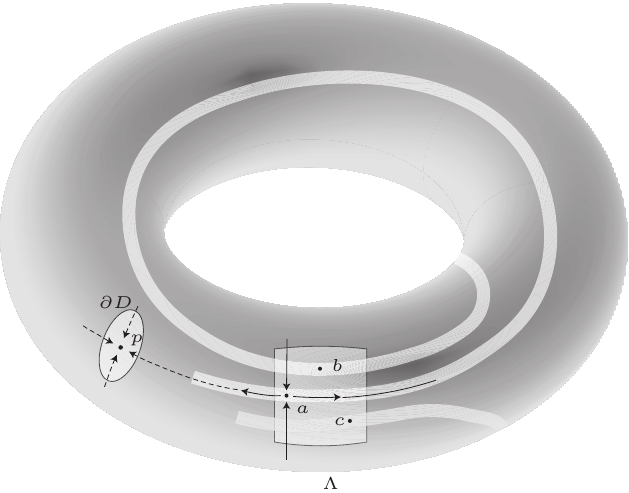}}
\caption{Representation of Kwapisz's example.}
\label{toro}
\end{center}\end{figure}

\normalsize

\subsection{Non-convex rotation set}

In this paragraph we present an example of a homeomorphism on the double-torus $\Sigma_2$ having a non convex rotation set, in contrast to the Misiurewicz-Ziemian result for $\mathbb{T}^2$. For this, we modify the example above so that the disk $D$ remains invariant. We call this example by $f_1$ and consider it defined on a torus $T$, and consider a copy of its inverse $f_2$ defined on a second torus $T'$. As mentioned above, the maps $f_1$, $f_2$ have unique saddle basic pieces $\Lambda_{f_1},\ \Lambda_{f_2}$ respectively, whose rotation sets equal the rotation set of the respective maps and have non-empty interior.

\smallskip

Then we can consider the connected sum $\Sigma_2$ of $T$ and $T'$ by gluing the boundary of the invariant disks, and the maps $f_1,f_2$ induce a natural homeomorphism $f:\Sigma_2\to\Sigma_2$ represented in the following picture.

\tiny

\begin{figure}[ht]\begin{center}

\centerline{\includegraphics[height=4cm]{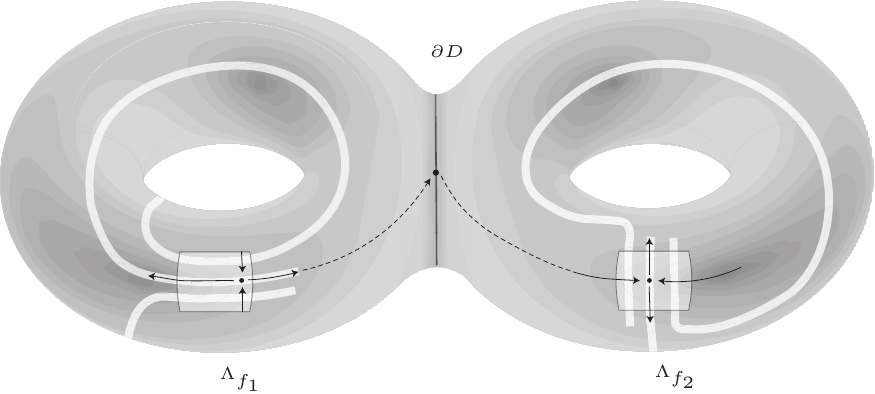}}
\caption{An example having non-convex rotation set.}
\label{toro1}
\end{center}\end{figure}

\normalsize

The two connected components that we called $T,T'$ of $\Sigma_2\setminus \partial D$ are
$f$-invariant, and we can name them so that $\rho_{\Lambda_{f_1}}(f)$ is contained in $H_1(T)$ and
$\rho_{\Lambda_{f_2}}(f)$ is contained in $H_1(T')$. These sets, $\rho_{\Lambda_{f_1}}(f)$ and $\rho_{\Lambda_{f_2}}(f)$, are, in light of Lemma \ref{l.conleycasesess}, two convex sets of full dimension in $H_1(T)$ and $H_1(T')$ respectively, which are complementary sub-spaces of dimension two in $H_1(\Sigma_2)$. Furthermore, by construction we have that
$$\rho(f)=\rho_{\Lambda_{f_1}}(f)\cup\rho_{\Lambda_{f_2}}(f),$$
otherwise one could find rotation vectors for the Kwapisz example which are not realized on the saddle basic piece. Thus we conclude that $\rho(f)$ is not convex.

\subsection{An example having a full-dimensional Polyhedra as rotation set}

We now modify our last example to obtain $g\in\textrm{Homeo}_0(\Sigma_2)$ having
a full dimensional rational polyhedra as rotation set. For this consider the map $f$ of the
last example. We can assume that the construction of $f$ implies a decomposition of its non-wandering set given by the two basic pieces $\Lambda_1=\Lambda_{f_1},\Lambda_2=\Lambda_{f_2}$ and a family of sinks
and sources given by contractible fixed points, which in particular have null rotation vectors.

\smallskip

Then composing with a time-one map $g_1$ of a flow supported on an annular neighbourhood of $\partial D$, we can construct a map $g=g_1\circ f\in\textrm{Homeo}_0(\Sigma_2)$ being a fitted axiom A diffeomorphism so that

\begin{itemize}

\item $g(\partial D)\subset T'$,

\item $\Omega(g)=\Omega(f)\setminus \partial D$,

\item $\Lambda_1\prec\Lambda_2$.

\end{itemize}

\tiny

\begin{figure}[ht]\begin{center}
\centerline{\includegraphics[height=3.5cm]{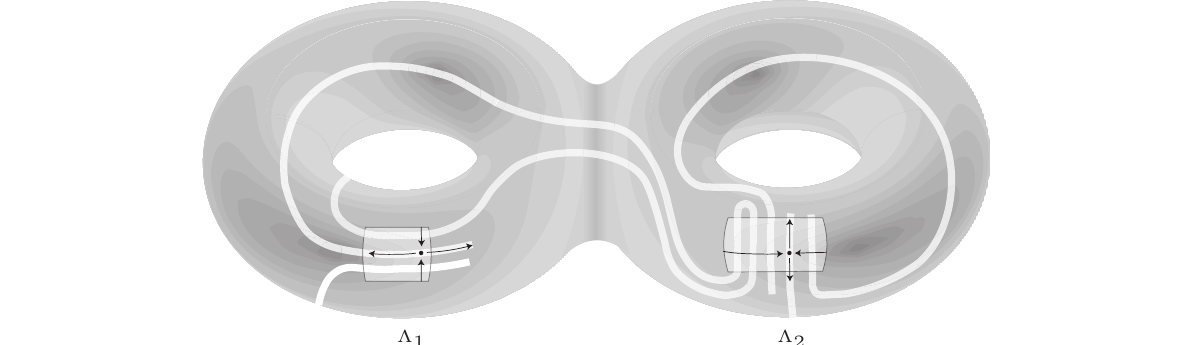}}
\caption{The map $g$ having a full dimensional rational polyhedra as rotation set.}
\label{bitoro2}
\end{center}\end{figure}

\normalsize

The map $g$ is represented in Figure \ref{bitoro2}. Thus in light of Theorem \ref{t.rhorhochain} we have that
$$\rho(g)=\textrm{conv}(\rho_{\Lambda_1}(g)\cup\rho_{\Lambda_2}(g))$$
which is a full dimensional rational polyhedra in $\R^4$.

\subsection{An example concerning the intersection of the convex blocks of $\rho(f)$}

As presented in our main result Theorem \ref{mainthm}, the rotation set of $f\in\mathcal{A}_0(\Sigma_g)$ is given by a union of convex sets
$C_1,\ldots,C_N\subset \R^{2\textrm{g}}$ all of them containing $0$ where $N$ is bounded by $\textrm{c}(\textrm{g})=2^{5\textrm{g}-3}$. Moreover, each $C_i$ can be characterized as the maximal elements with respect to the inclusion in the family of convex sets included in $\rho(f)$. We
call them the basic convex blocks of $\rho(f)$.

Our next example shows that these basic blocks can have a non-trivial intersection, for instance intersecting in a full dimensional set relative to each basic block. We support the example in Figure \ref{bitoro3} where we represent the map $h$ that we want to construct.


\tiny

\begin{figure}[ht]\begin{center}
\centerline{\includegraphics[height=3.5cm]{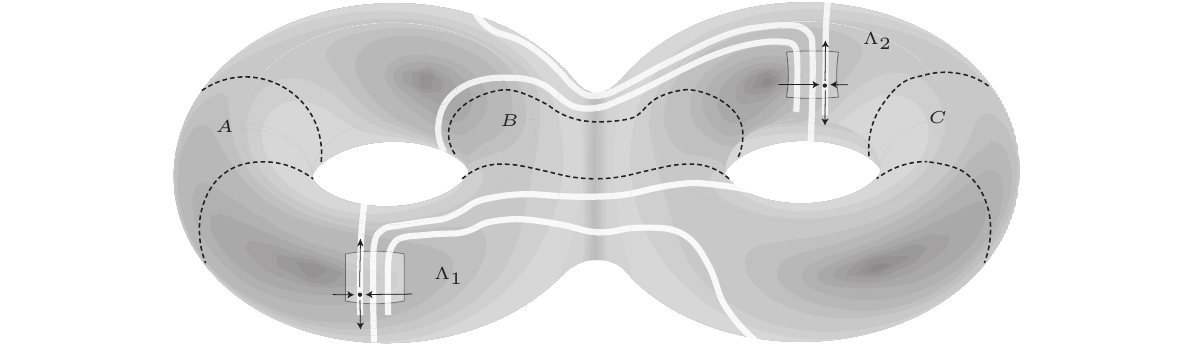}}
\caption{The map $h$ having non-trivial intersection between blocks.}
\label{bitoro3}
\end{center}\end{figure}
\normalsize

The axiom A diffeomorphism $h$ is intended to have three annular regions $A,B,C$, as shown in the figure, whose images under $h$ are contained in their own interior, defining respectively the attractors $\Omega_A,\Omega_B,\Omega_C$. The two connected
components of $\Sigma_2\setminus(A\cup B\cup C)$, denoted by $S_1,S_2$, contain horseshoe-like basic pieces $\Lambda_1,\Lambda_2$ which carry all the rotational information on each surface given by the convex rotation sets $C_1,C_2$ contained in $V=\textrm{H}_1(S_1)=\textrm{H}_1(S_2)$, which are full dimensional with respect to $V$. Due to the existence of the attracting annular regions one can choose $C_1$ and $C_2$ independently. This last construction can be achieved
by the same techniques employed by Kwapisz for the toral case.

Further, one can manage the construction so that the points of $\Omega_A,\Omega_B,\Omega_C$
which are heteroclinically related to either $\Lambda_1$ or $\Lambda_2$ have null rotation
vector. Thus in light of Theorem \ref{t.rhorhochain} we have that the rotation set of $h$ is given by
$$\rho(h)=C_1\cup C_2 \cup I_A\cup I_B \cup I_C $$
where
\begin{itemize}

\item $C_1$ and $C_2$ are full dimensional convex set of $V$ containing zero, whose interiors intersect,

\item $I_A$, $I_B$ and $I_C$ are intervals contained in $\textrm{H}_1(A),\textrm{H}_1(B),
\ H_1(C)$ respectively, containing zero.

\end{itemize}
We represent this situation in Figure \ref{inter}.
\tiny
\begin{figure}[ht]\begin{center}
\centerline{\includegraphics[height=6.5cm]{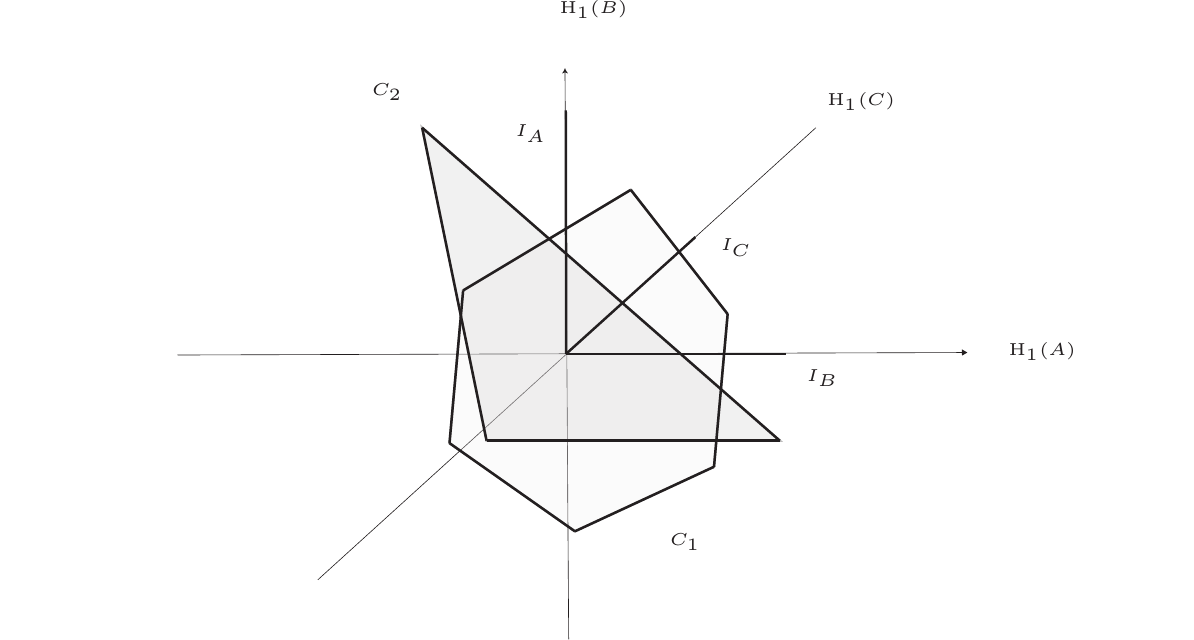}}
\caption{The rotation set of the map $h$.}
\label{inter2}
\end{center}\end{figure}
\normalsize

\subsection{The number of convex blocks of $\rho(f)$ may be exponential on $g$.} \label{s.ejemploexp}

We present a family of examples of $f\in\mathcal{A}_0(\Sigma_g)$ for $g>1$, satisfying that the number of basic convex blocks of $\rho(f)$ grows exponentially with $g$.

First we need to consider a specific Kwapisz map $h$ on an annulus $A$. Write $\partial A = \gamma^-\cup\gamma^+$  as union of its connected components.



A simple modification of the construction of \cite{Kw1} yields an axiom A diffeomorphism $h$ of $A$ with the following conditions:
\begin{enumerate}
\item $\partial A$ is fixed, $\gamma^+$ is an attractor and $\gamma^-$ a repellor.
\item There is a horseshoe-like basic piece $\Lambda$ so that $\rho_{\Lambda}(h)$ is an interval.
\item $W^s(\Lambda,h)$ accumulates on $\gamma^-$, and $W^u(\Lambda,h)$ accumulates on $\gamma^+$.
\item We may also require that all the non trivial rotation comes from $\Lambda$, and that $0\in \rho(h) = \rho_{\Lambda}(h)$. (This last point is not strictly necessary, but simplifies notations).


\end{enumerate}

Let $k=\left[ \frac{g}{2}\right]$, and consider the surface $\Sigma_g$ as shown in Figure
\ref{exponentialexample1} where for each $i=1,\ldots,k$ we define the essential curves
$\beta_i$, the non-separating annuli $A^0_i$, and $A^1_i$, and the separating annuli $A^*_i$. 


\tiny
\begin{figure}[ht]\begin{center}
\centerline{\includegraphics[height=5cm]{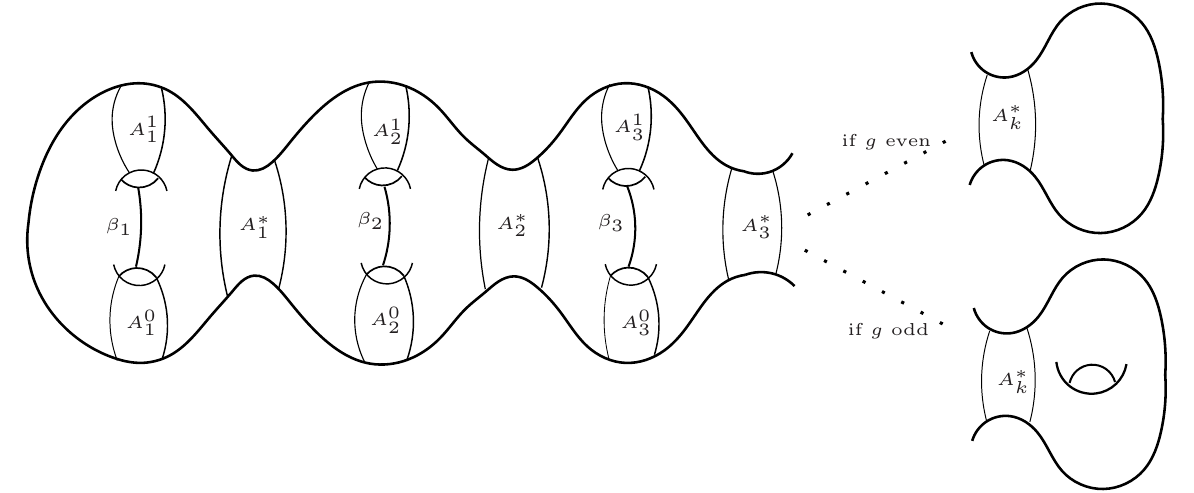}}
\caption{Decomposition of $\Sigma_g$ that will be used for our examples.}
\label{exponentialexample1}
\end{center}\end{figure}
\normalsize

Let us define $f_0\in\textrm{Homeo}_0(\Sigma)$ so that:

\begin{itemize}

\item $f_0$ is a copy of $h$ on each annuli $A^0_i$, $A^1_i$, and $A^*_i$ for $i=1,\ldots,k$, where the boundary components that appear on the left are identified with $\gamma^-$, and those on the right with $\gamma^+$.



\item In the complement of these annuli $f_0$ is given by a Morse-Smale
diffeomorphism $h_0$ satisfying the following properties:

\begin{itemize}
\item The non-wandering set of $h_0$ consists of finitely many irrotational fixed points.

\item For every $i=1,\ldots,k$, the curve $\beta_i$ is mapped towards its right side.

\item For every $i=1,\ldots,k$ and $j=0,1$ there is a point $x_{i,j}$ whose
$\alpha$-limit is contained in the right boundary component of $A_i^j$ and whose $\omega$-limit
is contained in the left boundary component of $A_i^*$.

\item For every $i=1,\ldots,k-1$ and $j=0,1$ there is a point $z_{i,j}$
whose $\alpha$-limit is contained in the right boundary component of $A_i^*$ and whose
$\omega$-limit is contained in the left boundary component of $A_{i+1}^j$.

\end{itemize}
\end{itemize}

Notice that the last two properties of $h_0$ imply that the boundary components of the annuli $A_i^j$, $A_i^*$ are not attractors nor repellors for $f_0$, so one can actually assume that they attract from their left side and repel from their right side.



Under these circumstances, one can modify the map $f_0$ by composing
with a time-one map $g_1$ of a flow whose support is contained in annular neighbourhoods of $\partial A^0_i$, $\partial A^1_i$, and $\partial A^*_i$ for $i=0,\ldots,k$, so that each of these boundary curves is mapped to its right. We assume that the support of $g_1$ is disjoint from the curves $\beta_i$, and intersects $\Omega(f_0)$ only on $\bigcup_{i=1}^k(\partial A^0_i\cup \partial A^1_i\cup \partial A^*_i) $. This way
we obtain $f=g_1\circ f_0 \in\mathcal{A}_0(\Sigma_g)$ represented in the next Figure,
and having the following features.

\tiny
\begin{figure}[ht]\begin{center}
\centerline{\includegraphics[height=5cm]{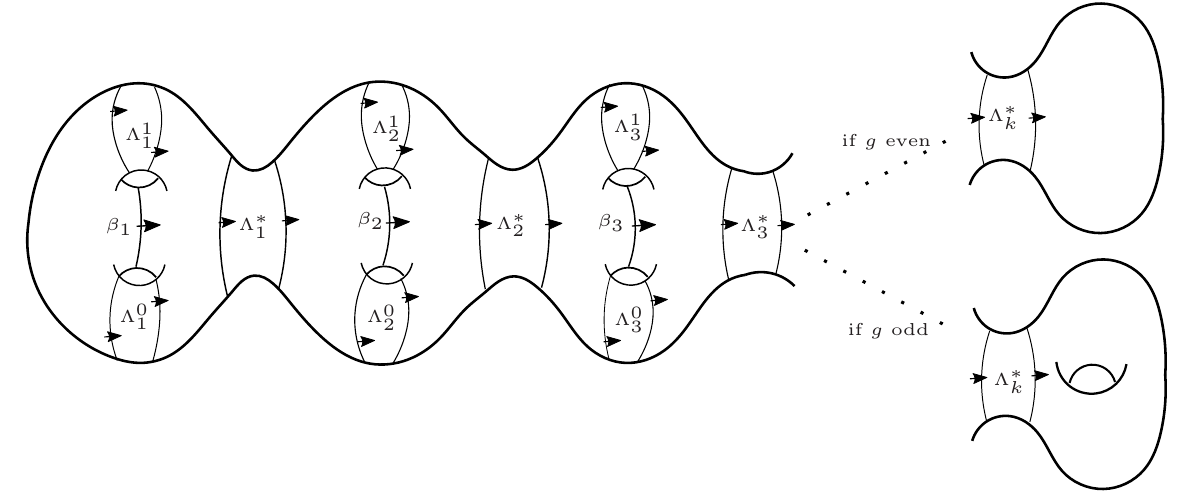}}
\caption{The map $f\in\mathcal{A}_0(\Sigma_g)$. The small arrows indicate
where $f$ maps the free curves given by $\partial A_i^j,\partial A_i^*$, and $\beta_i$.}
\label{exponentialexample}
\end{center}\end{figure}
\normalsize

\begin{enumerate}
\item $\Omega(f) = \Omega(f_0) \setminus \bigcup_{i=1}^k(\partial A^0_i\cup \partial A^1_i \cup \partial A^*_i)$.

\item $\rho_{\Lambda^j_i}(f)=\rho_{\Lambda^j_i }(f_0)$ is an interval containing $0$ in $H_1(A^j_i)\cong\R$, for $i=1,\ldots,k$, $j=0,1$.


\item $\Lambda_i^j\prec\Lambda_i^*$ for each $i=1,\ldots,k$, $j=0,1$.

\item $\Lambda_i^*\prec \Lambda_{i+1}^j$ for each $i=1,\ldots,k-1$, $j=0,1$.

\item For each $i=1,\ldots,k$, the basic pieces $\Lambda_{i}^0$ and $\Lambda_i^1$ are not
heteroclinically related. In order to ensure this point, one uses the fact that $f$
sends every curve $\beta_i$ to its right.


\end{enumerate}

The last three properties say that the graph $\Gamma$
given in Figure \ref{graph} is a full subgraph of $\mathcal{G}_f$, i.e. all the heteroclinical relations among the pieces $\Lambda^0_i$, $\Lambda^1_i$ and $\Lambda^*_i$, (for $i=1,\ldots,k$), arise from following the arrows of the directed graph $\Gamma$.

\tiny
\begin{figure}[ht]\begin{center}
\centerline{\includegraphics[height=3cm]{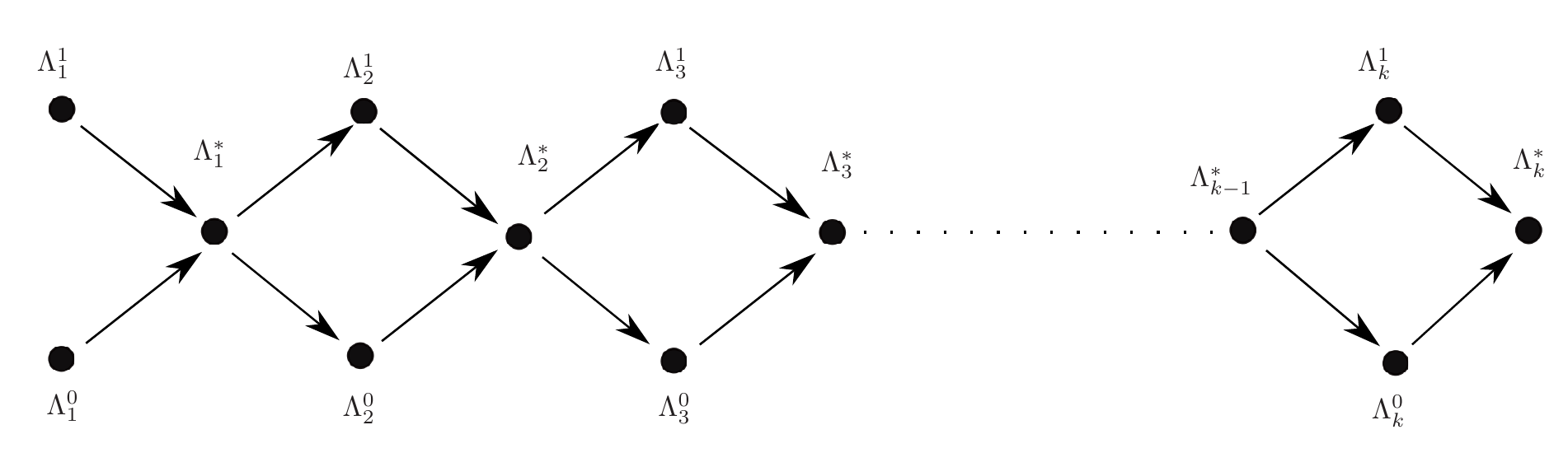}}
\caption{The graph $\Gamma$.}
\label{graph}
\end{center}\end{figure}
\normalsize

Since the basic pieces of $f$ that are not in $\Gamma$ have trivial rotation, and induce no other heteroclinical connections among the pieces in $\Gamma$, we see that $$\rho(f) = \bigcup_{\mathcal{C}}\rho_{\mathcal{C}}(f)$$ for $\mathcal{C}$ ranging over the maximal chains of $\Gamma$. We can index such chains by all the functions $$\xi:\{1,\ldots,k\}\to\{0,1\}, \quad \text{where}\quad \mathcal{C}(\xi) = (\Lambda_1^{\xi(1)}\prec\Lambda_1^{*}\prec\cdots\prec\Lambda_k^{\xi(k)}\prec\Lambda_k^{*} )$$
Let $C_{\xi} = \rho_{\mathcal{C}(\xi)}(f)$. For $i=1,\ldots,k$, notice that $H_1(A_i^*)=\{0\}$ and choose generators $[\alpha_i^j]$ of  $H_1(A_i^j) \cong\R$ for $j=0,1$. Notice from Figure \ref{exponentialexample1} that $\{[\alpha_i^j]:i=1,\ldots,k;j=0,1\}$ is a linearly independent set of $H_1(\Sigma_g;\R)$, so we get that

\begin{itemize}
\item by item (2) above, $C_{\xi}$ is a $k$-dimensional simplex in the subspace $V_{\xi}$ (of dimension $k$) spanned by $\{[\alpha_i^{\xi(i)}]:i=1,\ldots,k\}$,

\item and if $\xi\neq \xi'$, then  $V_{\xi}$ and $V_{\xi'}$ intersect in a subspace of dimension at most $k-1$.
\end{itemize}

Thus the sets $C_{\xi}$ for $\xi:\{1,\ldots,k\}\to\{0,1\}$ are the basic convex blocks of $\rho(f)$, and there are $2^k \sim 2^{g/2}$ of them, which shows the desired property for this family of examples.

\end{document}